\documentclass[11pt,leqno]{amsart}

\usepackage{amsmath}
\usepackage{amsthm,amssymb,amscd,amsfonts,amsbsy}
\usepackage{accents}
\usepackage{latexsym}
\usepackage{exscale}
\usepackage{fullpage} 
\usepackage{mathrsfs}
\usepackage{enumitem}
\usepackage[utf8]{inputenc}

\usepackage{braket}
\usepackage{faktor} 

\usepackage{graphicx}
\usepackage{bm}
\usepackage{bbm}
\usepackage{relsize}

\DeclareFontFamily{OT1}{pzc}{}
\DeclareFontShape{OT1}{pzc}{m}{it}{ <-> s*[1.2] pzcmi7t }{}
\DeclareMathAlphabet{\mathpzc}{OT1}{pzc}{m}{it}

\usepackage{hyperref} 
\usepackage{mathtools}
\mathtoolsset{showonlyrefs}

\NewCommandCopy\oldcite\cite
\RenewDocumentCommand\cite{O{}mE{*}{{}}}{%
\IfBlankTF{#1}{\IfBlankTF{#3}{\oldcite{#2}}{\oldcite[#3]{#2}}}{\oldcite[#1#3]{#2}}}
\newcommand\cites[1]{\cite{#1}}

\calclayout
\allowdisplaybreaks

\theoremstyle{plain}
\newtheorem{ABCtheorem}{Theorem}
    
\newtheorem{theorem}{Theorem}[section]
\newtheorem{lemma}[theorem]{Lemma}
\newtheorem{proposition}[theorem]{Proposition}

\newtheorem{corollary}[theorem]{Corollary}
\newtheorem{ABCcorollary}[ABCtheorem]{Corollary}

\newtheorem{innercustomthm}{Theorem} 
\newenvironment{customtheorem}[1]
  {\renewcommand\theinnercustomthm{#1}\innercustomthm}
  {\endinnercustomthm}

\theoremstyle{definition}
\newtheorem{definition}[theorem]{Definition}

\newtheorem{notation}[theorem]{Notation}

\theoremstyle{remark}
\newtheorem{remark}[theorem]{Remark}
\newtheorem*{remark*}{Remark}

\newtheoremstyle{noparens}%
{}{}%
{}{}%
{\bfseries}{.}%
{ }%
{\thmname{#1}\thmnumber{ #2}\thmnote{ #3}}
\theoremstyle{noparens}
\newtheorem*{claim}{Claim}

\newenvironment{claimproof}[1][]{%
\par
\noindent
$\ulcorner$ \hspace{0.1cm}\textit{Proof\ifx\newenvironment#1\newenvironment\else\space\fi#1.}
\ignorespaces}{\hfill$ {\scriptstyle\blacksquare}\,\,\,\,\lrcorner$\par\vspace{1.25ex}}

\DeclareMathOperator{\sech}{sech}

\newcommand{\Real}{\mathbb{R}}
\newcommand{\R}{\mathbb{R}}
\newcommand{\C}{\mathbb{C}}
\newcommand{\Natural}{\mathbb{N}}
\newcommand{\N}{\mathbb{N}}
\newcommand{\Integers}{\mathbb{Z}}
\newcommand{\Z}{\mathbb{Z}}
\newcommand{\Dscr}{\mathscr{D}}

\newcommand{\Mcal}{\mathcal{M}}
\newcommand{\Ncal}{\mathcal{N}}

\newcommand{\Gcal}{\mathcal{G}}
\newcommand{\Bcal}{\mathcal{B}}
\newcommand{\Scal}{\mathcal{S}}

\newcommand{\Rcal}{\mathcal{R}}

\newcommand{\V}{\mathcal{V}}
\newcommand{\U}{\mathcal{U}}

\newcommand{\Hcal}{\mathcal{H}}

\newcommand{\aha}{{\frac{1}{2}}}

\newcommand{\fr}[2]{\frac{#1}{#2}}
\newcommand{\supp}{\operatorname{supp}}

\newcommand{\eps}{\varepsilon}

\newcommand{\de}{\partial}
\newcommand{\RR}{\R^2}

\newcommand{\ph}{\varphi}
\newcommand{\wphi}{\phi}

\newcommand{\unif}{\text{{\rm unif}}}
\newcommand{\rc}{\text{{\rm rc}}}
\newcommand{\wt}{\widetilde}
\newcommand{\wh}{\widehat}
\newcommand{\jap}[1]{\braket{#1}}
\newcommand{\mb}[1]{\mathbf{#1}}
\newcommand{\mc}[1]{\mathcal{#1}}
\newcommand{\s}{\;\,}

\newcommand{\loc}{\mathrm{loc}}
\newcommand{\KP}{\mathrm{KP}}
\newcommand{\mKP}{\mathrm{mKP}}
\newcommand{\KdV}{\mathrm{KdV}}
\newcommand{\BMO}{\mathrm{BMO}}
\newcommand{\PP}[2]{#2_{\mathrm{p},#1}}
\newcommand{\PBMO}{\mathrm{BMO}_{\mathrm {p}}}
\newcommand{\PBMOx}[1]{\mathrm{BMO}_{\mathrm {p},#1}}
\newcommand{\PBMOl}{\PP\lambda\BMO}
\newcommand{\PBMOo}{\PP0\BMO}
\newcommand{\PHu}{\Hcal_{\mathrm{p}}^1}
\newcommand{\PHux}[1]{\Hcal_{\mathrm{p},#1}^1}
\newcommand{\p}{{\rm p}}
\newcommand{\OO}{\operatorname{O}}

\newcommand{\demu}{\partial_x^{-1}}
\newcommand{\Hcrit}{{\dot H^{-\aha,0}(\RR)}}
\newcommand{\wv}{{\wt v}}
\newcommand{\wV}{{\wt V}}
\newcommand{\wW}{{\wt W}}
\newcommand{\wu}{{\bar u}}

\newcommand{\GammaExp}{\accentset{\succ}{\Gamma}}

\newcommand{\M}{M}
\newcommand{\Mml}{\M_-^\lambda}
\newcommand{\Mpl}{\M_+^\lambda}

\newcommand{\Mult}{{\mathpzc M}}
\newcommand{\Fun}{\Phi}

\ExplSyntaxOn
\NewDocumentCommand{\sto}{}{\mathrel{\mathpalette \virgo_small:nn \to}}
\NewDocumentCommand{\sgets}{}{\mathrel{\mathpalette \virgo_small:nn \gets}}

\box_new:N \l_virgo_small_left_box
\box_new:N \l_virgo_small_right_box

\cs_new_protected:Nn \virgo_small:nn
 {
  \hbox_set:Nn \l_virgo_small_left_box { $#1#2$ }
  \box_set_trim:Nnnnn \l_virgo_small_left_box
   { 0pt }
   { -0.1\box_ht:N \l_virgo_small_left_box }
   { 0.6667\box_wd:N \l_virgo_small_left_box }
   { -0.4\box_ht:N \l_virgo_small_left_box }
  \box_set_clipped:N \l_virgo_small_left_box
  \hbox_set:Nn \l_virgo_small_right_box { $#1#2$ }
  \box_set_trim:Nnnnn \l_virgo_small_right_box
   { 0.6667\box_wd:N \l_virgo_small_right_box }
   { -0.1\box_ht:N \l_virgo_small_right_box }
   { 0pt }
   { -0.4\box_ht:N \l_virgo_small_right_box }
  \box_set_clipped:N \l_virgo_small_right_box
  \box_use:N \l_virgo_small_left_box 
  \box_use:N \l_virgo_small_right_box 
 }
\ExplSyntaxOff

\newcommand{\samB}{\mathpzc B}
\newcommand{\dtsmV}{\mathpzc V}
\newcommand{\eleV}{\mathbf V}
\newcommand{\eleVt}{{\mathbf V}_{\!\sto}}
\newcommand{\eleB}{\mathbf B}
\newcommand{\eleBt}{{\mathbf B}_{\sto}}
\newcommand{\dx}{\,dx}
\newcommand{\dxdy}{\,dx\,dy}

\newcommand{\tanhcnu}{\tanh\hspace{-0.3pt}\circ\hspace{2pt}\nu}

\newcommand{\sechscnu}{\sech^2\!\circ\hspace{2pt}\nu}
\newcommand{\etapmcnu}{\eta^\pm\!\!\circ\nu}
\newcommand{\etapcnu}{\eta^+\!\!\circ\nu}
\newcommand{\etamcnu}{\eta^-\!\!\circ\nu}

\newcommand{\vertiii}[1]{{\left\vert\kern-0.25ex\left\vert\kern-0.25ex\left\vert #1 
    \right\vert\kern-0.25ex\right\vert\kern-0.25ex\right\vert}}

\def\Xint#1{\mathchoice
    {\XXint\displaystyle\textstyle{#1}}%
    {\XXint\textstyle\scriptstyle{#1}}%
    {\XXint\scriptstyle\scriptscriptstyle{#1}}%
    {\XXint\scriptscriptstyle\scriptscriptstyle{#1}}%
    \!\int}
\def\XXint#1#2#3{{\setbox0=\hbox{$#1{#2#3}{\int}$}
    \vcenter{\hbox{$#2#3$}}\kern-.5\wd0}}
\def\fint{\Xint-}
    
\def\Xint#1{\mathchoice
    {\XXint\displaystyle\textstyle{#1}}%
    {\XXint\textstyle\scriptstyle{#1}}
    {\XXint\scriptstyle\scriptscriptstyle{#1}}%
    {\XXint\scriptscriptstyle\scriptscriptstyle{#1}}%
    \!\int}
\def\XXint#1#2#3{{\setbox0=\hbox{$#1{#2#3}{\int}$}
    \vcenter{\hbox{$#2#3$}}\kern-.5\wd0}}

\newcommand{\w}{z}
\newcommand{\g}{h}

\newcommand{\als}[1]{\begin{align*}#1\end{align*}}
\newcommand{\QQQ}{\qquad\qquad}
\newcommand{\Qq}{\qquad}

\numberwithin{equation}{section}

\begin{document}
\allowdisplaybreaks

\title[On the the B\"acklund transform and the stability of the line soliton of KP-II]{On the B\"acklund transform and the stability of the line soliton of the KP-II equation on $\R^2$}
\author{Lorenzo Pompili}

\begin{abstract}
    We study the Miura map of the KP-II equation on $\RR$ and the resulting B\"acklund transform, which adds a line soliton to a given solution. This work aims to develop a complementary approach to T. Mizumachi’s method for the $L^2$-stability of the line soliton, which the potential for generalization to multisolitons.
    
    We construct the B\"acklund transform by classifying solutions of the Miura map equation close to a modulated kink; this translates into studying eternal solutions of the forced viscous Burgers' equation under distinct boundary conditions at $\pm\infty$. We then show that its range, when intersected with a small ball in $|D_x|^\aha L^2(\RR)\cap L^2(\RR)\cap \braket{y}^{0-}\!L^1(\RR)$, forms a codimension-1 manifold.
    
    We prove codimension-1 $L^2$-stability of the line soliton in the aforementioned weighted space as a corollary, providing the first stability result at sharp regularity. The codimension-1 condition in the range of the B\"acklund transform is an intrinsic property, and we conjecture that it corresponds to a known long time behavior of perturbed line solitons. The stability is expected to hold without this condition, as in Mizumachi's works.
    
    Finally, we show the construction of a multisoliton addition map for $(k,1)$-multisolitons, $k\geq 1$.
\end{abstract}
\maketitle

\section{Introduction}
We consider the Kadomtsev--Petviashvili equation on the plane $\RR_{x,y}$
\begin{equation}\label{eq:KP-II}\tag{KP-II}
    u_t-6uu_x+u_{xxx}+3\partial_x^{-1}u_{yy}=0,
\end{equation}
a well-known two-dimensional generalization of the KdV equation
\begin{equation}\label{eq:kdv}\tag{KdV}
    u_t-6uu_x+u_{xxx}=0.
\end{equation}
The KdV equation can be seen as a special case of KP-II where solutions do not depend on the $y$ variable. A family of solutions of \eqref{eq:KP-II} is given by the KdV solitons
\[u(t,x,y)=\ph^{\lambda}(x-x_0-4\lambda^2t),\qquad \ph^{\lambda}(x):=-2\lambda^2\sech^2\left(\lambda x\right)=-\frac{c}{2}\sech^2\left(\frac{\sqrt c x}{2}\right),\]
for $\lambda>0$, $x_0\in \R$, and where the last equality holds\footnote{The parameter $c=4\lambda^2$ is the translational velocity of the soliton. The choice of parametrizing the family of solitons by the parameter $\lambda$ is natural from the inverse scattering point of view: $-\lambda^2$ is the ground state energy of the Schr\"odinger operator
$$ L^\KdV_u=-\partial_x^2+u $$
with potential $u=\ph^{\lambda}$, which is the Lax operator of a solution $u$ of the KdV equation. The ground state of $L_{\ph^{\lambda}}$ is $\sech(\lambda x)$.\\We remark that we intentionally chose the constants in equations \eqref{eq:kdv}, \eqref{eq:KP-II} so that the solitons are negative, despite physically representing water waves with positive elevation, so that the potential in the above Schr\"odinger operator coincides with the KdV solution itself, otherwise we would need a minus sign in front of $u$.} for $c=4\lambda^2$. In the context of the KP-II equation, the KdV soliton is called \emph{line soliton}, since it decays in the $x$ variable and is constant in the $y$ variable\footnote{Line solitons that are not parallel to the $y$-axis can be obtained by applying the KP-II Galilean symmetry \eqref{eq:KP-IIgalileanSymmetry} to the KdV solitons.}. The variable $y$ is often called the \emph{transversal direction}.\\

Historically, the KdV equation was one of the the first nonlinear dispersive models derived to describe travelling waves. The KdV soliton corresponds to the `wave of translation' discovered and studied by J. Scott Russell starting from 1834, often observable in shallow water or along narrow water channels. Russell's experiments, followed, among others, by the theoretical contributions of Boussinesq and Korteweg--de Vries, as well as the numerical experiments of Fermi--Pasta--Ulam--Tsingou and Kruskal--Zabusky, sparked significant interest among mathematicians and physicists, providing solid motivation for the development of soliton theory \cite{palais-1997-the_symmetries_of_solitons}.

Solitons of integrable dispersive PDEs are known to have a special behavior: they generally interact elastically with each other and with the `radiation', the part of the solution that decays in time. Their evolution is essentially decoupled from that of the rest of the solution, and they can be thought roughly as nonlinear eigenmodes of the equation. This is connected to the integrable PDEs being formally diagonalized by \emph{scattering transforms}, also called \emph{nonlinear Fourier transforms}. In physical terms, their shape being restored after interactions with other waves can be seen as a particle-like property, as suggested by the suffix `-on'.

Remarkably, this unique property of solitons can be made precise by the so-called \emph{soliton addition maps}, or \emph{B\"acklund transforms}\footnote{The term B\"acklund transform is an umbrella term that is used in the literature of integrable PDEs to describe various maps that conjugate the flows of two PDEs (often the same PDE). The name `soliton addition map' is more specific and refers to what is described in the paragraph. In this article, we will use the two terms synonymously.}, which commute with the flow of the respective integrable PDE and allow to nonlinearly add and subtract solitons from a given solution. These maps are naturally linked to the scattering transform, but they generally require different analytic techniques to be studied. They were derived and used to study the stability of solitons for several integrable PDEs in 1 space dimension \cite{mizumachi-pelinovsky-2012-nls-soliton-stability-backlund,Koch-tataru-2020-multisolitons-cubic-NLS-mKd,munoz-2023-sineGordon-kink,koch-yu-2023-asymptotic}.

What makes solitons of \eqref{eq:KP-II} on $\RR$ special, compared to those of other integrable PDEs, is their non-compact nature: the evolution of generic perturbations of the line soliton are described by scaling and translation parameters that depend on the $y$ variable. This introduces behaviors, such as those discussed in \cite{mizumachi2019phase} (see Subsection \ref{subsec:conjecture_on_Phi_and_h}), which, from a global perspective, suggest that the heuristic description provided in the third paragraph may be incomplete or require further refinement, although the same description works well locally in $y$. A B\"acklund transform for \eqref{eq:KP-II} is formally available and similar to that of \eqref{eq:kdv}, but the unboundedness of the soliton support poses some challenges for its well-definedness. The stability of the line soliton of \eqref{eq:KP-II} was proved by Mizumachi in \cite{mizumachi2015linesoli,mizumachi2018linesoli2,mizumachi2019phase}, with a proof that does not rely on the B\"acklund transform or on the integrability of the equation. The stability of KP-II multisolitons in $L^2(\RR)$, as well as their asymptotic stability, remains an open problem, except in the case of the line soliton.
\vspace{2mm}

The main objective of our work is to understand the soliton addition map of \eqref{eq:KP-II}. Our broader motivation is to look for a robust proof for the $L^2$ stability of general KP-II multisolitons, and possibly extend the same techniques to similar models. The goal of the present paper is to give a rigorous analytic treatment of the B\"acklund transform related to the line soliton. In particular, we are interested in understanding how the non-compactness of the line soliton plays a role in its properties, and how the transform gives information on the stability of the line soliton, compared to integrable models admitting localized solitons.

After constructing the soliton addition map and showing its properties, we find that the map is not surjective around the line soliton, but the range has codimension-1 in a weighted space of $L^2$ regularity. This allows to prove codimension-1 stability of the line soliton in $L^2(\RR)$ under such perturbations. The last result is comparable to some of the results of Mizumachi: in particular, it is weaker due to the codimension-1 condition, and stronger in the regularity assumptions.

We also construct the analogous B\"acklund transforms for a subclass of KP-II multisolitons, and plan to discuss analogous stability results for these multisolitons in a follow-up paper.

\subsection{Description of the problem and fundamental objects}\label{subsec:description}

\subsubsection{Properties of the KP-II equation}
The KP-II equation is invariant under space-time translations, although not under space rotations. It possesses several other symmetries: among others, a scaling symmetry
\begin{equation}\label{eq:scaling-symmetry-KP-II}
\Scal^\KP_\lambda u(t,x,y)=\lambda^2 u(\lambda^3 t,\lambda x,\lambda^2 y), 
\end{equation}
a reflection symmetry with respect to the $y$ variable
\begin{equation}\label{eq:reflection-symmetry-KP-II}
    \Rcal u(x,y,t)=u(x,-y,t),
\end{equation}
and two other important symmetries, which in the literature are both called \emph{Galilean symmetries}:
\begin{align}
    \Gcal^\KP_\sigma u(t,x,y)&:=u(t,x-\sigma y-3\sigma^2t,y+6\sigma t),\label{eq:KP-IIgalileanSymmetry}\\
    \Gcal^\KdV_\mu u(t,x,y)&:=\mathcal B_\mu u(t,x,y)-\mu:=u(t,x-6\mu t,y)-\mu.\label{eq:KdVgalileanSymmetry}
\end{align}
The second one is inherited from the KdV equation, while the first one does not apply to KdV, so we refer to them as the KP-II and the KdV Galilean symmetries respectively (here $\mathcal B$ stands for `boost'). The KP-II Galilean symmetry plays well with the well-posedness theory of the equation as it preserves the Sobolev spaces used in the standard theory, whereas the KdV symmetry adds a constant to the solution. In rigorous terms, since the antiderivative in the KP-II equation is not a priori uniquely defined for solutions with arbitrary growth at infinity, both Galilean symmetries need to be accompanied by auxiliary changes of variables for the function $v$ such that $v_x=u_{yy}$ appearing in the KP-II equation.

A natural function space whose norm is invariant under the symmetries $\Scal^\KP$ and $\Gcal^\KP$ is the Banach space $\Hcrit$, with norm defined by
\[ \|u\|_{\Hcrit}:=\int_{\RR} |2\pi\xi|^{-1}|\hat u(\xi,\eta)|^2d\xi\,d\eta. \]
It is known that the KP-II equation is globally well-posed in $H^k(\RR)$, $k\in\N$ for large data, as well as in $\Hcrit$ for small data. The globality of solutions in the first result follows from the conservation of the $L^2$ norm of the equation, while the latter comes from global-in-time bilinear estimates for solutions of the linear equation and the scaling invariance of the space $\Hcrit$. The line soliton $\ph$ does not lie in any of the aforementioned Banach spaces, but global well-posedness was proved in $H^k(\RR)+\ph$, $k\in\N$. These results are discussed in more detail in Section \ref{sec:time}.

\begin{remark}
    The antiderivative appearing in \eqref{eq:KP-II} can be understood in the following way. The operator $S=\partial_x^3+3\demu\partial_y^2$ is a skew-adjoint unbounded operator on $L^2(\RR)$ with an explicit domain given in terms of the Fourier transform. In particular, the unitary group $t\mapsto e^{-t(\partial_x^3+\demu\partial_y^2)}$ is a well-defined continuous group on $L^2(\RR)$ and other $L^2$-based Sobolev spaces, and Strichartz estimates for this group can be proved in full analogy with the linear Schr\"odinger equation, so that the Duhamel formulation of \eqref{eq:KP-II} makes perfect sense for general classes of functions. More refined bilinear estimates are required for proving well-posedness of the nonlinear equation (see \cite{koch-tataru-visan-2014-dispersive-equations-nonlinear-waves} for a good introduction). In addition to that, the operator $\demu\partial_y$ extends to a well-defined operator on a suitable Banach space containing the space of solutions from the well-posedness theory, even at low regularity (see Remark \ref{rk:Z-aha-is-in-L2-loc} and Theorem \ref{theorem:global-well-posedness-KP-II-HHK09}).
\end{remark}

\subsubsection{Modulational stability of the line soliton}

The line soliton $\ph$ is not orbitally stable in $L^2(\RR)$ \cite{mizumachi2015linesoli}. Since it is infinitely long on $\RR$, in general small perturbations of $\ph$ evolve so that the perturbed modulation parameters $\lambda$ and $x_0$ depend sensibly on the transversal variable. For this reason, what is natural to investigate is the so-called \emph{modulational stability} of $\ph$, in which we ask for the solution to remain close to a line soliton whose parameters (scaling and translation) are allowed to depend on the $y$ variable. We give the following definition.

\begin{definition}[Modulational stability]\label{def:modulational_orbital_stability}
We say that the line soliton $\ph$ is \emph{modulationally stable} in $L^2(\RR)$ under perturbations in a set $\mc N\subset L^2(\RR)$ provided that the following holds: for every initial datum $u_0\in \mc N+\ph$, there exist two continuous functions $x=x(t,y)$, $\lambda(t,y)$ called \emph{modulation parameters} such that the solution $u=u(t,x,y)$ of \eqref{eq:KP-II} with initial datum $u_0$ satisfies
\[ \sup_{t>0}\|u(t,x,y)-\ph^{\lambda(t,y)}(x-x(t,y))\|_{L^2(\RR)} \lesssim_{\hspace{1pt}\mc N} \|u_0-\ph\|_{L^2(\RR)}, \]
where the implicit constant only depends on the set $\mc N$.
\end{definition}
\noindent A desirable result is obtained when $\mc N$ contains a small ball centered at the origin in a Banach space $X\subset L^2(\RR)$ containing the space of test functions $\mathscr D(\RR)$, or at least a finite-codimensional submanifold of $X$ containing the origin. The definition above is adapted to our specific problem, but notice that it can be generalized in several ways allowing for different spaces and moduli of continuity, or considering asymptotic stability instead.

Due to the nature of the result we want to show, we will make extensive use of the following notation throughout the article.
\begin{notation}[Subscript notation]\label{notation:f_alpha}
For a function $f=f(x)$ and for $\alpha\in\R$, we will denote by $f_\alpha$ the function
\[
    f_\alpha(x):=f(x-\alpha).
\]
More generally, $f$ and $\alpha$ will be allowed to depend on $y$, or on $t,y$: in that case, $f_\alpha$ will denote
    \begin{equation}\label{eq:f_alpha}
        f_\alpha(t,x,y):=f(t,x-\alpha(t,y),y).
    \end{equation}
The only exception to this notation rule is the function $G_\alpha$ in Definition \ref{def:VandG}.
\end{notation}

\subsubsection{The Miura map and the mKP-II equation}\label{subsubsec:backlund_transformation_introduction}
The KP-II equation is related to the mKP-II equation
\begin{equation}\label{eq:mKP-II}\tag{mKP-II}
    v_t-6v^2v_x+v_{xxx}+3\de_x^{-1}v_{yy}+6v_x\de_x^{-1}v_y=0
\end{equation}
via the Miura map
\[M_\pm^\lambda(v)=-\left(\de_x^{-1}v_y\pm v_x-v^2+\lambda^2\right).\]
Formally, if $v(t,x,y)$ is a solution of \eqref{eq:mKP-II}, then $M_\pm^\lambda(v)(t,x-6\lambda^2t,y)$, i.e., $\mathcal B_{\lambda^2} M_\pm^\lambda(v)$, are solutions of \eqref{eq:KP-II}. This fact is rigorously true as long as all the terms $\demu\partial_y v$ apparing in the Miura map and in \eqref{eq:mKP-II} and are a distributional $x$-antidrivative of $v_y$ (see Proposition \ref{prop:miura-maps-mKP-to-KP}). The parameter $\lambda$ comes from the symmetries of the equation: it is related to the KdV Galilean symmetry, as it holds
\[
\Bcal_{\lambda^2} M^\lambda_\pm=\Gcal^\KdV_{\lambda^2} M^0_\pm,
\]
but we can also see it as coming from the scaling symmetry:
\[
M^\lambda_\pm\Scal^\mKP_{\lambda} =\Scal^\KP_{\lambda} M^1_\pm,
\]
where $\Scal^\mKP_{\lambda}:=\lambda^{-1}\Scal^\KP_{\lambda}$ is the scaling symmetry of \eqref{eq:mKP-II}. The mKP-II equation admits analogous Galilean symmetries, which are discussed in Section \ref{subsec:manySystems}.

There exist kink solutions of \eqref{eq:mKP-II} of the form
\begin{equation}\label{eq:kink_Q}
v(t,x,y)=Q^\lambda(x-x_0+2\lambda^2t)=Q^\lambda\left(x-x_0+\frac{c}{2}t\right),\qquad Q^\lambda(x):=\lambda\tanh(\lambda x),
\end{equation}
and the following relations hold:
\begin{equation}\label{eq:miura-maps-kink-to-soliton}
    \Mml(Q^\lambda)=0,\qquad \Mpl(Q^\lambda)=\ph^\lambda.
\end{equation}
The above identities indicate that the soliton is a special solution of \eqref{eq:KP-II} that is connected to the trivial solution through the Miura map. In the rest of this work, we will look at the case $\lambda=1$, without loss of generality thanks to the scaling symmetries $\mathcal S^\KP_\lambda$, $\mathcal S^\mKP_\lambda$, and set $M_\pm:=M_\pm^1$, $Q:=Q^1$, $\ph:=\ph^1$. 

\subsubsection{The B\"acklund transform}

Using the Miura maps, we can define a relation between two different functions $u$, $\wu$ if there exists a third function $v$ such that it holds
\begin{equation}
    \left\{
    \begin{aligned}
&M_+(v)=\wu,\\
&M_-(v)=u,
    \end{aligned}
    \right.
\end{equation}
or more rigorously, with some redundancy to allow a symmetric writing of the system,
\begin{equation}
    \left\{
    \begin{aligned}
&v_y+v_{xx}=(v^2)_x-\wu_x,\\
&v_y-v_{xx}=(v^2)_x-u_x,\\
&u-\wu=2v_x.
    \end{aligned}
    \right.
\end{equation}
It is desirable to look for a map
\[
\samB:(u,\gamma_0)\mapsto \wu
\]
that satisfies the above system, where $\gamma_0$ is an auxiliary parameter that allows to select one of the many pairs $(v,\wu)$ that would satisfy the system for fixed $u$. Clearly, $v=Q(\cdot-\gamma_0)$, $u=0$, $\wu=\ph(\cdot-\gamma_0)$ is a solution of the above system for all $\gamma_0\in\R$, so we expect $\R$ to be a natural parameter space for the second entry of $\samB$. By the commuting properties of the Miura map, we expect that if $u$ now depends on time and is a solution of \eqref{eq:KP-II}, the function $\wu$ is a new solution of \eqref{eq:KP-II}, at least for a suitably chosen time-dependent $\gamma_0$. Moreover, the effect of the transformation is that of nonlinearly adding a line soliton to the solution $u$ so that the new function is still a solution of \eqref{eq:KP-II}. When well-defined, the map $\samB$ is called \emph{soliton addition map}, or \emph{B\"acklund transform}. It is evident how such a map can give plenty of information on the dynamics of solutions close to the soliton. In particular, the map $\samB$ essentially conjugates the \eqref{eq:KP-II} flow around the zero solution with the same flow around the soliton. The stability of the line soliton can morally be translated to the problem of the stability of the zero solution of \eqref{eq:KP-II} if $\samB$ satisfies suitable continuity properties.

A large part of this article is dedicated to the construction of the soliton addition map $\samB$ and studying its properties, which will then be used to study the stability of the line soliton. In particular, we study the inversion of the Miura map $M_-$ around the pair $Q\mapsto 0$. This amounts to solving the \emph{viscous Burgers equation with forcing}
\begin{equation}\label{eq:miuraEquation}\tag{$\mathcal M$}
    v_y-v_{xx}=(v^2)_x-u_x
\end{equation}
for all $x,y\in\RR$, where $u$ is a given profile, and $v$ is a suitable perturbation of $Q$ to be determined.

\subsection{Context and previous work}
\subsubsection{Completely integrable dispersive equations}
The Miura map written above is part of the rich set of algebraic properties that form the integrable structure of the KdV and KP-II equations. Most of the known \emph{completely integrable} dispersive PDEs can formally be rewritten using a \emph{Lax pair}, a set of two differential operators $L(t)=L[u(t)], P(t)=P[u(t)]$ which depend only on the solution $u$ of the original equation at time $t$, that transform the PDE into the \emph{Lax equation}
\begin{equation}
    \frac{d}{dt}L(t)=[P(t),L(t)],
\end{equation}
where $[\cdot,\cdot]$ is the commutator between operators. The operator $L(t)$ is called the \emph{Lax operator} of the solution $u(t)$. Some of the most notable instances of dispersive integrable PDEs are the 1D cubic NLS, KdV, mKdV and Benjamin--Ono equations; higher-dimensional integrable equations include the KP-I and KP-II equations, the Davey--Stewartson, and the Veselov--Novikov equations (the latter admits a modified version of the Lax equation). The theory of integrable equations begun with the article by Gardner--Greene--Kruskal--Miura on the KdV equation in 1967 \cite{gardner-greene-kruskal-miura-1967-KdV} featuring the scattering transform of KdV, which as already mentioned diagonalizes the equation reducing it to a family of linear ODEs with constant coefficients, and was developed in the seventies and eighties with the discovery of scattering transforms for several integrable PDEs. The method used in \cite{gardner-greene-kruskal-miura-1967-KdV} to invert the scattering transform and recover the solution of the original PDE was extended to other models, and is now commonly known as \emph{inverse scattering transform (IST)} \cite{palais-1997-the_symmetries_of_solitons}.

Concerning the inverse scattering theory of KP-II on $\RR$, results on solutions that are perturbations of the line soliton are available. We mention the IST theory developed by Villarroel--Ablowitz \cite{Villarroel-Ablowitz-2004-On-IVP-of-KP-II-with-data-that-do-not-decay-along-a-line} after previous works by Boiti--Pempinelli--Pogrebkov--Prinari, and the subsequent extensive works by Wu on the KP-II line soliton and multisolitons \cite{Derchyi-Wu-2020-direct-scattering-Gr12-KP-II-soliton,wu2022stability}.

\subsubsection{Recent developments}
The main limitation of the use of the inverse scattering machinery, especially for PDEs on the Euclidean space, is that the inverse scattering transform is well-behaved only when the solution decays fast enough in space (for instance, $u_0\in L^1(\R;\jap{x}dx)$ for the KdV equation, see the survey \cite{aktosun-2004-IST-KdV-and-solitons} for an overview of the basic theory) and regular enough. These conditions are often strictly stronger than the ones actually needed for the well-posedness of the equation: typically, the initial datum is allowed to be in $H^s(\R^n)$ for a suitable $s\in\R$.

In more recent years, there was an increasing effort in employing the integrable structure in the study of dispersive PDEs in $L^2$-based Sobolev spaces and in spaces of critical regularity. The main and groundbreaking application of these techniques is low regularity well-posedness, which was established much earlier on the circle for some models, see for instance \cites{kappeler-topalov-2006-sharp-wellposedness-KdV-on-the-torus,kappeler-topalov-2005-sharp-wellposedness-mKdV-on-the-torus,Gerard-kappeler-topalov-2020-sharp-wellposedness-BO-on-the-torus}, and finally on the real line, with the first work in this direction being the celebrated article on the KdV equation in $H^{-1}(\R)$ by Killip--Visan \cite{killip-visan-2019-kdv-is-wellposed-in-H-1}, with several subsequent works on the cubic 1D NLS equation, modified KdV and many more, see for instance \cites{harrophGriffith-killip-visan-2020-sharp-wellposedness-cubicNLS-and-mKdV,bringmann-killip-visan-2021-global-well-posedness-5th-order-kdv,thierry-laurens-2023-sharp-gwp-of-KdV-with-exotic-spatial-asymptotics,Killip-Laurens-Visan-2024-sharp-well-posedness-BO}. Another important application which came slightly before the former, is the construction of generalized conservation laws at $H^s$-regularity which prove almost-conservation of the Sobolev norms: one of the earliest works on the real line is by Koch and Tataru \cite{Koch-tataru-2018-conserved-energies-cubic-NLS-mKdV-KdV} for the NLS, KdV and mKdV equations, with subsequent works on Gross--Pitaevskii \cites{koch-liao-2021-conserved-energies-1D-gross-pitaevskii,koch-liao-2023-conserved-energies-1D-gross-pitaevskii} and Benjamin--Ono \cite{talbut-2021-conservation-laws-BO} by other authors.

Another important instance of the use of the structure of integrable PDEs is the stability of solitons and multisolitons. In the eighties and nineties, there was a considerable effort in the study of solitons of integrable equations: important objects in this sense are the \emph{B\"acklund transforms} for integrable PDEs, which allow the explicit construction of multisolitons. Although these were mostly treated as algebraic manipulation of the equations, with the more recent well-posedness results on the equations, it became natural to ask whether these maps can be used to study the dynamics of solutions that are close to multisolitons. This was initially done for the KdV equation using the Miura map (see the next paragraph), and later for some other models such as cubic NLS \cite{mizumachi-pelinovsky-2012-nls-soliton-stability-backlund} and sine-Gordon \cite{munoz-2023-sineGordon-kink,koch-yu-2023-asymptotic}. The first work that used B\"acklund transforms for the study of solitons in a somewhat abstract and more general sense, and which serves as a significant inspiration for our paper, is that of Koch--Tataru \cite{Koch-tataru-2020-multisolitons-cubic-NLS-mKd}. The remarkable property of B\"acklund transforms is that, despite coming from the integrable structure and being linked to the scattering transform, they are generally robust enough that they do not need to rely on the whole analytic framework needed for the treatment of the inverse scattering transform.

\subsubsection{KdV solitons}
In the above discussion, the KdV equation was historically an important link in the chain that led to the current understanding and state of the art. The KdV Miura map,
\[ v\,\,\mapsto\,\, M^\lambda_{\KdV,\pm}(v):=\mp v_x+v^2-\lambda^2, \]
which maps solutions of the (defocusing) mKdV equation
\[ v_t-6v^2v_x+v_{xxx}=0 \]
to solutions of \eqref{eq:kdv} (up to an inertial change of frame of reference), allows in its simplicity to understand the power of B\"acklund transforms in soliton dynamics. The resulting soliton addition map, analogous to the one described earlier for KP-II, makes possible to establish a diffeomorphism\footnote{More precisely, the map $M_{\KdV,-}^\lambda$ has one-dimensional fibers, while the map $(\lambda,v)\mapsto M^\lambda_{\KdV,+}(v)$ is a diffeomorphism locally: this gives rise to a soliton removal map which is locally a submersion with two-dimensional kernel. To make it a full diffeomorphism, one needs to take into account two additional degrees of freedom, which naturally correspond to the choice of position and scaling parameter of the soliton. See \cite{buckKochKDVatH-1regularity} for a detailed treatment.} between a neighbourhood of the zero solution and a neighbourhood of the soliton $\ph$, so that one can reduce the stability of the soliton to the stability of the zero solution, which is a consequence of the conservation laws of the equation. This idea was used to prove the stability of KdV solitons back in 2003 by Merle--Vega \cite{merle-vega-2003-stability-of-solitons-for-KdV} for $L^2$ data, and later by Buckmaster--Koch \cite{buckKochKDVatH-1regularity} in the well-posedness critical topology $H^{-1}(\R)$ (we remind that the scaling-critical Sobolev space for KdV is $\dot H^{-\fr32}(\R)$). The picture was completed later by Killip--Visan \cite{killip-visan-2022-kdv-multisolitons}, who proved stability of multisolitons at sharp regularity.

\subsubsection{The KP-II equation and the line soliton}

The nonlinear $L^2$ stability of the KP-II line soliton was proved in \cite{mizumachi-tzvetkov-2012-stability-line-soliton-periodic-transverse-perturbations} on the cylinder $\R_x\times\mathbb T_y$ following the same idea of \cite{merle-vega-2003-stability-of-solitons-for-KdV} for the KdV case. In that setting, the soliton addition map coming from the KP-II Miura map still gives a local diffeomorphism connecting $0$- and $1$-solitons under those boundary conditions

The stability question on $\RR$ turns out to be much more delicate and challenging. The linearized evolution of KP-II around the line soliton admits resonant continuous eigenmodes in an exponentially weighted space, with eigenvalues accumulating at zero \cite{mizumachi2015linesoli}. These eigenmodes represent the modulations that a perturbed, infinitely long line soliton undergoes during the KP-II evolution. The presence of these eigenmodes suggests that the B\"acklund transform behaves differently from the KdV case, since the linearized KP-II equation around $u\equiv 0$ possesses a different spectrum. In particular, the argument relying on the Miura map used for KdV does not readily generalize to KP-II on $\RR$ and was not investigated before, unlike in the cylindrical geometry mentioned above where solitons essentially have finite length.

Eventually, Mizumachi solved the stability problem on $\RR$ in a striking series of papers \cites{mizumachi2015linesoli,mizumachi2018linesoli2,mizumachi2019phase} in which he proved \emph{modulational $L^2$-stability} of the line soliton as well as \emph{asymptotic stability} in suitable subsets of $\RR$, under polynomially decaying perturbations and under perturbations in $\de_x L^2(\RR)$, with additional regularity and smallness assumptions. The papers provide a precise description of the position and amplitude modulation parameters $x=x(t,y)$ and $\lambda=\lambda(t,y)$, which evolve under a 1D wave equation with damping. Remarkably, Mizumachi's arguments do not rely on the integrability of the KP-II equation, and were in fact used in later works on non-integrable PDEs \cite{mizumachi_shimabukuro-2020-benney-luke}. On the other hand, the role of the Miura map in the stability of the line soliton remained unclear.

\subsubsection{KP-II multisolitons}

As for many other integrable dispersive PDEs, the \eqref{eq:KP-II} equation admits a set of explicit, analytic solutions called \emph{multisolitons}, which represent the interaction of several (half-)line solitons. These naturally include the KdV multisolitons as a subclass, which look like a superposition of several parallel line solitons. Due to the KP-II Galilean symmetry \eqref{eq:KP-IIgalileanSymmetry}, single line solitons can be `tilted' with respect to the $y$ axis, and in fact more general multisolitons exist which combine several line solitons with different slopes in the $x,y$-plane, creating complicated and fascinating patterns that impressively resemble real water wave interactions, as noted in \cite{ablowitz-baldwin-2012-pictures-of-multisolitons}. See Subsection \ref{subsec:multisolitons} for more details, and we refer to \cites{ChakravartyKodama2007ClassificationOTlineSoliton,kodama2017solitonsAndGrassmannians} for an extensive treatment of KP-II multisolitons.

The KP-II multisolitons are expected to be stable up to modulations as in the case of the line soliton, although very little was proved on the subject. The linear stability was rigorously proved only very recently by Mizumachi for the case of the so-called elastic 2-line solitons \cite{mizumachi-2024-linear-stability-elastic-2-line-soliton}.
On the nonlinear dynamics, we have the recent preprint of D. Wu \cite{wu2022stability} which addresses the inverse scattering problem of KP-II around multisolitons and states an $L^\infty$ bound uniform in time for all perturbations in an $L^1-$based Sobolev space, where the $L^\infty$ estimate is given in terms of the full $L^1$ Sobolev norm of the initial perturbation. To the best of our knowledge, there are no other results at present concerning the long term nonlinear dynamics of KP-II multisolitons, except the line soliton. In particular, modulational stability in $L^2(\RR)$ and asymptotic stability are open.

\subsection{Strategy, main results, and structure of the paper}

As previously discussed, the aim of this article is to study the soliton addition map connecting the zero solution and the line soliton, and try to make use of it in the study of the nonlinear stability of the line soliton.

Let us make use of Notation \ref{notation:f_alpha}. Due to the relations \eqref{eq:miura-maps-kink-to-soliton}, one would hope that for a small, uniformly bounded in time solution $u$ of KP-II there exists an associated solution $v$ of mKP-II such that $M_-(v)=u$ which is close in some sense to a modulated kink $Q_\alpha$, $\alpha=\alpha(y)$, so that $\wu:= M_+(v)$ is a solution of KP-II close to a modulated line soliton $\ph_\alpha$. To make sense of the soliton addition map, one needs to solve the equation $M_-(v)=u$, which we rewrote as
\begin{equation*}\tag*{(\ref*{eq:miuraEquation})}
    v_y-v_{xx}=(v^2)_x-u_x,
\end{equation*}
for known $u$ and for solutions $v$ close to a modulated kink. As previously discussed, we expect to find a one-parameter family of solutions for given $u$, since by translation invariance $M_-(Q(\cdot-\gamma))=0$ for any $\gamma\in\R$. These simple observations suggest the following program:
\begin{itemize}
    \item classify the solutions $v$ to equation \eqref{eq:miuraEquation} close to a modulated kink and provide a suitable parametrization with respect to a real parameter $\gamma_0\in\R$: $v=\dtsmV(u,\gamma_0)$,
    \item define the soliton addition map as discussed in \ref{subsubsec:backlund_transformation_introduction},
    \[
    \samB(u,\gamma_0):=M_+(\dtsmV(u,\gamma_0))
    \]
    \item prove two-sided bounds for $\samB$ of the form $\|u\|_X\sim \|\samB(u,\gamma_0)-\ph_\alpha\|_X$ for suitable $\alpha$, for some Banach space $X$
    \item prove that $\samB$ commutes with the KP-II flow modulo the correct choice of $\gamma_0=\gamma_0(t)$
    \item characterize the range of $\samB$, or find sufficient conditions for a perturbation of the line soliton to fall in the range of the map
\end{itemize}
The above strategy would be able to reduce the modulational stability of the line soliton to the stability of the zero solution in the space $X$. It is natural to choose $X=L^2(\RR)$ due to the conservation of the $L^2$ norm along the KP-II flow. We want to follow this strategy in analogy with previous works on KdV and KP-II on the cylinder, but we have to solve the problems introduced by the infinite length of the solitons.

To construct the maps $\dtsmV$ and $\samB$, we choose to work in the critical space $\Hcrit$ to make use of the small-data global existence result in \cite{hadacHerrKoch2009wellPosednessKP-IIinCriticalSpace}, see Theorem \eqref{theorem:global-well-posedness-KP-II-HHK09}. The outcome of our research can be summarized as follows.
\begin{enumerate}
    \item For small $u\in\Hcrit$, equation \eqref{eq:miuraEquation} admits a 1-parameter family of solutions of the form $v=w+\tanh_\alpha$, $w\in L^3(\RR)$, $\alpha_y\in L^2(\R_y)$ parametrized by $\gamma_0\in\R$, which roughly coincides with $\alpha|_{y=0}$. \label{item:z-001}
    \item The soliton addition map is well-defined for generic small data in $\Hcrit$:\label{item:z-002}
    \[
    \samB\colon B^{\dot H^{-\aha,0}}_{\eps_0}(0)\times \R \to H^{-\aha,0}(\RR)+\{\ph_\alpha \,|\,\alpha_y\in L^2(\R_y)\},
    \]
    for a universal constant $\eps_0>0$. The second parameter $\gamma_0\approx\alpha(0)$ selects one solution $v$ from the 1-parameter family above, and sets $\samB(u,\gamma_0)=u-2v_x$. This has the effect of superimposing a modulated line soliton on $u$.
    \item The map $\samB$ satisfies a two-sided $L^2$ estimate. If $\wu=\samB(u,\gamma_0)$, then $\|u\|_{L^2(\RR)}$ is comparable with the $L^2$-distance between $\wu$ and the manifold of modulated line solitons \label{item:z-003}
    \[
    \{\ph_\alpha\,|\,\alpha\in C(\R_y),\,\|\alpha_y\|_{L^2(\R_y)}<\infty\},
    \]
    with a penalization factor that grows with $\|\alpha_y\|_{L^2(\R_y)}$.
    \item The map $\samB$ commutes with the KP-II flow. More precisely, assuming that $u=u(t)$ is a solution of KP-II in $\Hcrit$ and $\gamma_{0,0}\in\R$, $\wu(t)=\samB(u(t),\gamma_0(t))$ is a solution of KP-II for a suitable $t\mapsto\gamma_0(t)$ with $\gamma_0(0)=\gamma_{0,0}$. If $u(0)$ is also in $L^2$, then $\wu$ coincides with the unique solution of KP-II given by the well-posedness theory. \label{item:z-004}
    \item The range of $\samB$ contains $\ph+\Ncal$, where \label{item:z-005}
    \[
    \Ncal\subset\Hcrit\cap L^2(\RR)\cap L^1(\RR)\cap L^1(\RR;\sech^2(x)(1+|y|)^\eps dx\,dy)=:Y_\eps(\RR)
    \]
    is an analytic codimension-1 submanifold of the Banach space $Y_\eps(\RR)$ for any $\eps>0$. Functions outside this manifold are not in the range of $\samB$.
    \item As a consequence of (\ref*{item:z-003}), (\ref*{item:z-004}), (\ref*{item:z-005}), the line soliton is modulationally stable in $L^2$ under perturbations in $\Ncal$.
\end{enumerate}
The strategy involved will allow to make a natural generalization with little additional effort.
\begin{enumerate}
    \item[(7)] We construct a multisoliton addition map for $(k,1)$-multisolitons (see the definition in Section \ref{sec:preliminaries}) for any $k\geq 1$, which equals $\samB$ for $k=1$ up to a change of variables.
\end{enumerate}

Our result for the first point is the following theorem, which classifies the solutions of equation \eqref{eq:miuraEquation} close to a modulated kink. The preimage of small data $u\in\Hcrit$ through the Miura map $M_-$ under suitable conditions is a one-parameter family of solutions to \eqref{eq:miuraEquation} which can be parametrized by the position of the kink at $y=0$.

We use the notation $f_\alpha(x,y):=f(x-\alpha(y),y)$ as in Notation \ref{notation:f_alpha}, we let $\rho$ be a standard mollifier in $\RR$, and $\eta^\pm(x):=(1+e^{\mp2x})^{-1}$.

\begin{ABCtheorem}[Classification of solutions of \eqref{eq:miuraEquation} close to a kink]\label{theorem:theorem_1}
    Let $u\in \Hcrit$ be small enough, and $\gamma_0\in\R$. There exists a unique solution $v$ to equation \eqref{eq:miuraEquation} of the form $v=w+\tanh_\sigma$, $w\in L^3(\RR)$, $\sigma\in C(\R_y)$, $\sigma_y\in L^2(\R_y)$, satisfying the localization condition
    \[
    \int_{\RR} \rho(x-\gamma_0,y) v(x,y) \,dx\,dy=0.
    \]
    The solution can be decomposed into
    \[ v=\tanh_\alpha+\eta^+_\alpha\cdot (v^+-1) +\eta^-_\alpha\cdot (v^-+1) +\omega, \]
    where $v^\pm$ are the unique solutions of \eqref{eq:miuraEquation} in $L^3(\RR)\pm 1$, $\alpha\in C(\R_y)$ is determined uniquely by the orthogonality condition
    \[
    \int_{\R} \omega(x,y)\,dx=0\;\,\forall y\in \R,
    \]
    and it holds the estimate
    $$ \|v^\pm\mp 1\|_{L^3(\RR)}+\|\cosh_\alpha \omega\|_{C_0 L^2\cap L^2H^1}+\|\alpha_y\|_{L^2}\lesssim\|u\|_{\dot H^{-\aha,0}(\RR)}. $$
    Moreover, $v^\pm$ satisfy all the bounds of Corollary \ref{cor:estimates-for-wtv-part-2} (with $\lambda=\pm 1$), and the shift $\alpha$ satisfies
    \[
    \sup_{y_1,y_2\in\R}\frac{|\alpha(y_2)-\alpha(y_1)|}{\log(2+|y_2-y_1|)}\lesssim \|u\|_\Hcrit.
    \]
    The map $\dtsmV:(u,\gamma_0)\mapsto v$ is continuous assuming the topology of the codomain is $L^3_\loc(\RR)$.\\
    Finally, there exist $\psi,\psi^\pm\in L^6_\loc(\RR)$, with $1/\psi,1/\psi^\pm\in L^6_\loc(\RR)$, $\psi_x,\psi^\pm_x\in L^2_\loc(\RR)$, $\psi, \psi^\pm >0$ a.e., uniquely defined up to a positive multiplicative constant by the systems
    \begin{equation}
        \left\{
        \begin{aligned}
            &(\de_y-\de_x^2+u)\psi^\pm=0,\\
            &v^\pm=\de_x(\log(\psi^\pm)),
        \end{aligned}
        \right.
        \qquad\qquad
        \left\{
        \begin{aligned}
            &(\de_y-\de_x^2+u)\psi=0,\\
            &v=\de_x(\log(\psi)),
        \end{aligned}
        \right.
    \end{equation}
    and it holds, up to a positive multiplicative constant,
    \[
    \psi=(1-\theta)\psi^++\theta\psi^-
    \]
    for a unique $\theta\in(0,1)$ that depends bijectively on $\gamma_0$ for fixed $u$.
\end{ABCtheorem}
Morally, the map $\dtsmV$ is the inverse of the Miura map $M_-$ (more precisely, $\dtsmV$ is a right inverse of the map $(M_-,v\mapsto\gamma_0)$ defined on a suitable domain). Note that the products $u\psi$, $u\psi^\pm$ are well-defined with the regularity assumptions of the involved distributions.

The explicit formula in the end is motivated in Section \ref{sec:preliminaries} and it essentially relies on the Cole--Hopf transformation to turn Burgers' equation into a linear heat equation. Simple solutions of \eqref{eq:miuraEquation} can be combined using the linearity of the heat equation to obtain new solutions: in particular, solutions with different limits at infinity can be written as a nonlinear superposition of some special, atomic solutions which we call \emph{elementary solutions}, which are constant at infinity. This idea is especially useful to treat the time-dependent problem in Theorem \ref{theorem:theorem_2}. Note that, as discussed in Section \ref{sec:preliminaries},
\[
L_u=\de_y-\de_x^2+u
\]
is the Lax operator of the KP-II equation.

Thanks to Theorem \ref{theorem:theorem_1}, we can define the soliton addition map $\samB$. We give an additional definition that is needed to state the bounds on $\samB$.
\begin{definition}[The soliton addition map]\label{def:soliton_addition_map_and_L^2_phi(RR)}
We define the line-soliton addition map (or B\"acklund transform) of KP-II as
\begin{equation}\label{eq:backlund-transformation-definition-introduction}
    \samB(u,\gamma_0):=u-2\partial_x\dtsmV(u,\gamma_0),
\end{equation}
with $\dtsmV$ as in Theorem \ref{theorem:theorem_1}. Moreover, for $\wu\in \mathscr D'(\RR)$, we define
    \als{
        |\wu|^2_{L^2_\ph(\RR)}&:=\inf\left\{\|w\|_{L^2(\RR)}^2+\|\sigma_y\|^2_{L^2(\R_y)}\,\big|\,\wu=w+\ph_\sigma\right\},\\
        L^2_\ph(\RR)&:=\left\{\wu\in\mathscr D'(\RR)\,|\,|\wu|_{L^2_\ph(\RR)}<\infty\right\}.
        }
\end{definition}
With some more work, we obtain the following Corollary of Theorem \ref{theorem:theorem_1}.
\begin{ABCcorollary}\label{cor:almost-conservation-L2-norm-around-line-soliton}
    In the hypotheses of Theorem \ref{theorem:theorem_1}, if in addition $u\in L^2(\RR)$ is small enough, then $\wu:=\samB(u,\gamma_0)\in L^2_\ph(\RR)$ and it holds the double $L^2$-estimate
    \[
    \|u\|_{L^2(\RR)}\sim |\wu|_{L^2_\ph(\RR)}.
    \]
\end{ABCcorollary}

The next Theorem states that the transformation $\samB$ in \eqref{eq:backlund-transformation-definition-introduction} commutes with the flow of \eqref{eq:KP-II}, up to the choice of the additional 1-dimensional parameter. The statement relies on the global well-posedness theory of KP-II for small data in $\Hcrit$ proved in \cite{hadacHerrKoch2009wellPosednessKP-IIinCriticalSpace}, see Theorem \ref{theorem:global-well-posedness-KP-II-HHK09}.

\begin{ABCtheorem}\label{theorem:theorem_2}
    Let $u_0\in\Hcrit$ be small enough, and $\gamma_{0,0}\in\mathbb R$. Let $u\in C_b([0,\infty),\Hcrit)$ be the global solution of \eqref{eq:KP-II} with $u|_{t=0}=u_0$. There exists a continuous function $t\mapsto\gamma_0(t)$, $\gamma_0(0)=\gamma_{0,0}$ such that the curve $\wu(t):=\samB(u,\gamma_0(t))$ lies in $L^2_\loc([0,\infty)\times \RR)$, is a solution of the KP-II equation in distributional form, and can be decomposed as $\wu(t)=\ph(x-\alpha(t,y))+u(t)+w(t)$, with the estimates
    \[ \sup_{t\geq 0}\left[\sup_{y_1,y_2\in\R}\frac{|\alpha(t,y_2)-\alpha(t,y_1)|}{\log(2+|y_2-y_1|)}+\| \alpha_y(t,\cdot)\|_{L^2_y}+\|w(t,\cdot,\cdot)\|_{H^{-\aha,0}(\RR)}\right]\lesssim \|u_0\|_{\Hcrit},
    \]
    \[
    \left\|\fr d{dt} \gamma_0-4\right\|_{L^2_\unif(0,\infty)}\lesssim\|u_0\|_{\Hcrit}.
    \]
    If in addition $u_0\in L^2(\RR)$ and is small enough, then $\wu-\ph(x-\alpha(0,y)-4t)\in C([0,\infty),L^2(\RR))$, it holds the double estimate
    \[
    |\wu(t)|_{L^2_\ph(\RR)}\sim|\wu(0)|_{L^2_\ph(\RR)},\qquad t\geq 0,
    \]
    and $\wu$ is the solution of KP-II coming from the well-posedness theory (see Proposition \ref{prop:KP-IIglobalWellPosednessAroundSlackSoliton}).
\end{ABCtheorem}
The second part of the above theorem shows that $|\cdot|_{L^2_\ph}$ is an almost-conserved quantity for small solutions of \eqref{eq:KP-II} around a modulated line soliton, at least for those solutions generated by our B\"acklund transform $\samB$. We conjecture that the same holds for all small $L^2$ perturbations of $\ph$.

Next, we provide information on the range of $\samB$. The following Theorem provides necessary and sufficient conditions in the restricted setting of small output data in a mildly weighted space. For $\eps>0$, define the weighted space $L^1_{\sech^2,\eps}(\RR)$ with norm
\[
\|f\|_{L^1_{\sech^2,\eps}(\RR)}:=\|(1+|y|)^{\eps}\sech^2(x)f\|_{L^1(\RR)},
\]
and the Banach space
\[
Y_\eps(\RR):=\Hcrit\cap L^2(\RR)\cap L^1(\RR)\cap L^1_{\sech^2,\eps}(\RR).
\]
\begin{ABCtheorem}[Codimension-1 manifold in the range of $\samB$]\label{theorem:range_of_samB_contains_manifold_of_cod_1}
    There exists an analytic map
    \[
    \Phi:L^1(\RR)\cap L^2(\RR)\to\R
    \]
    with $\Phi(0)=0$, $D\Phi(0)\cdot \dot g=\aha\int_{\RR}\ph\dot g\,dx$, with non-vanishing differential everywhere, and invariant under the reflection symmetry \eqref{eq:reflection-symmetry-KP-II}, such that the following holds.\\
    Let $\eps>0$. Let $g\in Y_\eps(\RR)$ be small enough (depending on $\eps$). The following are equivalent:
    \begin{itemize}
        \item $\Phi(g)=0$,
        \item there exists a small $u\in \Hcrit$ and $\gamma_0\in\R$ such that $g+\ph=\samB(u,\gamma_0).$
    \end{itemize}
\end{ABCtheorem}
The functional $\Phi$ is constructed in Section \ref{sec:range}. It involves solving a parabolic equation with the input $g$ as multiplicative potential, and its power series around $g=0$ can be computed explicitly. The functional $\Phi$ is connected to the KP-II scattering transform: the number $\Phi(g)$ corresponds to the coefficient multiplying a singular term in the continuous scattering data of the function $\wu=\ph+g$ (the nonlinear Fourier transform of $\wu$) at the two spectral parameters corresponding to the line soliton, which are set to $\pm 1$ in this paper. These singularities are described by Wu in \cite{Derchyi-Wu-2020-direct-scattering-Gr12-KP-II-soliton}*{equation (3.16)} (the reader can compare $\gamma=\gamma_1$ in \cite{Derchyi-Wu-2020-direct-scattering-Gr12-KP-II-soliton}*{equation (3.8)} with $\Phi(g)$ in our Definition \ref{def:Phi} and note that they essentially coincide). The condition $\Phi(g)=0$ is thus roughly equivalent to the scattering transform of $\ph+g$ being non-singular. We remark that the above is in clear contrast with the KdV scattering transform, where all the information concerning solitons is contained in the discrete scattering data, and the continuous scattering data of sufficiently localized solutions, including perturbed solitons, are non-singular.

By combining Theorems \ref{theorem:theorem_2} and \ref{theorem:range_of_samB_contains_manifold_of_cod_1} with the conservation of the $L^2$ norm along the KP-II flow, we obtain an $L^2$ stability result for the line soliton in a codimension-1 manifold of data of sharp ($L^2$) regularity.
\begin{ABCcorollary}\label{cor:codimension_1_stability}
    For every $\eps>0$, there exists $\delta=\delta(\eps)>0$ such that the line soliton of KP-II is modulationally stable in $L^2$, in the sense of Definition \ref{def:modulational_orbital_stability}, under perturbations in the manifold
    \[
    \Ncal=\{g\in Y_\eps(\RR)\;|\;\Phi(g)=0,\,\|g\|_{Y_\eps(\RR)}<\delta\},
    \]
    which is an analytic, regular codimension-1 submanifold of $Y_\eps(\RR)$ containing the origin.
\end{ABCcorollary}

The assumptions in Corollary \ref{cor:codimension_1_stability} are sharp in terms of regularity, unlike previous results on the stability of the KP-II line soliton on $\RR$. Although we do not provide details here, there seem to be no obstructions in using the properties of $\samB$ to prove asymptotic stability of the line soliton under perturbations in the above manifold by making use of scattering of small $\Hcrit$ solutions of \eqref{eq:KP-II}, proved in \cite{hadacHerrKoch2009wellPosednessKP-IIinCriticalSpace}.

The stability theorems for the line soliton proved by Mizumachi do not assume any finite-codimension condition. In fact, we expect the codimension-1 condition $\Phi=0$ in Corollary \ref{cor:codimension_1_stability} to be removable. A key difference between our works is that in this paper we do not consider modulations in the scaling parameter $\lambda$ of the line soliton, which could be related to this discrepancy.

On the other hand, as Theorem \ref{theorem:range_of_samB_contains_manifold_of_cod_1} states, the condition $\Phi=0$ is necessary for a perturbation of the line soliton to fall in the range of the soliton addition map $\samB$. The proof and the conclusion of Theorem \ref{theorem:range_of_samB_contains_manifold_of_cod_1} seem to not depend sensibly on the function space used to define $\samB$ as long as the space is scaling critical, a condition that in turn is natural to have uniqueness of solutions of \eqref{eq:miuraEquation} and a well-defined map $\samB$. Moreover, as noted after Theorem \ref{theorem:range_of_samB_contains_manifold_of_cod_1}, the functional $\Phi$ is directly linked to the scattering transform. In particular, the condition $\Phi=0$ is a no-singularity condition for the continuous scattering data of $\ph+g$. We conclude that this `missing degree of freedom' in the soliton addition map is an intrinsic feature of the B\"acklund transform and the KP-II equation itself, arising from the interplay between analytic and algebraic properties of the integrable structure. This is remarkable, and contrasts with the common intuition, which has been shown to be valid at least for several 1-dimensional models, that solitons of integrable PDEs can be thought as being entirely independent, and in fact easily removable from the rest of the solution. The failure of this description is linked to the unbounded nature of KP-II solitons. We state a conjecture on the connection between this degree of freedom and the long time behavior of perturbed line solitons in Section \ref{sec:range}, which we did not see in previous works on the line soliton.

It remains unclear whether there exists a well-behaved generalization of the soliton addition map that can describe generic perturbations of the line soliton, without the codimension-1 condition. On the other hand, we think that some modification of the arguments needed to prove Corollary \ref{cor:codimension_1_stability} can lead to a complete stability result.

We conclude by noting that the approach to the stability of solitons using B\"acklund transforms shows promise for potential generalizations to the multisoliton case. In section \ref{sec:six} we derive the analogous B\"acklund transform for a subclass of KP-II multisolitons, and we plan to discuss analogous stability results to that in Corollary \ref{cor:codimension_1_stability} in a follow-up paper.

\subsubsection{Structure of the paper}
Our plan for the present article is as follows. In section \ref{sec:preliminaries}, we discuss the formalism of the Lax equation, its connection with \eqref{eq:KP-II}, \eqref{eq:mKP-II} and the Miura map, and all the tools and heuristics that we derive from the Lax equation in order to study the KP-II equation. We also focus on the construction of multisolitons and give a heuristic discussion on the elementary solutions of system \eqref{eq:systemForlittlev-miura+mKP-II}.

In Section \ref{sec:space}, we study the Miura map for fixed time. We prove Theorem \ref{theorem:theorem_1}, which allows to define the B\"acklund transform in Definition \ref{def:soliton_addition_map_and_L^2_phi(RR)}, and Corollary \ref{cor:almost-conservation-L2-norm-around-line-soliton}. We provide there the main tools needed for the construction of the elementary solutions.

In Section \ref{sec:time}, we look at the time-dependent problem and prove Theorem \ref{theorem:theorem_2}. We first review the well-posedness and regularity properties of solutions of the KP-II equation, then we define the elementary solutions of \eqref{eq:systemForlittlev-miura+mKP-II} and show their basic properties. Before the proof of Theorem \ref{theorem:theorem_2}, we state the nonlinear superposition of elementary solutions in Proposition \ref{prop:eleVt_superposition-of-elementary-functions}, which allows to select the time-dependent parameter $\gamma_0$ in Theorem \ref{theorem:theorem_2}.

In section \ref{sec:range}, we prove Theorem \ref{theorem:range_of_samB_contains_manifold_of_cod_1} and discuss a conjecture on the codimension-1 condition $\Phi=0$.

In section \ref{sec:six}, we briefly discuss the multisoliton addition map.

\subsection*{Acknowledgments}

The author wishes to express his deepest gratitude to Professor Herbert Koch for his constant guidance, support, and exceptional generosity. The author is indebted to Professor Tetsu Mizumachi for his kind hospitality and the invaluable discussions during a visit to Hiroshima.

The author acknowledges the support of the DAAD through the program `Graduate School Scholarship Programme, 2020' (57516591), and of the Hausdorff Center for Mathematics under Germany's Excellence Strategy--EXC--2047/1-390685813.

\subsection{Notation}\label{subsec:notations}
\begin{itemize}
    \item $x,y,t$ are the space-time variables, with corresponding frequency variables $\xi,\eta,\tau$. We will often use the space variables as subscripts, for instance $\R_x$, $\RR_{x,y}$, whenever we are working with functions that depend only on some of the variables.
    \item (Subscript notation; cf. Notation \ref{notation:f_alpha}) For a function $f$ of the variable $x$ and $\alpha\in\R$, we will denote by $f_\alpha$ the function
    \[
        f_\alpha(x):=f(x-\alpha).
    \]
    More generally, $f$ and $\alpha$ will be allowed to depend on $y$ or $t,y$: in that case, $f_\alpha$ will denote
        \begin{equation}
            f_\alpha(t,x,y):=f(t,x-\alpha(t,y),y).
        \end{equation}
    The only exception to this notation rule is the function $G_\alpha$ in Definition \ref{def:VandG}.
    \item Throughout this article, we define the two functions of the $x$ variable $$ \eta^+(x):=\aha(1+\tanh(x)),\quad \eta^-(x):=\aha(1-\tanh(x)). $$
    \item (Convention on hyperbolic functions). The following statements are written for the function `$\tanh$', and apply as well to all other hyperbolic functions ($\cosh,\sech,...$) and to $\ph$, $Q$, $\eta^\pm$.
    \begin{itemize}[label=\textopenbullet]
        \item When we want to evaluate the function at a given number $x$, we will write $\tanh(x)$. We will \emph{never} write `$\tanh x$'.
        \item When taking products between $\tanh$ and functions inside brackets, we will write for instance $\tanh\cdot (f+g)$.
        \item When not given an argument as input, we will consider $\tanh$ as a function of the $x$ variable. When comparing it to functions with more variables, we will think of $\tanh$ as being constant in the other variables.
    \end{itemize}
    \underline{Examples:} for $w=w(x,y)$ and $\alpha=\alpha(y)$:
    \begin{itemize}[label=\textopenbullet]
        \item the expression `$\tanh\sech^2$' denotes the map $x\mapsto \tanh(x)\sech^2(x)$,
        \item the expression `$\cosh w$' denotes the map $(x,y)\mapsto \cosh(x)w(x,y)$,
        \item the expression `$\cosh_\alpha w+\ph_\alpha$' denotes $(x,y)\mapsto \cosh(x-\alpha(y))w(x,y)+\ph(x-\alpha(y))$.
    \end{itemize}
    \item $C(A,B)$ denotes the space of continuous maps from $A$ to $B$. For $\Omega$ a metric space and $X$ a Banach space, we denote by $C_b(\Omega,X),C_0(\Omega,X)\subset C(\Omega,X)$ the Banach spaces of continuous functions that are respectively bounded, and small outside compact sets (more precisely, $C_0(\Omega,X)$ is the closure of the space $C_c(\Omega,X)$ of compactly supported continuous functions). We will omit $X$ when $X=\R$. In one case, $X$ will be a Fr\'echet space, and we will rely on the notion of boundedness in such spaces.
    \item In $\R^n$, with variable $z$ and frequency variable $\zeta$, we define the usual smooth Littlewood--Paley projector $P_\lambda$ on dyadic annuli with frequencies $\zeta\sim \lambda\in 2^\Z$, as for instance in \cite{bahouri-chemin-danchin-2011-fourier-analysis-and-nonlinear-PDEs}. We define $P_{\leq \lambda}$, $P_{>\lambda}$ similarly.
    \item If $z=(x,y,\dots)$, $\zeta=(\xi,\eta,\dots)$, we denote by $P_\lambda^x$ the Littlewood--Paley projection with respect to the frequency variable $\xi$ only (analogously for $P_{\leq \lambda}^x$, $P_{>\lambda}^x$).
    \item We define the homogeneous and inhomogeneous Besov norms as
    $$ \|u\|^q_{\dot B^s_{p,q}}:=\sum_{\lambda\in 2^\Z}\|P_\lambda u\|^q_{L^p}, \qquad \|u\|^q_{B^s_{p,q}}:=\|P_{\leq 1}u\|_{L^p}^q+\sum_{\lambda\in 2^\Z,\,\lambda>1}\|P_\lambda u\|^q_{L^p}, $$
    with the obvious modification for $q=\infty$, and denote by $\dot B^s_{p,q}(\R^n)$ and $B^s_{p,q}(\R^n)$ the respective Besov spaces. When $p=q=2$, we will write $H^s$ instead of $B^s_{2,2}$.
    \item Given a Banach space $X$ and an interval $I$, the norm of the associated Bochner space $L^p(I,X)$ is denoted by
    \[
    \|u\|_{L^p X}^p=\int_I \|u(s)\|_X^pds.
    \]
    When the interval is not $\R$ or $[0,\infty)$, or when $I=[0,T]$, the norm will be denoted respectively by $\|u\|_{L^p_I X}$, $\|u\|_{L^p_T X}$.
    \item The spaces $H^{s_1,s_2}(\R^2)$ are defined by the norm
    $$ \|u\|_{H^{s_1,s_2}(\R^2)}^2=\int_{\RR} \jap{\xi}^{2s_1}\jap{\eta}^{2s_2}|\wh u(\xi,\eta)|^2d\xi \,d\eta, $$
    with the obvious modification for the homogeneous version (cf. \cite{hadac2008well-posednessKP-II}, \cite{hadacHerrKoch2009wellPosednessKP-IIinCriticalSpace}).
    \item The anisotropic Besov spaces $\dot B_{p,q}^{s,t}(\RR)$ are defined by the norm
    $$\|u\|_{\dot B_{p,q}^{s,t}}=\left(\sum_{\lambda\in 2^\Integers} \lambda^{qs}\| P_\lambda^x|D_y|^t u\|_{L^p(\RR)}^q\right)^{1/q},$$
    with the obvious modification for $q=\infty$. Note that the norm of $\dot B_{p,q}^{s,0}$ is invariant under the KP-II Galilean symmetry.
    \item For an open set $\Omega\subset\Real^d$, we define $L^p_\unif (\Omega)$ as the Banach space induced by
    \[ \|g\|_{L^p_{\unif}(\Omega)}:= \sup_Q \|g\|_{L^p(Q\cap\Omega)}, \]
    where the supremum is taken over all dyadic cubes of $\Real^d$ with side $1$.
\end{itemize}

\subsection*{Cross-reference list}

\begin{itemize}
    \item The maps $M_\pm^\lambda$ and the transformations $\Bcal,\Gcal^\KP,\Gcal^\KdV$ are defined in the introduction. $\Gcal^\mKP$ is defined in Subsection \ref{subsec:manySystems}

    \item $|\cdot|_{L^2_\ph(\RR)},\,\, L^2_\ph(\RR)$, see Definition \ref{def:soliton_addition_map_and_L^2_phi(RR)}.
    
    \item $Y_\eps(\RR)$ is defined before Theorem \ref{theorem:range_of_samB_contains_manifold_of_cod_1}.

    \item The map $\dtsmV$ is defined in Theorem \ref{theorem:theorem_1}, and characterized explicitly in Lemma \ref{lemma:change-of-variables-c-alpha_0-gamma_0}.
    
    \item The map $\samB$ (soliton addition map, or KP-II B\"acklund transform), see Definition \ref{def:soliton_addition_map_and_L^2_phi(RR)}.

    \item The map $\eleV$ is a generalization of $\dtsmV$ up to a change of variables, see Proposition \ref{prop:kink-addition_map_eleV}.
    
    \item The map $\eleVt$ is a time-dependent version of $\eleV$, see Proposition \ref{prop:eleVt_superposition-of-elementary-functions}.

    \item $v^\pm$, $\wv^\pm$, see Corollary \ref{cor:estimates-for-wtv-part-2}.
    
    \item $G_\alpha$, see Definition \ref{def:VandG}

    \item $\Gamma,\Gamma^{(c)},\Gamma^\pm$, see Definition \ref{def:heat-operators-in-appendix}.

    \item The maps $\eleB$, $\eleBt$ are generalizations of $\samB$ to $(k,1)$-multisolitons, see Section \ref{sec:six}.

    \item Equations \eqref{eq:KP-II}, \eqref{eq:mKP-II}, see the beginning of the introduction.

    \item Equations \eqref{eq:mKP-II}, \eqref{eq:miuraEquation}, see Subsection \ref{subsec:description}.

    \item System \eqref{eq:systemForlittlev-miura+mKP-II}, see Subsection \ref{subsec:manySystems}.

    \item $C^{0,\alpha}_\unif$, see Lemma \ref{lemma:holder_regularity_for_difference_of_primitive_solutions}
\end{itemize}

\section[Preliminaries: the Miura map and the integrability]{Preliminaries: the role of the integrability and the Miura map}\label{sec:preliminaries}

\subsection{Lax pair and compatibility condition}
The KP-II equation is a completely integrable PDE, with a structure resembling the one of the KdV equation.
\subsubsection{Lax-pair formulation}
Fix $\sigma\in\{\pm 1\}$. The KP-II equation is related to the operators defined below called the \emph{Lax pair}, the first of which is called \emph{Lax operator}:
\begin{align*}
    L_{u(t)}&:=\sigma\partial_y-\partial_{xx}+u(t),\\
    P_{u(t)}&:=-4\partial_{xxx}+3\left(u(t)\partial_x+\partial_xu(t)\right)+3\sigma(\partial_x^{-1}u_y(t)),
\end{align*}
where $u(t),u_x(t),(\partial_x^{-1}u_y(t))$ are meant as the multiplication operators by the same functions. Under the hypothesis that $\partial_x^{-1}u_y(t)$ is well-defined in a suitable sense, the function $u$ solves \eqref{eq:KP-II} if and only if the associated Lax pair satisfies the equation
\[\frac{d}{dt}L_{u(t)}=[P_{u(t)},L_{u(t)}],\]
where $[P,L]:=PL-LP$ is the commutator. As already anticipated, solutions $u$ of \eqref{eq:KP-II} generally lie in Banach spaces of functions on which the operator $\demu\partial_y$ is well-defined (see Remark \ref{rk:Z-aha-is-in-L2-loc}), and $\demu\partial_y u$ coincides with the term appearing in the KP-II equation.
\subsubsection{Compatibility condition}
As for the KdV equation, the KP-II equation is formally equivalent to the compatibility condition for the \emph{Lax system}
\begin{equation}\label{eq:systemCompatibility-original}
    \left\{
    \begin{aligned}
    &\sigma\psi_y-\psi_{xx}+u\psi=-\lambda^2\psi\\
    &\psi_t=-4\psi_{xxx}+6u\psi_x+3u_x\psi+3\sigma(\de_x^{-1}u_y)\psi,
    \end{aligned}
    \right.
\end{equation}
for a given $\lambda\in\C$, which says how generalized eigenfunctions of the Lax operator evolve in time. This formulation is related to the Miura map as shown in the next subsection. We note that in system \eqref{eq:systemCompatibility-original} we can make use of the transformation $\psi\mapsto e^{\lambda^2 y}\psi$ to set $\lambda=0$ (this feature is not present in the Lax operator of \eqref{eq:kdv}). Moreover, the choice of $\sigma$ corresponds to a reflection in the $y$ coordinate. We will thus mostly consider the case $\lambda=0$, $\sigma=1$,
\begin{equation*}\label{eq:systemCompatibility}\tag{{\theequation}$'$}
    \left\{
    \begin{aligned}
    &\psi_y-\psi_{xx}+u\psi=0\\
    &\psi_t=-4\psi_{xxx}+6u\psi_x+3u_x\psi+3(\de_x^{-1}u_y)\psi.
    \end{aligned}
    \right.
\end{equation*}
When referring to the Lax operator, we will always think of the one with $\sigma=1$.

\subsection{Relation between the Lax pair, the Miura map and mKP-II}\label{subsec:manySystems}
\subsubsection{The mKP-II equation in system form}
Unlike \eqref{eq:KP-II}, \eqref{eq:mKP-II} needs an auxiliary function to be written down without using antiderivatives. The mKP-II equation in system form reads
\begin{equation}\label{eq:mKP-II-systemForm}
        \left\{
    \begin{aligned}
    &w_x=v_y\\
    &v_t+v_{xxx}-6v^2v_x+3w_y +6v_xw=0.
    \end{aligned}
    \right.
\end{equation}
There is another clever rewriting that, interestingly enough, can be derived from the Lax equation. In fact, system \eqref{eq:systemCompatibility} is equivalent, under the change of variables $\psi=e^V$, to the system
    \begin{equation}\label{eq:systemForCapitalV}
        \left\{
    \begin{aligned}
    &V_y-V_{xx}=V_x^2-u\\
    &V_t+4V_{xxx}+4V_x^3+12V_xV_{xx}-6uV_x-3u_x-3\de_x^{-1}u_y=0.
    \end{aligned}
    \right.
    \end{equation}
    Derivating with respect to the $x$ variable and setting $v=V_x$, we obtain the following system:
    \begin{equation}\label{eq:systemForlittlev-miura+mKP-II}\tag{$\Mcal$--mKP-II}
        \left\{
    \begin{aligned}
    &v_y-v_{xx}=(v^2)_x-u_x\\
    &v_t+4v_{xxx}+(12vv_{x}+4v^3-6uv-3u_x)_x-3u_y=0.
    \end{aligned}
    \right.
    \end{equation}
It is an easy exercise to show that systems \eqref{eq:mKP-II-systemForm} and \eqref{eq:systemForlittlev-miura+mKP-II} are equivalent for $u,v,w\in L^p_\loc$ with $p$ large enough, under the change of variables $w=-u+v^2+v_x$. In other words, system \eqref{eq:systemForlittlev-miura+mKP-II}, which in turn is just the compatibility system \eqref{eq:systemCompatibility} which describes the evolution of the Lax operator, is nothing but a reformulation of \eqref{eq:mKP-II} in system form which contains equation \eqref{eq:miuraEquation}. This fact provides insights on the link between the two dispersive models and is going to be useful.
Coming back to the discussion on the eigenvalue $-\lambda^2$, we also note that $v=\partial_x\log(\psi)$, so that the transformation $\psi\mapsto e^{\lambda^2 y}\psi$ used to change the value of $\lambda$ leaves the function $v$ unchanged, this is one more reason why we should not worry about the eigenvalue $\lambda$ in the equation and set it to be zero. Note that if we had to take into account $\lambda^2$, the function $V$ would be shifted by the additive term $-\lambda^2 y$.

We note that the map $\psi\mapsto V=\log(\psi)$ is bijective, assuming $\psi>0$. The map $V\mapsto v=V_x$ is clearly not injective a priori without further conditions, although it turns out (see Lemma \ref{lemma:existenceOfPrimitiveSolution} and Lemma \ref{lemma:existenceOfPrimitiveSolutionTimeDependent}) that $V$ is determined from $v$ only up to an additive constant which is independent of space and time variables (hence, $\psi$ is determined from $v$ up to a multiplicative constant), assuming $V$ satisfies system \eqref{eq:systemForCapitalV}. The free undetermined constant is natural due to the linearity of the Lax system \eqref{eq:systemCompatibility}.

\subsubsection{Properties of the mKP-II equation in system form, \eqref{eq:systemForlittlev-miura+mKP-II}}
We first understand the symmetries of the mKP-II equation in a broad sense, that is, we look for symmetries of the system \eqref{eq:systemForlittlev-miura+mKP-II}. The following two symmetries can be both guessed from the KP-II symmetries by looking at the Miura map:
\begin{equation}\label{eq:mKP-II-symmetry}
(u,v)\mapsto(\Gcal^{\KP}_\sigma u,\Gcal^{\mKP}_\sigma v),
\end{equation}
\begin{equation}\label{eq:KdV-symmetry-for-mKP-II}
(u,v)\mapsto(\Gcal^{\KdV}_\mu u,\Bcal_\mu v),
\end{equation}
where $\Gcal^\mKP_\sigma v=\Gcal^\KP_\sigma v-\frac{\sigma}{2}$. These translate into symmetries of mKP-II in system form \eqref{eq:mKP-II-systemForm}, since the two systems are equivalent. The two symmetries do not necessarily correspond to symmetries of \eqref{eq:mKP-II} in a well-posed setting, especially the second one. These symmetries are possible due to the arbitrary choice of an additive constant in the term $\partial_x^{-1}v_y$, which in a general setting need not be determined univocally and canonically from $v$.

Now we give a precise version of the property of the Miura map to conjugate the mKP-II flow and the KP-II one, which is stated and proved in \cite{kenigMartel2006mKP-IIwellPosedness} in terms of system \eqref{eq:mKP-II-systemForm}.
\begin{proposition}\label{prop:miura-maps-mKP-to-KP}
    Let $\Omega\subset \R^3$ open. Assume $v\in L^4_\loc(\Omega)$, $v_x,w\in L^2_\loc(\Omega)$ satisfy the system \eqref{eq:mKP-II-systemForm}. Then, $u:=-w+v^2+v_x$ and $\wu:=-w+v^2-v_x$ belong to $L^2_\loc(\Omega)$ and solve the KP-II equation distributionally, that is,
    \begin{equation}\label{eq:KP-II-distributional}
        (u_t-3(u^2)_x+u_{xxx})_x+3u_{yy}=0,
    \end{equation}
    and the same holds for $\wu$. Equivalently, if $v$ as above and $u\in L^2_\loc(\Omega)$ satisfy the system \eqref{eq:systemForlittlev-miura+mKP-II}, then $u$ and $\wu:=u-2v_x$ satisfy \eqref{eq:KP-II-distributional}.
\end{proposition}

\subsection{Multisolitons and elementary solutions}\label{subsec:multisolitons}
The following is a way of constructing solutions of the KP-II equation and is described for instance in \cite{ChakravartyKodama2007ClassificationOTlineSoliton} and \cite{kodama2017solitonsAndGrassmannians}. Let $f_1,\dots,f_N$ be positive solutions of the system
\begin{equation}\label{eq:linear-system-to-generate-f_j's}
    \left\{
\begin{aligned}
    &\partial_yf=\partial_x^2f,\\
    &\partial_tf=-4\partial_x^3 f.
\end{aligned}
\right.
\end{equation}
Define the $\tau$ function as the Wronskian determinant
\begin{equation}\label{eq:wronskian-determinant}
    \tau(x, y, t)=\operatorname{Wr}\left(f_1, \ldots, f_N\right)=\left|\begin{array}{cccc}
f_1 & f_2 & \cdots & f_N \\
f_1^{\prime} & f_2^{\prime} & \cdots & f_N^{\prime} \\
\vdots & \vdots & & \vdots \\
f_N^{(N-1)} & f_2^{(N-1)} & \cdots & f_N^{(N-1)}
\end{array}\right|,
\end{equation}
where $f_k^{(j)}:=\partial_x^j f_k$. It is known that the function
\begin{equation}\label{eq:tau-generates-u-sol-of-KP-II}
    \wu(t,x,y):=-2\partial_x^2\log(\tau(x,y,t))
\end{equation}
formally solves the KP-II equation as long as the $\tau$ function does not vanish.

It is known (see e.g. the paper \cite{lin-zhang-2019-ancient-solutions-heat-equaiton}) that any non-negative solution of the heat equation on the whole $\R_x\times\R_y$ can be achieved as an integral sum of exponential solutions $e^{\lambda x+\lambda^2 y}$, in particular any positive solution of system \eqref{eq:linear-system-to-generate-f_j's} has to have the form
\[
f(x,y,t)=\int_\R e^{\lambda x+\lambda^2y-4\lambda^3 t} d\mu(\lambda),
\]
where $\mu$ is a non-negative Borel measure on $\R$. The multisolitons of KP-II can be generated using formula \eqref{eq:wronskian-determinant} by taking the functions $f_j$ as finite sums of the above exponential solutions: we describe below the procedure, which loosely follows the notation of \cite{ChakravartyKodama2007ClassificationOTlineSoliton}. First choose two positive integers $0<N<M$, a matrix $A=(a_{nm})\in \operatorname{Mat}_{N\times M}(\R)$, real spectral parameters $\lambda_1<\lambda_2<\dots<\lambda_M$, and phase parameters $\theta_{1,0},\dots,\theta_{M,0}$. Let
\[
\theta_m(x,y,t):=\lambda_mx+\lambda_m^2y-4\lambda_m^3t+\theta_{m,0}\,, \qquad f_n(x,y,t)=\sum_{m=1}^M a_{nm}e^{\theta_m(x,y,t)}.
\]
Assuming $\operatorname{Rank(A)}=N$ and that all $N\times N$ minors of $A$ are non-negative and some irreducibility condition on $A$ (see \cite{ChakravartyKodama2007ClassificationOTlineSoliton}*{Condition 2.2}), the function $\tau$ in \eqref{eq:wronskian-determinant} is strictly positive, and the function $\wu$ in equation \eqref{eq:tau-generates-u-sol-of-KP-II} is a $(M-N,N)$-multisoliton. The term means that for fixed time, the solution looks like a sum of $M$ tilted `half line' solitons coming from infinity, which exhibit a nontrivial interaction in a compact region of space. Moreover, of the $M$ half lines interacting, $N$ come from $y=+\infty$, and $M-N$ come from $y=-\infty$. The directions of the $M$ solitons are given by the relations between the spectral parameters $\lambda_1,\dots,\lambda_M$. If we take $N=1$, it is not possible to construct all multisolitons, but only the so called $(k,1)$-multisolitons, $k=M-1\geq 1$, which possess one line as $y\to+\infty$ and $k$ lines as $y\to-\infty$. We called them \emph{tree-shaped} solitons. For $M=2$, one obtains the line soliton, possibly transformed under the KP-II Galilean symmetry \eqref{eq:KP-IIgalileanSymmetry}. Taking $M=3$, the corresponding solution is the $(2,1)$-soliton, which is known as \emph{Miles resonance solution}, or as `$Y$-shaped' soliton. The reflection symmetry \eqref{eq:reflection-symmetry-KP-II} can be used to exchange the number of outgoing solitons at $y\to-\infty$ and $y\to\infty$.

\subsubsection{Solutions of \eqref{eq:miuraEquation} as superpositions of elementary solutions}
The following simplified discussion is enough to present one of the key ideas in this work. If $u$ is small in a suitable sense, one can expect that for a given parameter $\lambda_m\in\R$, there exists a time-dependent eigenfunction $\psi^{(m)}$ of the Lax operator, solution of \eqref{eq:systemCompatibility}, which behaves like $e^{\theta_m(x,y,t)}$ up to multiplicative lower order terms. We call such functions $\psi^{(m)}$ \emph{elementary Lax-eigenfunctions} and are essentially the \emph{Jost solutions} of the Lax operator with potential $u$. Thanks to the equivalence between systems \eqref{eq:systemCompatibility}, \eqref{eq:systemForCapitalV}, \eqref{eq:systemForlittlev-miura+mKP-II}, the problem of constructing $\psi^{(m)}$ essentially reduces to that of finding solutions $v^{(m)}:=\partial_x\log(\psi^{(m)})$ to \eqref{eq:systemForlittlev-miura+mKP-II} which converge to constants $\lambda_m\in\R$ at infinity. The precise reduction is made possible by to Lemmas \ref{lemma:existenceOfPrimitiveSolutionTimeDependent} and \ref{lemma:primitive-solutions-time-dependent-are-nice}. These \emph{elementary solutions} $v^{(m)}$ of \eqref{eq:systemForlittlev-miura+mKP-II} are the building blocks that can be used to construct perturbations of the line soliton and multisolitons. For fixed time $t$, the function $v^{(m)}$ is a solution to \eqref{eq:miuraEquation} with datum $u(t)$ that is equal to the constant $\lambda_m$ at spacial infinity. The existence and uniqueness of such a function is essentially given by Corollary \ref{cor:estimates-for-wtv-part-2}. The elementary solutions are introduced in Definition \ref{def:elementary-solutions-of-system-Miura+mKP}.

If $(v^{(m)})_m$ are elementary solutions associated to the spectral parameters
$(\lambda_m)_m$, it follows from the linearity of system \eqref{eq:systemCompatibility} that formally
\[
v:=\partial_x\log\Big(\sum_{m=1}^Me^{\int v^{(m)}dx}\Big)
\]
solves system \eqref{eq:systemForlittlev-miura+mKP-II}. Consequently, $\wu=u-2\partial_x v$ is a solution of KP-II by the property of the Miura map in Proposition \ref{prop:miura-maps-mKP-to-KP}. This motivates the explicit formula in Theorem \ref{theorem:theorem_1}.

\section{The Miura map of the KP-II equation}\label{sec:space}

In this section we will prove Theorem \ref{theorem:theorem_1}, in particular we work with fixed, small $u\in \dot H^{-\aha,0}(\R^2)$, which is a scaling-critical space for both equations \eqref{eq:miuraEquation} and \eqref{eq:KP-II}. As mentioned in the introduction, a satisfactory global well-posedness theory of KP-II for small data in this space holds, see Theorem \ref{theorem:global-well-posedness-KP-II-HHK09}, in particular the smallness in $\Hcrit$ is preserved by the KP-II flow.

The Miura map equation in PDE form, which we rewrote as
\begin{equation}\tag*{(\ref*{eq:miuraEquation})}
    v_y-v_{xx}=(v^2)_x-u_x
\end{equation}
is a viscous Burgers equation on the whole space-time $\R_{x,y}^2$ with external forcing. The parameter $\lambda^2=1$ in the definition of $M_-$ does not appear in the equation, but it will play a role in terms of the boundary conditions. As discussed in the introduction, formally
\[
M_-^\lambda(Q^\lambda(\cdot-\gamma))=0\qquad\forall\gamma\in\R,
\]
so $Q(x-\gamma)=\tanh(x-\gamma)$ solves \eqref{eq:miuraEquation} with $u\equiv 0$ for all $\gamma\in\R$, and we expect that all solutions of \eqref{eq:miuraEquation} with the same space asymptotics of $Q$ form a 1-parameter family. As a preliminary remark, the boundary condition
\begin{equation}\label{eq:miuraEquation-loose-boundary-conditions-CauchyProblem}
    v(x,y)-\tanh(x)\to 0\quad\text{as}\quad x\to\pm\infty.
\end{equation}
is too weak to be equipped to equation \eqref{eq:miuraEquation} to have just a 1-parameter family of solutions, since viscous shocks with different velocities can come from $y=-\infty$ with different speeds. An example is given by the `multikink' solutions
\begin{equation}\label{eq:counterexample-multikink}
    u\equiv 0,\qquad v(x,y)=\partial_x\log(e^{x+y+a}+1+e^{-x+y+b}),
\end{equation}
which are all different for $a,b\in\R$. A more restrictive condition, the one that we are interested in, is to impose that $v$ is close to $Q$ for all values of $y$, so what we ask for is that
\begin{equation}
    \exists \alpha:\R\to\R \text{ such that } v(x,y)-\tanh(x-\alpha(y))\to 0\,\,\text{as}\,\, x,y\to\infty,
\end{equation}
possibly in some averaged sense. We will ask for $v-\tanh(x-\alpha(y))\in L^3(\RR)$. 

Looking at the evolution in the $y$ variable, since $u$ decays to zero, we can see that the asymptotic profiles of the solution are expected to be translations of the viscous shock $\tanh(x)$ and there should be stability of this profile at $y\to+\infty$. Defining $z:=v-\tanh(x-\gamma)$, $\gamma\in\R$, and using Notation \ref{notation:f_alpha}, the new variable $z$ solves
\begin{equation}
    z_y-z_{xx}-2(\tanh_\gamma z)_x-(z^2)_x=-u_x.
\end{equation}
This way it is clear that in \eqref{eq:miuraEquation} we have transport towards the center of the shock, and a phenomenon of accumulation of mass close to the shock itself that corresponds to a shift of the position of the kink.

\subsection{Decomposition of the solution and uniqueness}

A first guess would be to study the equation \eqref{eq:miuraEquation} in some new variables $w,\alpha$, where
$$v=w+\tanh_\alpha$$
and $\alpha=\alpha(y)$ is chosen in a suitable way so that it represents the `position' of the kink at ordinate $y$. Having chosen to work in the critical space $\Hcrit$, we will see that the decomposition will have to be a little more sophisticated to prove uniqueness of solutions. In Appendix \ref{sec:existence_of_solutions_with_L2_data} we outline the same setting with data $u\in L^2(\RR)$ for which the decomposition is fixed simply by choosing
$$ \int_\R\sech^2(x-\alpha(y))w(x,y)\dx=0\quad\forall y, $$
although a uniqueness theory in $\RR$ is missing due to the strict subcriticality of the data.

\subsubsection{Solutions of \eqref{eq:miuraEquation} that are constant at infinity}
The following Lemma shows the well-posedness of equation \eqref{eq:miuraEquation} on the whole $\RR$ in the simple case where solutions go to zero at infinity, and collects some estimates that will be useful later on. We mention that with the same techniques, it is possible to show the well-posedness of the initial value problem associated to equation \eqref{eq:miuraEquation}, that is, assigning an initial datum $v_0=v|_{y=0}\in \dot H^{-\aha}(\R)$ and a forcing $u\in \dot H^{-\aha,0}(\R\times(0,\infty))$. We refer to the definition of the anisotropic Sobolev and Besov spaces $H^{s,t}(\RR)$, $B^{s,t}_{p,q}(\RR)$ in the notations. We will write $\jap{c(x-\alpha)}^{-1}$ to denote the function $(x,y)\mapsto\jap{c(x-\alpha(y))}^{-1}$.

\begin{lemma}\label{lemma:estimates-for-wtv} Let $c\neq 0$ and $\alpha\in C(\R_y)$ be such that $\|\alpha_y-c\|_{L^2(\Real_y)}\lesssim 1$. Let $u\in \dot H^{-\aha,0}(\RR)$ small. There exists a unique solution $v\in L^3(\RR)$ to equation \eqref{eq:miuraEquation}, and it satisfies the bounds
    \begin{equation}\label{eq:estimates-for-wv-lambdaiszero}
    \begin{aligned}
        \|v\|_{C_0\dot H^{-\aha}_x\cap\dot H^{\fr12,0}\cap\dot H^{0,\fr14}\cap L^3(\RR)}+\|\demu v_y\|_{\dot B^{-1/2,0}_{2,\infty}(\RR)}\QQQ&\\
        +|c|^\aha\|\jap{c(x-\alpha)}^{-1}v\|_{L^2_{x,y}}&\lesssim \|u\|_{\Hcrit}.
    \end{aligned}
    \end{equation}
    Moreover, we have the additional estimates
    \begin{align}
        \|v\|_{C_0 L^2}+\|v_x\|_{L^2(\RR)}+|c|\|\jap{c(x-\alpha)}^{-1} v\|_{L^2(\RR)}&\lesssim \|u\|_{L^2(\RR)},\label{eq:estimates-for-wv-lambdaiszero-L2}\\
        \| v_x\|_{H^k(\RR)}+\|\demu v_y \|_{H^k(\RR)}&\lesssim_k \|u\|_{H^k(\RR)},\quad k\geq 0,\label{eq:estimates-for-wv-lambdaiszero-H^k}\\
        \|v\|_{L^2(\RR)}+\|\partial_x^{-1}v\|_{C_0L^2\cap L^6(\RR)}+|c|\|\jap{c(x-\alpha)}^{-1}&\demu v\|_{L^2(\RR)}\lesssim \|u\|_{\dot H^{-1,0}(\RR)} \label{eq:estimates-for-wv-lambdaiszero-H^-1}
    \end{align}
    whenever the right-hand sides are finite and $\|u\|_\Hcrit$ is small enough. The data-to-solution map $u\mapsto v$ is analytic in all the above topologies.
\end{lemma}
We will extensively use Proposition \ref{prop:spacetime_estimate_heat_general}, which proves the estimate
\[
\||\partial_x|^s(\de_y-\de_x^2)^{-1} u\|_{L^rL^\sigma}\lesssim \|u\|_{L^pL^q}
\]
for suitable $1\leq p,q,r,\sigma\leq\infty$, $s\in[0,2]$, as well as other linear estimates from Appendix \ref{appendix:heat-equation}.
\begin{proof}
    The existence and uniqueness of a small $v$ follows by a standard fixed point argument in the $L^3(\RR)$ topology using Proposition \ref{prop:spacetime_estimate_heat_general}. In fact, call $\Gamma$ the operator $(\partial_y-\partial_x^2)^{-1}$ defined as the integral operator on $\RR$ with respect to the heat kernel, as in Appendix \ref{appendix:heat-equation}. Then one has
    \[
    \|\Gamma \partial_x f\|_{L^3(\RR)}\lesssim \|f\|_{\Hcrit}, \qquad\|\Gamma\partial_x f\|_{L^3(\RR)}\lesssim \|f\|_{L^{3/2}(\RR)}.
    \]
    Since obviously $\|f^2\|_{L^{3/2}(\RR)}=\|f\|_{L^3(\RR)}^2$, we find that the map
    $$F: v\mapsto \Gamma\partial_x(v^2-u)$$
    is a contraction in a small ball $\overline B_{\eps_0}(0)\subset L^3(\RR)$ assuming $u\in \Hcrit$ is small enough, in particular there exists a unique small $v \in L^3(\RR)$ such that $F(v)=v$, hence the claim. The data-to-solution map is analytic due to the fact that the solution is obtained using the Banach fixed-point theorem.

    The uniqueness for large $v\in L^3$ is a consequence of the uniqueness for small $v$, translation invariance, and the fact that for small solutions $v$ one has that $v|_{\{y<-M\}}$ depends only on $u|_{\{y<-M\}}$, for any given $M\in \R$. In fact, given $M\in\R$, one can prove an analogous uniqueness statement for small $v$ with $\RR$ replaced by $\{y<-M\}$, using the same estimates as before. By translation invariance, the smallness condition does not depend on $M$. Now given two solutions $v_1$ and $v_2$ in $L^3(\RR)$, one has $\|v_1\|_{L^3(\{y<-M\})}$ and $\|v_2\|_{L^3(\{y<-M\})}$ are arbitrarily small for $M$ large enough, so by the inequality
    \als{
        \|v_2-v_1\|_{L^3(\{y<-M\})}&= \|\partial_x\Gamma(v_2^2-v_1^2)\|_{L^3(\{y<-M\})}\\
        &\lesssim \|v_2+v_1\|_{L^3(\{y<-M\})}\|v_2-v_1\|_{L^3(\{y<-M\})},
        }
    it holds that $v_1=v_2$ a.e. on ${L^3(\{y<-M\})}$ for $M$ large enough. The same is true for all $M\in \R$ by a bootstrap argument, assuming $v_1$ is the unique small solution.
    
    The estimate on the first term on the left hand side of \eqref{eq:estimates-for-wv-lambdaiszero} follows from the estimates of Proposition \ref{prop:spacetime_estimate_heat_general} and the last estimate of Remark \ref{rk:estimates_for_the_heat_operator}, since $v=\partial_x\Gamma (v^2-u)$ and 
    \[ \|v^2\|_{L^{3/2}(\RR)}\lesssim \|u\|_\Hcrit. \]

    The estimate \eqref{eq:estimates-for-wv-lambdaiszero-L2}, except for the weighted estimate, comes by Proposition \ref{prop:spacetime_estimate_heat_general} and a fixed point argument in $L^3(\RR)\cap L^6(\RR)$, which yields
    \begin{align*}
    \|\Gamma\partial_x(v^2-u)\|_{L^6(\RR)\cap C_0 L^2 \cap L^2_y \dot H^1_x}&\lesssim \|v^2\|_{L^2(\RR)}+ \|u\|_{L^2(\RR)}\\
    &\leq \|v\|_{L^6(\RR)}\|v\|_{L^3(\RR)}+\|u\|_{L^2(\RR)}.
    \end{align*}
    The weighted estimate in \eqref{eq:estimates-for-wv-lambdaiszero-L2} follows from inequality \eqref{eq:z-024} of Lemma \ref{lemma:weighted-estimates-heat}, since $v=\partial_x\Gamma(v^2-u)$, with $u,v^2\in L^2(\RR)$. For the weighted estimate in \eqref{eq:estimates-for-wv-lambdaiszero} instead, we note that $v=w+z$, $w:=\partial_x\Gamma u$, $z:=\partial_x\Gamma v^2$. The weighted estimate for $w$ follows precisely from inequality \eqref{eq:z-025} of Lemma \ref{lemma:weighted-estimates-heat} after a linear change of coordinates. For $z$, Proposition \ref{prop:spacetime_estimate_heat_general} and H\"older's inequality give us
    \als{
    |c|^\aha\|\jap{c(x-\alpha)}^{-1} z\|_{L^2(\RR)}&\leq |c|^\aha\|\jap{c(x-\alpha)}^{-1} \|_{L^\infty_yL^2_x}
    \|\partial_x\Gamma v^2\|_{L^2_yL^\infty_x}\\
    &\lesssim \|v^2\|_{L^{3/2}(\RR)}\\
    &=\|v\|_{L^3(\RR)}^2\\
    &\lesssim \|u\|_{\Hcrit}^2.
    }
    
    For the second term in estimate \eqref{eq:estimates-for-wv-lambdaiszero}, we note that we can simply recover the antiderivative of $v_y$ from equation \eqref{eq:miuraEquation},
    $$ \demu v_y=v_x+v^2-u, $$
    and $u,v_x\in \Hcrit\hookrightarrow \dot B^{-\aha,0}_{2,\infty}(\RR)$. Moreover, it holds the chain of continuous embeddings
    \[
    \dot H^{0,\fr14}(\RR)\hookrightarrow L^2_xL^4_y\hookrightarrow L^4_yL^2_x, \qquad L^2_yL^1_x\hookrightarrow \dot B^{-\aha,0}_{2,\infty}(\RR),
    \]
    so that $v^2\in B^{-\aha,0}_{2,\infty}(\RR)$ and the estimate is proved.

    Estimate \eqref{eq:estimates-for-wv-lambdaiszero-H^k} follows with standard techniques of well-posedness at higher regularity similarly as for estimate \eqref{eq:estimates-for-wv-lambdaiszero-L2}, or simply by the analyticity and translation invariance of the data-to-solution map $u\mapsto v$. The first part of estimate \eqref{eq:estimates-for-wv-lambdaiszero-H^-1} follows similarly, this time making use of the simple estimates
    \[
    \|v^2\|_{L^{6/5}(\RR)}\leq\|v\|_{L^{3}(\RR)}\|v\|_{L^{2}(\RR)}, \qquad \|\de_x\Gamma f\|_{L^2(\RR)}\lesssim \|f\|_{L^{6/5}(\RR)}.
    \]
    For the second term in \eqref{eq:estimates-for-wv-lambdaiszero-H^-1}, since $v\in L^3(\RR)$ is small and in addition $\|v\|_{L^2}\lesssim \|u\|_{\dot H^{-1,0}}$, we can simply set $V:=\Gamma(v^2-u)$ and note that
    \als{
    \|V\|_{C_0L^2\cap L^6(\RR)}&\lesssim \|u\|_{\de_xL^2(\RR)}+\|v^2\|_{L^{6/5}(\RR)}\\
    &\lesssim \|u\|_{\dot H^{-1,0}(\RR)}(\|u\|_{\Hcrit}+1)\\
    &\lesssim \|u\|_{\dot H^{-1,0}(\RR)}
    }
    for small $u$ in $\Hcrit$, where we used the estimates of Proposition \ref{prop:spacetime_estimate_heat_general} in the first inequality. Since $v=\partial_x\Gamma(v^2-u)$, we have $V_x=v$, in particular $v$ is the derivative of a function that belongs to $C_0L^2\cap L^6(\RR)$, with the desired bounds from above. Finally, the bound on the third term follows from the linear estimate \eqref{eq:z-024} in Lemma \ref{lemma:weighted-estimates-heat} with $s=0,1$, since $\demu v=\Gamma (v^2-u)$.
\end{proof}
The next Corollary proves similar estimates for solutions to \eqref{eq:miuraEquation} with constant boundary conditions at infinity. More precisely, we look for solutions $v\in L^3(\RR)+\lambda$, $\lambda\in\R$. Note that if $v$ is one such solution, then $\wv:=v-\lambda\in L^3(\RR)$ is a solution to
\begin{equation}\label{eq:wtvlambda}
    \wv_y-\wv_{xx}- 2\lambda\wv_x=(\wv^2)_x-u_x,
\end{equation}
so this Corollary is equivalently proving estimates for `tilted' solutions to equation \eqref{eq:miuraEquation}. We use the tilde above `$v$' to express the fact that we are removing the leading part (in this case, the constant $\lambda$) from a solution $v$ of \eqref{eq:miuraEquation}. We will try to keep this convention consistent throughout the article\footnote{For example, an analogous notation will be used for solutions $V$ of equation \eqref{eq:primitiveEquation}, whose leading part will be $\lambda x+\lambda^2 y$, so we will have $\wV:=V-(\lambda x+\lambda^2 y)$.}.

\begin{corollary}[Solutions of \eqref{eq:miuraEquation} close to a constant $\lambda\in\R$]\label{cor:estimates-for-wtv-part-2}
\hspace{3cm}

    {\rm(a)} Let $c\neq-2\lambda$, $u\in \dot H^{-\aha,0}(\RR)$ small and $\alpha\in C(\R_y)$ be such that $\|\alpha_y-c\|_{L^2(\Real_y)}\lesssim 1$. There exists a unique solution $v\in L^3(\RR)+\lambda$ to equation \eqref{eq:miuraEquation}. The function $\wv:=v-\lambda$ is the unique solution in $L^3(\RR)$ to equation \eqref{eq:wtvlambda}, and it satisfies the bounds
    \begin{equation*}
        \|\wv\|_{C_0\dot H^{-\aha}\cap L^3(\RR)}+\|\wv_x\|_{\Hcrit}+|c+2\lambda|^\aha\|\jap{|c+2\lambda|(x-\alpha)}^{-1}\wv\|_{L^2(\RR)}\lesssim \|u\|_\Hcrit.
    \end{equation*}
    Moreover, the estimates \eqref{eq:estimates-for-wv-lambdaiszero-L2}, \eqref{eq:estimates-for-wv-lambdaiszero-H^k}, \eqref{eq:estimates-for-wv-lambdaiszero-H^-1} hold for $\wv$ as well (with $\demu \wv_y$ replaced by $\demu \wv_y-2\lambda \wv$ and $c$ replaced by $c+2\lambda$), and the data to solution map is analytic in all the involved topologies.

    {\rm(b)} Let $u\in \dot H^{-\aha,0}(\RR)$ small and let $\lambda_1\neq \lambda_2$. The two solutions $\wv^{(1)}, \wv^{(2)}$ to equation \eqref{eq:wtvlambda} with $\lambda=\lambda_1,\lambda_2$ given by part (a) satisfy $\wv^{(1)}-\wv^{(2)}\in C_0(\R_y,L^2(\R_x))$ and
    \begin{equation*}
        \|\wv^{(1)}-\wv^{(2)}\|_{C_0 L^2\cap L^6(\RR)}\lesssim |\lambda_1-\lambda_2|^\aha \|u\|_{\Hcrit}.
    \end{equation*}
\end{corollary}
\begin{proof}
    Part (a) is a direct consequence of Lemma \ref{lemma:estimates-for-wtv} thanks to the change of coordinates $(x,y)\mapsto (x+2\lambda y,y)$, which conjugates equations \eqref{eq:miuraEquation} and \eqref{eq:wtvlambda} in $L^3(\RR)$ while keeping the $L^3$-norm and the $\Hcrit$-norm of $u$ invariant, and from the operation of subtracting a constant $\lambda$ from $v$, which also conjugates the same equations in the spaces $L^3(\RR)+\lambda$ and $L^3(\RR)$ respectively. We thus restrict our attention to part (b). By the scaling symmetry
    \[v\mapsto \lambda v(\lambda\cdot,\lambda^2\cdot), \qquad u\mapsto \lambda^2 u(\lambda\cdot,\lambda^2\cdot) \]
    and the above change of coordinates with suitable $\lambda$, we may assume $\lambda_1=1$, $\lambda_2=-1$ without loss of generality, and call $\wv^{(1)}=\wv^+$, $\wv^{(2)}=\wv^-$. Consider the function $w:=\wv^+-\wv^-$. It satisfies the equation
    \begin{equation}\label{eq:step_to_prove_regularity_gain_for_v+-v-___w}
        w_y-w_{xx}=((\wv^++\wv^-)w)_x+2(\wv^++\wv^-)_x.
    \end{equation}
    Consider the two solutions $k^\pm$ of the linear equation
    \begin{equation}\label{eq:step_to_prove_regularity_gain_for_v+-v-___kpm}
        k^\pm_y-k^\pm_{xx}=2\wv^\pm_x.
    \end{equation}
    Then the difference $z:=w-(k^++k^-)=:w-k$ satisfies
    \begin{equation}\label{eq:step_to_prove_regularity_gain_for_v+-v-___z}
        z_y-z_{xx}-((\wv^++\wv^-)z)_x=((\wv^++\wv^-)k)_x.
    \end{equation}
    Now, from Proposition \ref{prop:spacetime_estimate_heat_general} we know that
    \[
    \|(\partial_y-\partial_{xx})^{-1}\partial_x g\|_{C_0L^2\cap L^6(\RR)}\lesssim \|g\|_{L^2(\RR)}.
    \]
    Note that $\|\wv^\pm\|_{L^3(\RR)}\lesssim \|u\|_{\Hcrit}$, so $(\wv^++\wv^-)$ is small and lies in a scaling critical space of sufficient regularity. By Proposition \ref{prop:spacetime_estimate_heat_general},
    \als{
        \|(\partial_y-\partial_{xx}-\partial_x(\wv^++\wv^-)) z\|_{\partial_x L^2(\RR)}&\geq \|(\partial_y-\partial_{xx})z\|_{\partial_x L^2(\RR)}\\
        &\QQQ-\|(-\partial_x(\wv^++\wv^-))z\|_{\partial_x L^2(\RR)}\\
        &\gtrsim \|z\|_{C_0L^2\cap L^6(\RR)}-\|(\wv^++\wv^-)\|_{L^3}\|z\|_{L^6(\RR)},
        }
    so for small enough $u\in\Hcrit$, we can upgrade the previous estimate to
    \[
    \|(\partial_y-\partial_{xx}-\partial_x(\wv^++\wv^-))^{-1}\partial_x g\|_{C_0L^2\cap L^6(\RR)}\lesssim \|g\|_{L^2(\RR)}.
    \]
    By looking again at equation \eqref{eq:step_to_prove_regularity_gain_for_v+-v-___z}, it is then clear that we only need to show that $k\in C_0L^2\cap L^6$, because then the same holds for $z$ thanks to the above estimate (considering $g=(\wv^++\wv^-)k\in L^2(\RR)$), and for $w=z+k$ simply by summation. Let's thus consider $k^+$ without loss of generality. Considering $w^+:=\wv^+-k^+$, it holds
    $$ w^+_y-w^+_{xx}=((\wv^+)^2)_x-u_x.$$
    In particular, we have
    \als{
        \|k^+\|_{L^3}&\leq \|\wv^+\|_{L^3}+\|w^+\|_{L^3} \\
        &\lesssim\|(\wv^+)^2\|_{L^{3/2}}+\|u\|_{\dot H^{-\aha,0}}\\
        &\lesssim \|u\|_{\dot H^{-\aha,0}}
        }
    due to the estimate $\||\de_x|\Gamma f\|_{L^3}\lesssim \|f\|_{L^{3/2}+\Hcrit}$ of the heat kernel given by Proposition \ref{prop:spacetime_estimate_heat_general}. But we also have $k=2\partial_x\Gamma \wv^+=2T((\wv^+)^2-u)$, where
    $$ T=\partial_x\Gamma\partial_x\Gamma^-=-\aha\partial_x[\Gamma-\Gamma^-], $$
    where the last equality can be checked by means of the Fourier transform, using the notation $\Gamma^\pm:=(\partial_y-\partial_x^2\pm 2\partial_x)^{-1}$ as in Appendix \ref{appendix:heat-equation}. The above operator behaves at least as well as a heat kernel and an $x$-derivative of the heat kernel, since $\partial_x^2\Gamma$ is bounded on $L^p(\RR)$, $1<p<\infty$ (see Proposition \ref{prop:spacetime_estimate_heat_general}), so in the end $k\in C_0L^2\cap L^p(\RR)$ for any $3\leq p<\infty$, with the desired bound on the norms.
\end{proof}

\subsubsection{Decomposition in terms of simpler solutions}
Throughout the rest of the subsection, we will assume $u\in \dot H^{-\aha,0}(\RR)$ small, and consider $v^\pm\in L^3(\RR)\pm 1$ the unique solutions of \eqref{eq:miuraEquation} given by Corollary \ref{cor:estimates-for-wtv-part-2}. Recall that $\wv^\pm:=v^\pm \mp 1\in L^3(\RR)$ are solutions to \eqref{eq:wtvlambda} for $\lambda=\pm 1$, that is,
\begin{equation}\label{eq:wtvpm}
    \wv^\pm_y-\wv^\pm_{xx}\mp 2\wv^\pm_x=((\wv^\pm)^2)_x-u_x.
\end{equation}
We will write $f_\alpha(x,y):=f(x-\alpha(y),y)$ as in Notation \ref{notation:f_alpha} several times for the rest of the section.
\begin{definition}\label{def:VandG}
    Consider the two real functions defined in the introduction
    $$ \eta^\pm(x):=\aha(1\pm\tanh(x)). $$
    For $u\in\Hcrit$, let $v^\pm$ as above, and define
    \als{
    G(x,y;\alpha)&:= \eta^+(x-\alpha)v^+(x,y)+\eta^-(x-\alpha)v^-(x,y)\\
    &=\tanh(x-\alpha)+\eta^+(x-\alpha)\wv^+(x,y)+\eta^-(x-\alpha)\wv^-(x,y).
    }
    Using the subscript Notation \ref{notation:f_alpha} to denote translations in the $x$ variable, we will write
    \als{
    G(\cdot,\cdot;\alpha)&=\eta^+_\alpha v^+ +\eta^-_\alpha v^-\\
    &=\tanh_\alpha+\eta^+_\alpha \wv^+ +\eta^-_\alpha \wv^-.
    }
    Finally, through an abuse of our own Notation \ref{notation:f_alpha}, we will simply write $G_\alpha$ to denote the function $(x,y)\mapsto G(x,y;\alpha)$. Depending on the context, $\alpha$ might denote a function of the variable $y$, in which case $G_\alpha$ will have the obvious definition $G_\alpha(x,y):=G_{\alpha(y)}(x,y)$.
\end{definition}
In the following Lemma we momentarily drop the dependence of $\wv^\pm,G$ on the variable $y$ and prove some properties of $G(\cdot,y;\cdot)$ for fixed $y\in\R$.
Note that, by a straightforward computation, the quantity $G_{a}-G_{b}$ can be rewritten in terms of $a,b$ and $\wv^+-\wv^-$ only:
\begin{equation}\label{eq:z-022}
    G_{a}-G_{b}=(\tanh_{a}-\tanh_{b})+(\eta^+_{a}-\eta^+_{b})(\wv^+-\wv^-).
\end{equation}

\begin{lemma}[Properties of $G_\alpha$]\label{lemma:expDecomposition}
    Let $\wv^\pm\in H^{-1/2}(\R_x)$ such that $\wv^+-\wv^-=: h\in L^2(\R_x)$ is small enough. For $a\in\R$, let
    \[ G_a(x)=\tanh(x-a)+\eta^+(x-a)\wv^+(x)+\eta^-(x-a)\wv^-(x). \]
    \begin{enumerate}[label=\rm{(\alph*)}]
        \item It holds $G_a-G_b\in L^2(\Real;\cosh^2(x)dx)$ for any fixed $a,b\in \R$, and we have the bounds \label{item:expDec-001}
    $$ (b-a)\leq\int_\R (G_a-G_b)dx\leq 3(b-a) \quad \forall\, a,b\in\R. $$
        \item Let $v\in L^2(\Real;\cosh^2(x)dx)+G_0$. There exists a unique $\alpha\in\R$ such that the decomposition $v=\omega+G_\alpha$ satisfies \label{item:expDec-003}
        $$ \int_\R \omega(x)dx=0. $$
        \item The map \label{item:expDec-004}
        \begin{align*}
        L^2(\Real;\cosh^2(x)dx)\times L^2(\R)&\to L^2(\Real;\cosh^2(x)dx)\times\Real\\
        (z, h)&\mapsto (\omega,\alpha),
        \end{align*}
        with $\alpha$ and $\omega$ as in \ref{item:expDec-003} and $v=z+G_0$, is well-defined and smooth.
        \item If $v=\omega+G_\alpha$ as in \ref{item:expDec-003}, it holds \label{item:expDec-005}
        \[
        \|\cosh_\alpha\omega\|_{L^2(\R)}\lesssim \inf_{a\in\R}\|\cosh_a (v-G_a)\|_{L^2(\R)}
        \]
        whenever the right-hand side is small enough.
    \end{enumerate}
\end{lemma}
Note that by \ref{item:expDec-001}, for any $a\in\R$, the property $v\in L^2(\Real;\cosh^2(x)dx)+G_a$ is equivalent to $v\in L^2(\Real;\cosh^2(x)dx)+G_0$.
\begin{proof}
    We first write
    \[
    G_a-G_b=(\tanh_a-\tanh_b)+(\eta^+_a-\eta^+_b)(\wv^+-\wv^-)
    \]
    and note that the first term and the first factor in the product are bounded by $\sech^2(x)$ up to a multiplicative constant, in particular the integral is well-defined. We can compute its derivative with respect to $a$, which reads
    \als{
        \frac{d}{da}\int_\R(G_a-G_b)\,dx&=-\int_\R\aha\sech_a^2\cdot\,(2+\wv^+-\wv^-)\,dx\\
        &=-\int_\R\aha\sech_a^2\cdot\,(2+h)\,dx,
        }
    and using the smallness of $h\in L^2(\R)$, we get that this derivative lies between $-1$ and $-3$, proving \ref{item:expDec-001}. Note that being the above quantity a convolution with a Schwartz function, the map
    \[
    a\mapsto \int_\R (G_a-G_b)\,dx
    \]
    is smooth, and has a global smooth inverse by the strict sign-definiteness of the derivative. Part \ref{item:expDec-003} is a direct consequence of part \ref{item:expDec-001}, in particular of the bijectivity of the above map for fixed $b=0$.
    
    Concerning part \ref{item:expDec-004}, we first show that the map is well-defined. That is, we show that $\omega$ and $\alpha$ only depend on $z$ and $h=\wv^+-\wv^-$, and not on more information on $\wv^+,\wv^-$. We have $z=v-G_0=\omega+(G_\alpha-G_0)$, so, $\omega=z-(G_\alpha-G_0)$. The number $\alpha$ is thus the unique number such that $\int_\R z-(G_\alpha-G_0)dx=0$. By \eqref{eq:z-022}, for fixed $\alpha$ the integrand depends only on $z$ and $h$, so $\alpha$ is indeed only determined by them. The same holds for $\omega$, since $\omega=z-(G_\alpha-G_0)$, and $G_\alpha-G_0$ depends only on $\alpha$ and $h$ by \eqref{eq:z-022}. For the smoothness, first note that the map
    $$L^2(\R)\times\R\to L^2(\R;\cosh^2(x)dx),$$
    $$(h,a)\mapsto G_a-G_b$$
    is smooth, as it can be verified directly from equation \eqref{eq:z-022} (note in particular that the map is affine in the variable $h$). The dependence on $(h,a,b)$ is smooth as well since $G_a-G_b=(G_a-G_0)-(G_b-G_0)$. Now, call $z=v-G_0$, and let $v=\omega+G_\alpha$ be the unique decomposition of $v$. From what we have just shown, the map
    \[
    F:(h,\omega,\alpha)\mapsto \omega+G_\alpha-G_0
    \]
    is well-defined and smooth, where $\alpha\in\R$, $h$ ranges in $L^2(\R)$, $\omega$ ranges over the closed subspace of functions $f\in L^2(\R;\cosh^2(x)dx)$ such that $\int_\R f\;dx=0$. Moreover, the differential with respect to $(\omega,\alpha)$ is
    \[
    D_{\omega,\alpha}F(h,\omega,\alpha)\cdot(\dot\omega,\dot\alpha)=\dot\omega-\dot\alpha \sech_{\alpha}^2\cdot(1+h/2).
    \]
    It follows that $D_{\omega,\alpha}F(h,\omega,\alpha)$ is invertible on $L^2(\R;\cosh^2(x)dx)$, since $\sech^2_\alpha\cdot(1+h/2)$ has non-zero mean for small $h$, and $\dot\omega$ is a generic vector with zero mean. Part \ref{item:expDec-004} thus follows by the implicit function theorem and the uniqueness of the decomposition.

    Finally, we prove \ref{item:expDec-005}. From the smoothness statement in part \ref{item:expDec-004}, we have
    \[
    \|\cosh \omega\|_{L^2}+|\alpha| \lesssim \|\cosh \cdot (v-G_0)\|_{L^2}
    \]
    when the right-hand side is small enough. With the same smallness assumption, we can assume $|\alpha|\leq 1$, so
    \[
    \|\cosh_{\alpha} \omega\|_{L^2} \lesssim \|\cosh \cdot ( v-G_0)\|_{L^2}
    \]
    follows by monotonicity, since $\cosh(x-\alpha)\leq e^{|\alpha|}\cosh(x)$. Estimate \ref{item:expDec-005} follows from the above estimate by translation invariance.
\end{proof}

\subsubsection{Uniqueness}
We now start working for the uniqueness of solutions of \eqref{eq:miuraEquation}. The main issue is that the position of the kink depends sensibly on the source term at very negative $y$, so the difference of two generic solutions is hard to control. The first idea is to consider the new variable $\omega:=v-G_\alpha$, which is expected to decay exponentially as $x\to\pm\infty$. The transport towards the kink converts the decay in $x$ into exponential decay in $y$, after fixing $\alpha$ via a suitable orthogonality condition. This leaves room for all the needed a priori estimates.

Secondly, two solutions of \eqref{eq:miuraEquation} with the ansatz $v^j=G_{\alpha^j}+\omega^j$ will turn out to be comparable due to the fact that the quantity
\[
\int_\R v^1-v^2\,dx=:I
\]
does not depend on $x$, and in fact it can be used as a measure of the distance between the kinks of the two solutions. When $I=0$, the two solutions will share the same $\alpha$ due to a good choice of the orthogonality condition, and the argument described above proves estimates for the difference of two solutions, and shows that $v^1=v^2$. When $I\neq 0$, the distance between the two kinks will be approximately $I$ for all $y\in\R$, and the two will be different solutions. The number $I$, or equivalently the value of $\alpha$ at any fixed $y$, will thus be the real parameter that describes the family of solutions.

As a side effect of the above argument, we will see a gain of regularity in the perturbation $\omega$, essentially due to the fact that $\wv^+-\wv^-$ is more regular than $\wv^\pm$, by Corollary \ref{cor:estimates-for-wtv-part-2} part (b).\\

We first give existence, a decomposition and a priori bounds on the solution of the initial value problem with initial time $y=0$ associated to equation \eqref{eq:miuraEquation} with conditions $\pm 1$ at $\pm\infty$, and satisfying the above ansatz. To facilitate the reader, we recall that for $\alpha\in C(\R_y)$,
\als{
G_0&=G_0(x,y):=\eta^+(x)v^+(x,y)+\eta^-(x)v^-(x,y),\\
G_\alpha&=G_\alpha(x,y):=\eta^+(x-\alpha(y))v^+(x,y)+\eta^-(x-\alpha(y))v^-(x,y).
}

\begin{proposition}\label{prop:newExpEstimates}
    In the hypothesis $u\in \dot H^{-\aha,0}(\R\times(0,\infty))$ small, let $v^\pm,\wv^\pm,G$ be as in Definition \ref{def:VandG}. Let $v_0\in L^2(\R_x;\cosh^2(x)dx)+G_0(\cdot,0)$ be such that
    \[
    \|\cosh_{a_0} \cdot (v_0-G_{a_0}(\cdot,0))\|_{L^2_x}
    \]
    is small enough for some $a_0\in \Real$. There exists a unique solution $v\in C([0,\infty),H^{-\aha}(\R_x))+\tanh(x)$ of equation \eqref{eq:miuraEquation} with initial datum $v_0$ such that
    \[
    v-G_0\in C([0,\infty),L^2(\Real_x;\cosh^2(x)dx)).
    \]
    Moreover, the decomposition
    \[
    v=\omega+G_{\alpha},
    \]
    given for every $y$ by Lemma \ref{lemma:expDecomposition}, is such that $\alpha\in C([0,\infty))$ and
    $$\|\cosh_\alpha \omega\|_{L^\infty L^2\cap L^2H^1}+\|\alpha_y\|_{L^2}\lesssim \|u\|_{\dot H^{-\aha,0}}+\|\cosh_{a_0} \cdot (v_0-G_{a_0}(\cdot,0))\|_{L^2_x}.$$
\end{proposition}
Here we think of $u$ as a distribution in $\Hcrit$ supported on $\{y\geq 0\}$, so all we said so far makes sense, including the definition of $v^\pm$ and $G$. The corresponding functions $\wv^\pm$ will also be identically zero for negative $y$.\\
The decomposition is obtained by applying Lemma \ref{lemma:expDecomposition} part \ref{item:expDec-003} at each $y\in\R$. In particular, $\int \omega(x,y)\,dx$ for all $y\geq 0$, and $\alpha(y)$ is uniquely determined by this condition.
\begin{proof}
    Consider $z:=v-G_0$, and call $V_\alpha:=\eta^+_\alpha \wv^++\eta^-_\alpha \wv^-=G_\alpha-\tanh_\alpha$. The equation for $z$ reads
    \begin{align*}
        z_y-z_{xx}-2(\tanh z)_x-2(V_0 z)_x&=(z^2)_x+\aha\sech^2\!\cdot\,((\wv^+)^2-(\wv^-)^2)\\
        &\QQQ-\frac{1}{4}\left[\sech^2\!\cdot\,(\wv^+-\wv^-)^2\right]_x+2\sech^2 V_0,
    \end{align*}
    so a unique solution $z$ exists and belongs to $C([0,T],L^2(\Real_x;\cosh^2(x)dx))$ for a short time $T>0$ thanks to Lemma \ref{lemma:localWellPosednessExponentiallyLocalized}. The global existence of $z$ follows from the estimate we still need to prove and the blowup alternative. By the continuity of the decomposition map in Lemma \ref{lemma:expDecomposition}, we can assume that $v_0$, $u$ are test functions, so that the solution belongs to $C^\infty([0,T],\sech(x)H^\infty)$.
    
    We now prove the estimate on the solution. Using the equations for $\wv^\pm$ and the smoothness of the decomposition map, $\omega$ satisfies the equation
    \begin{align*}
        \omega_y-\omega_{xx}-2(\tanh_\alpha \omega)_x-2(V_\alpha \omega)_x&=(\omega^2)_x+\aha\sech^2_\alpha\cdot\,((\wv^+)^2-(\wv^-)^2)\\
        &\QQQ-\frac{1}{4}\left[\sech^2_\alpha\cdot\,(\wv^+-\wv^-)^2\right]_x\\
        &\QQQ+2\sech^2_\alpha V_\alpha+\frac{\alpha_y}{2}\sech^2_\alpha\cdot\,(2+\wv^+-\wv^-).
    \end{align*}
    The estimate on $\alpha_y$ follows integrating the equation in $x$:
    $$ |\alpha_y(y)|=\left|\frac{\int_{\R}\sech^2_\alpha\cdot\,((\wv^+)^2-(\wv^-)^2)\,dx+4\int_{\R}\sech_\alpha^2\cdot\,V_\alpha \,dx}{(\int_{\R}\sech^2_\alpha\cdot\,(2+\wv^+-\wv^-)) \,dx}\right| $$
    and the estimate $\|\alpha_y\|_{L^2_y}\lesssim \|u\|_\Hcrit$ follows from Corollary \ref{cor:estimates-for-wtv-part-2}, by part (b) and the weighted estimate of part (a), since the denominator of the above fraction is bounded away from $0$ for small $u$, uniformly in $y$.\\
    Integrating against $\cosh^2_\alpha \omega$ and integrating by parts, we get the weighted energy estimate
    \begin{align*}
        \aha \frac{d}{dy} \left[\int_{\R} (\cosh_\alpha \omega)^2 dx\right] &=-\int_{\R}(\cosh_\alpha \omega)_x^2\,dx-\int_{\R}(\cosh_\alpha \omega)^2 dx +2\int_{\R}\sech^2_\alpha\cdot\,(\cosh_\alpha \omega)^2 dx\\
        &\QQQ+ 2\int_{\R} V_\alpha (\cosh_\alpha\omega)[(\cosh_\alpha\omega)_x+\tanh_\alpha\cdot\,(\cosh_\alpha\omega)]\,dx\\
        &\QQQ-\fr23 \int_{\R} \tanh_\alpha\sech_\alpha \cdot\,(\cosh_\alpha \omega)^3 dx\\
        &\QQQ+\fr14\int_{\R}\sech_\alpha\cdot\,(\wv^+-\wv^-)^2(\cosh_\alpha \omega)_x \,dx\\
        &\QQQ+ \int_{\R}\sech_\alpha f\cdot\,(\cosh_\alpha \omega) \dx,
    \end{align*}
    where
    \[
    f=\fr12((\wv^+)^2-(\wv^-)^2)+\fr14\tanh_\alpha \cdot\,(\wv^+-\wv^-)^2+2V_\alpha+\fr{\alpha_y}{2} (2+\wv^++\wv^-). 
    \]
    Looking at the first three terms on the right hand side, we can bound the positive term with the two negative ones by making use of the orthogonality condition and the fact that the self-adjoint operator $-\partial_x^2-2\sech^2$ is non-negative on the subspace of $L^2$ functions that are orthogonal to the function $\sech$, by Lemma \ref{lemma:lowerBoundOnTheQuadraticForm}. More precisely, under the condition $\int\omega\,dx=0$, it holds
    \[ \int_{\R}(\cosh_\alpha \omega)_x^2 dx+\int_{\R}(\cosh_\alpha \omega)^2 dx-2\int_{\R}\sech^2_\alpha\cdot\,(\cosh_\alpha \omega)^2 dx \gtrsim\|\cosh_\alpha \omega\|_{H^1}. \]
    Integrating the weighted estimate in time from $0$ to $T$ thus gives
    \begin{align*}
        \|\cosh_\alpha\omega\|^2_{L^\infty L^2}+\|\cosh_\alpha\omega\|^2_{L^2H^1}&\lesssim \|\cosh_{\alpha_0}\omega_0\|^2_{L^2}+(\|\wv^+\|_{L^3L^3}+\|\wv^-\|_{L^3L^3})\|\cosh_\alpha\omega\|_{L^6L^6}\\
        &\QQQ\Qq\times(\|(\cosh_\alpha\omega)_x\|_{L^2L^2}+\|\cosh_\alpha\omega\|_{L^2L^2})\\
        &\QQQ+\|\cosh_\alpha\omega\|_{L^3L^3}^3\\
        &\QQQ+\|\cosh_\alpha\omega\|_{L^3L^3}(\|\wv^+\|_{L^3L^3}+\|\wv^-\|_{L^3L^3})^2\\
        &\QQQ+\|(\cosh_\alpha\omega)_x\|_{L^2L^2}\|\wv^+-\wv^-\|_{L^4L^4}^2+\\
        &\QQQ+(\|\sech_\alpha \wv^+\|_{L^2L^2}+\|\sech_\alpha \wv^-\|_{L^2L^2})\\
        &\QQQ\Qq\times(\|\cosh_\alpha\omega\|_{L^2L^2}+\|\alpha_y\|_{L^2_y}\|\cosh_\alpha\omega\|_{L^\infty L^2})\\
        &\QQQ +\|\alpha_y\|_{L^2_y}\|\cosh_\alpha\omega\|_{L^2L^2}.
    \end{align*}
    The cubic term can be controlled by the left hand side assuming $\omega$ remains small: by a continuity argument, the latter follows by the assumption we have on the smallness of $\|\cosh_{a_0}\cdot (v_0-G_{a_0}(\cdot,0))\|_{L^2}$, which is equivalent to the smallness of $\|(\cosh_{\alpha}\omega)|_{y=0}\|_{L^2}$ by Lemma \ref{lemma:expDecomposition} part \ref{item:expDec-005}. All the remaining terms on the right hand side are controlled by what is on the left hand side, the $L^2$-norm of $\alpha_y$, and the norm of $u$ in $\Hcrit$, thanks to the estimates of Corollary \ref{cor:estimates-for-wtv-part-2} (note in particular that by Corollary \ref{cor:estimates-for-wtv-part-2} we have the bound $\|\wv^\pm\|_{L^3L^3}+\|\wv^+-\wv^-\|_{L^6L^6}+\|\sech_\alpha \wv^\pm\|_{L^2L^2}\lesssim \|u\|_\Hcrit$), where the multiplicative constants do not depend on $T$. So the a priori estimate holds uniformly in $T>0$, thus it holds on $[0,\infty)$ as well.
\end{proof}
\begin{lemma}\label{lemma:uniqueness_1}
    Let $v^1$, $v^2$ be two solutions of \eqref{eq:miuraEquation} with initial data $v_0^1$, $v_0^2$ as in Proposition \ref{prop:newExpEstimates}. Then the quantity
	$$ \int_{\R}(v^2-v^1)dx=:I  $$
	is well defined, finite and does not depend on $y$. Moreover, if $v^j=\omega^j+G_{\alpha^j}$ as in Lemma \ref{lemma:expDecomposition} part \ref{item:expDec-003}, we have that $|I|\leq|\alpha^1-\alpha^2|\leq 3|I|$ for all $y\in\R$.
\end{lemma}
\begin{proof}
    We can assume the data $u$ and $v_0^j-G_0(\cdot,0)$ to be test functions as in the proof of the previous Proposition. It holds $w:=v^2-v^1=\omega^2-\omega^1+(G_{\alpha^2}-G_{\alpha^1})$. The finiteness of the above integral is given by the estimates in the previous Proposition and Lemma \ref{lemma:expDecomposition} part \ref{item:expDec-001}. The quantity $w$ satisfies the equation
    \begin{equation}\label{eq:z-002}
        w_y-w_{xx}=((v^1+v^2)w)_x,
    \end{equation}
    and the independence of $y$ of the integral in the statement follows by integrating the equation in $x$. The last statement follows by Lemma \ref{lemma:expDecomposition} part \ref{item:expDec-001}, since
    \[
    \int w\,dx=\int (G_{\alpha^2}-G_{\alpha^1})\,dx.\qedhere
    \]
\end{proof}

\begin{lemma}\label{lemma:uniqueness_2}
	In the hypotheses of the previous Lemma, calling $w:=v^2-v^1$, assuming $I=0$, we have the estimate
	$$ \|e^{\eps_0 y}\cosh_\alpha w\|_{L^\infty L^2\cap L^2H^1}\lesssim \|\cosh_{a_0} w(\cdot,0)\|_{L^2}, $$
	where $\alpha=\alpha^1=\alpha^2$ is shared by both solutions thanks to the previous Lemma, $a_0\in\R$, and where $\eps_0>0$ is a universal constant.
\end{lemma}

\begin{proof}
    We perform the same approximation argument as in the previous Lemma. We follow the same strategy as in the proof of Proposition \ref{prop:newExpEstimates}, and we define again $V_\alpha:=\eta^+_\alpha\wv^++\eta^-_\alpha\wv^-$. Integrating the equation \eqref{eq:z-002} against $\cosh^2_\alpha w$ and multiplying by $e^{2\eps y}$, calling $z:=e^{\eps y}w$, we get
	\begin{align*}
		\aha\fr{d}{dy}\left[\int_{\R}(\cosh_\alpha z)^2 dx\right]&=-\int_{\R}(\cosh_\alpha z)_x^2 dx-(1-\eps)\int_{\R}(\cosh_\alpha z)^2 dx\\
            &\QQQ+2\int_{\R}\sech_\alpha^2\cdot\,(\cosh_\alpha z)^2 dx\\
            &\QQQ-\int_{\R}(2V_\alpha+\omega^1+\omega^2)(\cosh_\alpha z)(\cosh_\alpha z)_x \,dx\\
            &\QQQ-\int_{\R}(2V_\alpha+\omega^1+\omega^2)\tanh_\alpha\cdot\,(\cosh_\alpha z)^2 dx\\
            &\QQQ-\alpha_y\int_{\R}\tanh_\alpha\cdot\,(\cosh_\alpha z)^2 dx.
	\end{align*}
    The third term on the right-hand side is controlled by the first two terms thanks to the orthogonality condition $\int w\,dx=0$ as we already discussed in the proof of Proposition \ref{prop:newExpEstimates}: the only difference is the factor $(1-\eps)$ in front of the second term, which does not change the argument assuming $\eps$ is smaller than a suitable $\eps_0>0$. For the remaining terms, we proceed as in Proposition \ref{prop:newExpEstimates}, so we give short details. It is enough to estimate $\wt v^\pm,\omega^j$ in $L^3(\RR)$, $\tanh_\alpha\alpha_y$ in $L^2_yL^\infty_x$, and $\alpha_y$ in $L^2_y$, since we will have bounds for $\cosh_\alpha z$ in $L^2_{x,y}\cap L^6_{x,y}\cap L^4_yL^2_x$ and $(\cosh_\alpha z)_x$ in $L^2_{x,y}$. The former quantities are controlled by the $\Hcrit$-norm of $u$, thanks to the estimates we have on $v^\pm$ in Corollary \ref{cor:estimates-for-wtv-part-2}, and to the estimate for $\omega^j$ of Proposition \ref{prop:newExpEstimates}. After integrating in $y$, and substituting $z$ with $e^{\eps y}w$, we thus obtain the desired inequality.
\end{proof}

For small $u\in\Hcrit$ and $v$ being a solution of \eqref{eq:miuraEquation}, we consider the following properties:
\begin{gather}
    v-G_0\in C(\R_y,L^2(\R_x;\cosh^2(x)dx)), \label{eq:v_minus_G_0_is_in_exponentially_weighted_L^2}\tag{H1} \\
    \lim_{y\to-\infty}\inf_{\gamma\in\R}\|\cosh_\gamma \cdot\,(v(\cdot,y)-G_{\gamma}(\cdot,y))\|_{L^2(\R)}=0, \label{eq:v_minus_G_0_is_small_at_minus_infinity}\tag{H2}\\
    \lim_{y\to+\infty}\inf_{\gamma\in\R}\|\cosh_\gamma \cdot\,(v(\cdot,y)-G_{\gamma}(\cdot,y))\|_{L^2(\R)}=0, \label{eq:v_minus_G_0_is_small_at_plus_infinity}\tag{H3}\\
    v=w+\tanh_\sigma, \quad \exists w\in L^3(\RR),\,\sigma_y\in L^2(\R_y).\label{eq:v_is_in_L^3_plus_tanh_sigma}\tag{H4}
\end{gather}
with $G$ as in Definition \ref{def:VandG}.

\begin{corollary}\label{cor:uniquenessOfMiura}
    Let $u\!\in\! \dot H^{-\aha,0}(\RR)$ small. Let $v$ be a solution of \eqref{eq:miuraEquation} on $\RR$ satisfying \eqref{eq:v_minus_G_0_is_in_exponentially_weighted_L^2}, \eqref{eq:v_minus_G_0_is_small_at_minus_infinity}.
    Consider the unique decomposition $v=\omega+G_{\alpha}$ with $\int \omega\,dx=0$ $\forall y\in\Real$ given by Lemma \ref{lemma:expDecomposition}. Then, $v$ satisfies \eqref{eq:v_minus_G_0_is_small_at_plus_infinity}, \eqref{eq:v_is_in_L^3_plus_tanh_sigma}, and it holds
        \[ \|\cosh_{\alpha^j} \omega\|_{L^\infty L^2\cap L^2H^1}+\|\alpha_y\|_{L^2}\lesssim\|u\|_{\dot H^{-\aha,0}(\RR)}.\]
    Moreover, if $v^1$, $v^2$ are two such solutions, with decompositions $v^j=\omega^j+G_{\alpha^j}$, then either $v^1\equiv v^2$ or $\alpha^1(y)-\alpha^2(y)\neq 0$ $\forall y\in\Real$.
\end{corollary}
\begin{proof}
    By Lemma \ref{lemma:expDecomposition} part \ref{item:expDec-005}, the assumption \eqref{eq:v_minus_G_0_is_small_at_minus_infinity} is equivalent to
    \[
    \| (\cosh_{\alpha^j} \omega^j)|_{y=y_0}\|_{L^2(\R)}\to 0\quad\text{ as }y_0\to-\infty.
    \]
    The a priori estimate follows by Proposition \ref{prop:newExpEstimates} after a translation in the $y$ variable, truncating first the solutions on $\R_x\times(y_0,\infty)$, and then sending the parameter $y_0$ to $-\infty$. Using Lemma \ref{lemma:expDecomposition} part \ref{item:expDec-005}, the assumption \eqref{eq:v_minus_G_0_is_small_at_plus_infinity} can be verified from the same a priori estimate by sending $y_0$ to $+\infty$ along a suitable sequence, since $\cosh_\alpha \omega \in L_y^\infty L^2_x\cap L^2_yL^2_x$. Assumption \eqref{eq:v_is_in_L^3_plus_tanh_sigma} follows from the a priori estimate of this Corollary on $\omega$, $\alpha$, the definition of $G$, and the fact that $\wv^\pm\in L^3(\RR)$.

    For the last statement, assume $\alpha^1-\alpha^2=0$ for some $y$. Then by Lemma \ref{lemma:uniqueness_1} it holds $\alpha^1\equiv\alpha^2=:\alpha\equiv 0$, $I=0$, and by Lemma \ref{lemma:uniqueness_2} it holds
    $$ \|e^{\eps_0 (y-y_0)}\cosh_\alpha w\|_{L^\infty_{(y_0,\infty)} L^2\cap L^2_{(y_0,\infty)}H^1}\lesssim \|\cosh_{\alpha(y_0)} w(\cdot,y_0)\|_{L^2}. $$
    Sending $y_0$ to $-\infty$ shows that $w=0$, and the claim follows.
\end{proof}

\begin{proposition}[Uniqueness of solutions of \eqref{eq:miuraEquation}]\label{prop:unconditional_uniqueness_miura}
    Let $u\in \Hcrit$ small enough. Let $v$ be a solution to \eqref{eq:miuraEquation} satisfying \eqref{eq:v_is_in_L^3_plus_tanh_sigma}. Then, $v$ satisfies \eqref{eq:v_minus_G_0_is_in_exponentially_weighted_L^2}, \eqref{eq:v_minus_G_0_is_small_at_minus_infinity}, \eqref{eq:v_minus_G_0_is_small_at_plus_infinity}. In particular, by Corollary \ref{cor:uniquenessOfMiura}, $v$ is uniquely determined by $u$ and by the value of $\alpha$ at any fixed $y$, where $v=\omega+G_\alpha$ is the decomposition as in Lemma \ref{lemma:expDecomposition}.
\end{proposition}
\noindent Lemma \ref{lemma:change-of-variables-c-alpha_0-gamma_0} will later tell us that $v=\eleV^{(-1,1)}(u,c)$ (see Proposition \ref{prop:kink-addition_map_eleV}) for a suitable $c\in \R$ that is uniquely determined by $u$ and $\alpha(0)$.

\begin{proof}
    Consider $v^\pm\in L^3(\RR)\pm 1$ as in Definition \ref{def:VandG}, and $\wv^\pm:=v^\pm\mp 1$. For $w$ as in the statement of the Proposition, after some rewriting, it holds
    \[
    w_y-w_{xx}-2w_x=-4(\eta^-_\sigma w)_x+(w^2)_x-u_x+\sigma_y\cdot \sech_\sigma^2.
    \]
    The function $z:=w-\wv^+$ solves
    \[
    z_y-z_{xx}-2z_x=((w+\wv^+)z)_x-4(\eta^-_\sigma w)_x+\sigma_y\cdot \sech^2_\sigma.
    \]
    The function $\omega(x,y):=e^{x-\sigma(y)}z(x,y)$ satisfies
    \begin{equation}\label{eq:z-023}
    \omega_y-\omega_{xx}+\omega=(\de_x-1)[(w+\wv^+)\omega-2\sech_\sigma w]+\sigma_y e^{x-\sigma}\sech^2_\sigma-\sigma_y\,\omega.
    \end{equation}
    The operator $\GammaExp:=(\de_y-\de_x^2+1)^{-1}$ satisfies all the estimates satisfied by $\Gamma$ stated in Proposition \ref{prop:spacetime_estimate_heat_general} with similar proofs, with in addition
    \[
    \|\GammaExp f\|_{L^pL^q}\lesssim \|f\|_{L^pL^q}
    \]
    due to the fact that the kernel $\GammaExp(x,y)=e^{-y}\Gamma(x,y)$ belongs to $L^1(\RR)$. We thus have the estimate
    \als{
    \|\omega\|_{C_0L^2\cap L^6_{x,y}}&\lesssim \|(w+\wv^+)\omega\|_{L^2_{x,y}}+\|\sech_\sigma w\|_{L^3_yL^2_x}\\
    &\QQQ+\|\sigma_y e^{x-\sigma}\sech^2_\sigma\|_{L^2_{x,y}}+\|\sigma_y\,\omega\|_{L^2_{x,y}}\\
    &\lesssim (\|\wv^+\|_{L^3_{x,y}}+\|w\|_{L^3_{x,y}}+\|\sigma_y\|_{L^2_y})\|\omega\|_{C_0L^2\cap L^6_{x,y}}\\
    &\QQQ+\|w\|_{L^3_{x,y}}+\|\sigma_y\|_{L^2_y},
}
    The estimate holds on any half plane $\R_x\times (-\infty,y_0]$, $y_0\in \R_y$, since the kernel $\GammaExp$ is identically zero for negative $y$. In particular, choosing $y_0$ to be negative enough so that all the quantities multiplying $\|\omega\|$ on the right hand side are small enough, it holds
    \[
    \|\omega\|_{C_0L^2\cap L^6_{x,y}}\lesssim \|w\|_{L^3_{x,y}+L^6_{x,y}}+\|\sigma_y\|_{L^2_y}
    \]
    on that given half plane, provided that the left hand side is finite. By a density argument and the uniqueness of solutions for the initial value problem associated to the heat equation, the finiteness condition of the left hand side can be removed assuming $z\in L^3(\RR)$, so for our original $z$ we have
    \[
    \|e^{x-\sigma(y)}z\|_{C_0L^2\cap L^6_{x,y}}\lesssim \|w\|_{L^3_{x,y}+L^6_{x,y}}+\|\sigma_y\|_{L^2_y}. \]
    Thus, $e^{x-\sigma(y)}(w-\wv^+)\in C_0((-\infty,y_0],L^2(\R_x))\cap L^6(\R\times (-\infty,y_0))$. An analogous energy estimate
    \als{
     \|\omega\|_{L^\infty_IL^2\cap L^6_IL^6}&\lesssim \|\omega|_{y=y_0}\|_{L^2(\R)}+(\|\wv^+\|_{L^3_IL^3_x}+\|w\|_{L^3_IL^3_x}+\|\sigma_y\|_{L^2_I})\|\omega\|_{L^\infty_IL^2\cap L^6_IL^6_x}\\
    &\QQQ+\|w\|_{L^3_IL^3_x}+\|\sigma_y\|_{L^2_I}
    }
    holds on finite strips of the form $\R\times I$, $I=(y_0,y_1)$. The above argument applied iteratively on finitely many small intervals $I$ shows that 
    \[
    e^{x-\sigma(y)}(w-\wv^+)\in C(\R_y,L^2(\R_x))\cap L^6(\RR),
    \]
    with the $L^2$-norm going to zero as $y\to-\infty$. The same holds for $e^{-(x-\sigma(y))}(w-\wv^+)$, with an identical proof. Now note that it holds
    \als{
    v-G_\sigma&=w-\eta_\sigma^+\wv^+-\eta_\sigma^-\wv^-\\
    &=\eta_\sigma^+\cdot (w-\wv^+)+\eta_\sigma^-\cdot (w-\wv^-),
    }
    and $\cosh_\sigma\eta^\pm_\sigma\leq e^{\pm(x-\sigma(y))}$, so by the two estimates we proved,
    \[
    \cosh_\sigma\cdot (v-G_\sigma)\in C(\R_y,L^2(\R_x))\cap L^6(\RR),
    \]
    with vanishing $L^2$-norm as $y\to-\infty$. This directly implies \eqref{eq:v_minus_G_0_is_in_exponentially_weighted_L^2}, and by Lemma \ref{lemma:expDecomposition} part \ref{item:expDec-005}, condition \eqref{eq:v_minus_G_0_is_small_at_minus_infinity} follows as well. The rest of the statement follows from Corollary \ref{cor:uniquenessOfMiura}.
\end{proof}

Following what we do in Appendix \ref{sec:existence_of_solutions_with_L2_data}, we could now run a compactness argument to obtain the existence of eternal solutions with the above bounds (this will work due to the a priori estimates we have on $v$). It turns out that we can avoid this, since we are able to find explicit formulas for the solutions in terms of $v^\pm$.

\subsection{Exact formula for solutions of \texorpdfstring{$($\ref{eq:miuraEquation}$)$}{(M)}}\label{subsec:exact_solutions}
The main idea to construct solutions of \eqref{eq:miuraEquation} is simple. Burgers' equation is formally equivalent to a heat equation with potential, through the Cole--Hopf transformation. So there is a way of taking superpositions of different solutions by the linearity of the heat equation. In particular, we are allowed to interpolate between $v^+$ and $v^-$, which can be seen as the limit of the solution family as $\alpha\to\pm\infty$: note, in fact, that $v^+$ and $v^-$ satisfy equation \eqref{eq:miuraEquation} with boundary conditions $\pm 1$ as $(x,y)\to\infty$. We make things slightly more general by considering an arbitrary combination of solutions of \eqref{eq:miuraEquation} from Corollary \ref{cor:estimates-for-wtv-part-2}: this is the content of Proposition \ref{prop:kink-addition_map_eleV}, which we prove at the end of this subsection. We will show in the next subsection how the solutions constructed in this way are the ones we are looking for (in particular, they satisfy the assumptions of Proposition \ref{prop:unconditional_uniqueness_miura}).

One conceptual difficulty is the fact that the derivative in the Cole--Hopf transformation is not a bijective operation, and its inverse introduces undetermined $y$-dependent constants. However, there is a natural way of defining the antiderivative of a solution of \eqref{eq:miuraEquation} up to a constant that does not depend on $x$ nor on $y$. To see this, we consider the natural primitive equation of \eqref{eq:miuraEquation},
\begin{equation}\label{eq:primitiveEquation}
    V_y-V_{xx}=V_x^2-u.
\end{equation}
\begin{lemma}\label{lemma:existenceOfPrimitiveSolution}
    Let $v\in L^2_\loc(\RR)$ be a distributional solution to equation \eqref{eq:miuraEquation}. Then, there exists a solution $V$ to equation \eqref{eq:primitiveEquation} such that $\partial_x V=v$. Moreover, $V$ is unique up to an additive constant.
\end{lemma}
\begin{proof}
    Assume $v\in L^2_\loc(\RR)$ is a solution to equation \eqref{eq:miuraEquation}. By Lemma \ref{lemma:distributional-antiderivative}, there exists $\wV\in \mathscr D'(\RR)$ that satisfies $\partial_x\wV=v$. It follows that
    \[
    \wV_y-\wV_{xx}=\wV_x^2-u-g,
    \]
    where $g\in \Dscr(\R_y)$ (that is, $g\in \Dscr(\RR)$ such that $\partial_x g=0$). Consider $G\in \Dscr(\R_y)$ a primitive of $g$ which is still independent on $y$ and define $V:=\wV+G$. Then, it is immediate to see that $V$ satisfies \eqref{eq:primitiveEquation} and $\partial_x V=v$. Moreover, assume $T$ has the same properties of $V$. Then, calling $W:=T-V$,
    \begin{gather*}
        W_x=v-v=0,\\
        W_y=T_y-V_y=(v_x+v^2-u)-(v_x+v^2-u)=0,
    \end{gather*}
    so that $W$ is constant.
\end{proof}

\subsubsection{Estimates in parabolic BMO spaces}
By scaling considerations, and from the fact that we have uniqueness only up to an additive constant, we can guess that the right space for $V$ is some space of BMO type, but we need to take into account the scaling symmetry of equation \eqref{eq:miuraEquation}, which is parabolic and incompatible with the usual, Euclidean $\BMO(\RR)$. We give the definition of BMO spaces with parabolic metric, after a brief, self-contained overview of the general theory of metric measure spaces of homogeneous type. We refer in particular to Coifman--Weiss \cite{coifman-weiss-1977-extension-hardy-spaces}, and to the other classical references \cites{stein-1970-singular-integrals,stein-1993-harmonic-analysis}.
\begin{definition}[Spaces of homogeneous type and BMO, {\cite{coifman-weiss-1977-extension-hardy-spaces}}]\label{def:spaces_of_homogeneous_type_and_BMO}
    Let $(X,d)$ be a metric\footnote{One can more generally consider quasi-metrics on $X$, where the triangle inequality only holds up to an absolute multiplicative constant. We will always assume that the metric is finite.} space, with open balls denoted by $\Bcal_r(x)$, $r\geq 0$, $x\in X$, and $\mu$ be a Borel measure on $X$ such that $\mu(\Bcal_r(x))>0$ for any $r>0$, $x\in X$. The triple $(X,d,\mu)$ is a \emph{space of homogeneous type} if the measure $\mu$ is \emph{doubling}, i.e., there exists $A>0$ such that
    \[
    \mu(\Bcal_r(x))\leq A\mu(\Bcal_{r/2}(x)),\quad\forall \,r>0,\,x\in X.
    \]
    We define $\BMO^0(X,d,\mu)$, or simply $\BMO^0(X)$, as the space of all (classes of a.e. equal) measurable functions such that the seminorm
    \[ \|f\|_{\BMO}:= \sup_\Bcal \fint_\Bcal |f-f_\Bcal|d\mu \]
    is finite, where the supremum is taken over all balls $\Bcal$ with respect to the distance $d$, and where $f_\Bcal:=\fint_\Bcal f d\mu:=\frac{1}{\mu(\Bcal)}\int_\Bcal f d\mu$. We then define\footnote{When $\mu(X)<\infty$, it is common to define $\BMO(X)$ simply as $\BMO^0(X)$, and to set
    \[
    \|f\|_{\BMO}= \sup_\Bcal \fint_\Bcal |f-f_\Bcal|d\mu+\left|\fint_X f\,d\mu\right|.
    \]
    We do not adopt this distinction here for the sake of uniformity. This has minor importance, since the applications of this theory in this article will only concern the case $\mu(X)=\infty$.
    } $\BMO(X):=\BMO^0(X)/\R$ (note that $\|\cdot\|_{\BMO}$ is well-defined on $\BMO(X)$).
\end{definition}
The above is simply the usual definition of the space $\BMO$ in the context of metric measure spaces of homogeneous type (see the discussion in \cite{coifman-weiss-1977-extension-hardy-spaces}*{§2}). From the classical theory, we know that the space $(\BMO(X),\|\cdot\|_{\BMO})$ is a Banach space, and it is immediate to verify that $\BMO^0(X)$ is a Banach space as well when equipped with the norm
\[
\|f\|_{\BMO^0_\phi}:=\|f\|_{\BMO}+\left|\int_X f\phi\,d\mu\right|,
\]
for some $\phi\in L^\infty_c(X)$ such that $\int_X \phi \,d\mu \neq 0$ (all such norms are equivalent), and embeds continuously into\footnote{From here we assume that closed balls in $X$ are compact, so that $L^p_\loc(X)$ is a Fr\'echet space with all the usual properties. This will hold for the parabolic BMO spaces. It holds in general when $X$ is complete, see \cite{coifman-weiss-1977-extension-hardy-spaces}*{§4, second paragraph}.} $L^1_\loc(X)$. From here on, for a function $f\in \BMO^0(X)$, we will denote with the same name the class of functions $[f]_\R=\{f+c\,|\,c\in\R\}\in \BMO(X)$ when there is no ambiguity.

One property of BMO functions in spaces of homogeneous type is the John--Nirenberg Lemma. The proof is analogous to that of the Euclidean case, which was originally proved in \cite{john-nirenberg-1961-on-functions-of-BMO} (see also \cite{coifman-weiss-1977-extension-hardy-spaces}*{footnote 22}).
\begin{lemma}[John--Nirenberg inequality]
    Let $(X,d,\mu)$ be a space of homogeneous type. There exist constants $C,\gamma>0$, depending on the doubling constant $A$, such that for every ball $\Bcal$ it holds
    \[
    \fint_\Bcal e^{\frac{\gamma}{\|f\|_{\BMO}}|f(x)-f_\Bcal|}d\mu\leq C.
    \]
\end{lemma}

\begin{corollary}\label{cor:exponential-integrability-of-BMO-functions}
    It holds $\BMO^0(X)\hookrightarrow L^p_\loc(X)$ for all $p<\infty$. Moreover, $e^f\in L^p_\loc(X)$ whenever $p\leq \gamma /\|f\|_{\BMO}$.
\end{corollary}
We state the classical logarithmic-growth bound for BMO functions, the proof of which can be found in \cite{grafakos-2014-modern-fourier-analysis}*{§3.1} for the Euclidean case, as a guided exercise.
\begin{lemma}\label{lemma:BMO-logarithmic-growth}
    Let $(X,d,\mu)$ be a space of homogeneous type. The following inequality holds for functions in $\BMO(X)$:
    \[
    |f_{\Bcal_{r}(x_1)}-f_{\Bcal_{r}(x_2)}|\lesssim \log\left(2+\frac{d(x_1,x_2)}{r}\right)\|f\|_{\BMO}.
    \]
\end{lemma}
We now apply the general theory above to a very special case, considering \emph{parabolic BMO spaces} as follows.
\begin{definition}[The space $\PBMOx\lambda(\RR)$]\label{def:parabolic_BMO}
    Define the parabolic norm\footnote{We will call it norm, although it is homogeneous only with respect to a parabolic rescaling of coordinates.} on $\RR$
    \[
    |(x,y)|_{\rm p,\lambda}:=\max\{|x+2\lambda y|,|y|^{\aha}\}
    \]
    and the corresponding \emph{parabolic metric} as
    \[ d_{\rm p,\lambda}(p,q):=|p-q|_{\rm p,\lambda}. \]
    We denote by $\PBMOx\lambda(\RR):=\BMO(\RR,d_{\rm p,\lambda},\mu)$ the BMO space with respect to the space of homogeneous type $(\RR,d_{\rm p,\lambda},\mu)$, where $\mu$ is the Lebesgue measure. We will denote by $\PBMOx\lambda^0(\RR)$ the Banach space $\BMO^0(\RR,d_{\rm p,\lambda},\mu)\hookrightarrow \mathscr S'(\RR)$, and call $\PBMO(\RR):=\PBMOo(\RR)$.
\end{definition}

The space $(\R^2,d_{\rm p,\lambda},\mu)$ is indeed a space of homogeneous type ($d_{\rm p,\lambda}$ is a metric) with doubling constant $A=8$, since $\mu(\Bcal_r(x,y))=r^3$. The distance $d_{\rm p,\lambda}$ is the `tilted' version of the natural parabolic metric $d_{\rm p}=d_{\rm p,0}$, that is, $d_{\rm p,\lambda}$ is the composition between $d_{\rm p,0}$ and the linear transformation $(x,y)\mapsto (x+2\lambda y,y)$.

We use the notation $\Gamma^{(c)}:=(\de_y-\de_x^2+c\de_x)^{-1}$ as in Appendix \ref{appendix:heat-equation}.

\begin{lemma}
    The operator $\Gamma^{(-2\lambda)}$ is bounded from $L^{3/2}(\RR)+\dot H^{-\aha,0}(\RR)$ to $\PBMOl(\RR)$, with uniform constants in $\lambda$.
\end{lemma}
By the change of coordinates $(x,y)\mapsto(x-cy,y)$, we can assume $\lambda=0$. The estimate follows from the estimates of Lemma \ref{lemma:heat_kernel_and_Hardy_space_H1} by the $\Hcal^1$-BMO duality. Alternatively, a direct proof involves linear estimates for the operators $(\partial_y-\partial_x^2)^{-1}$, $|\partial_x|^\aha(\partial_y-\partial_x^2)^{-1}$, and is analogous to the boundedness of the Riesz potential operator $(-\Delta)^{-\frac{s}{2}}$, $s=\frac{n}{p}$, from $L^p(\R^n)$ to $\BMO(\R^n)$ in the Euclidean case, often referred to as `endpoint Sobolev embedding', see \cite{stein-1970-singular-integrals}*{V, §6.17}.
\begin{corollary}\label{cor:BMO-bounds-for-Vlambda}
    Let $\lambda\in\R$. Let $v\in L^3(\RR)+\lambda$ be the unique solution to \eqref{eq:miuraEquation} given by Corollary \ref{cor:estimates-for-wtv-part-2}, and let $V$ the solution of \eqref{eq:primitiveEquation}, $V_x=v$, defined up to a constant by Lemma \ref{lemma:existenceOfPrimitiveSolution}. It holds the bound
    \begin{equation}\label{eq:BMO-bounds-for-Vlambda}
        \|\wV\|_{\PBMOl(\RR)}\lesssim \|u\|_{\dot H^{-\aha,0}(\RR)},
    \end{equation}
    where $\wV:=V-(\lambda x+\lambda^2 y)$. Moreover, if $u\in\Hcrit$ is small enough, then $\psi:=e^V,1/\psi\in L^6_\loc(\RR)$, $\partial_x\psi=v\psi\in L^2_\loc(\RR)$, and it holds
    \begin{equation}\label{eq:psi_is_in_the_kernel_of_Lax_operator}
        (\partial_y-\partial_{xx}+u)\psi=0.
    \end{equation}
\end{corollary}
Estimate \eqref{eq:BMO-bounds-for-Vlambda} is an extension of Corollary \ref{cor:estimates-for-wtv-part-2} part (a). Note that the product $u\psi$ is well defined in $L^1(\R_y,B^{-\aha}_{1,2\;\loc}(\R_x))\subset B^{-1/2,0}_{1,2\;\loc}(\RR)$ with the above regularity hypotheses on $\psi$.
\begin{proof}
    Let $\wv:=v-\lambda$ as in Corollary \ref{cor:estimates-for-wtv-part-2}. From equation \eqref{eq:primitiveEquation}, the function $\wV:=V-(\lambda x+\lambda^2 y)$ satisfies
    \[ \wV_y-\wV_{xx}-2\lambda \wV_x=(\wV_x)^2-u, \]
    \[ \wV_x=\wv, \]
    so the bound \eqref{eq:BMO-bounds-for-Vlambda} follows from the linear estimates of the previous Lemma and from the bound $\|\wv\|_{L^3(\RR)}\lesssim \|u\|_{\Hcrit}$ of Corollary \ref{cor:estimates-for-wtv-part-2}. Since $V-(\lambda x+\lambda^2 y)\in \PBMOl(\RR)$, by Corollary \ref{cor:exponential-integrability-of-BMO-functions} we have $\psi=e^V,1/\psi=e^{-V}\in L^6_\loc(\RR)$ if $\|u\|_{\Hcrit}$ is small enough, and it is easy to verify that $\partial_x\psi=v\psi\in L^2_\loc(\RR)$ thanks to the $L^3$-bound on $\wv$.
\end{proof}
The next Lemma shows that the difference of two solutions of \eqref{eq:primitiveEquation} is more regular than BMO, in analogy with Corollary \ref{cor:estimates-for-wtv-part-2}, part (b). To state it, we define the semi-normed space
    \[
    C^{0,\alpha}_\unif(\RR):=\left\{f\in C(\RR)\,|\,|f|_{C^{0,\alpha}_\unif}:=\sup_{u,v\in\RR,\,|u-v|\leq 1}\frac{|f(u)-f(v)|}{|u-v|^{\alpha}}<\infty\right\},
    \]
where the function $|\cdot|$ denotes the Euclidean norm on $\RR$. Clearly, the space $\faktor{C^{0,\alpha}_\unif(\RR)}{\R}$ is a Banach space when equipped with $|\cdot|_{C^{0,\alpha}_\unif}$.
\begin{lemma}\label{lemma:holder_regularity_for_difference_of_primitive_solutions}
    Let $u\in\Hcrit$ small, and $\lambda_1<\lambda_2$. Consider the solutions $v^{(j)}\in L^3(\RR)+\lambda_j$ of \eqref{eq:miuraEquation} given by Corollary \ref{cor:estimates-for-wtv-part-2}. Let $V^{(j)}$ as in Corollary \ref{cor:BMO-bounds-for-Vlambda}, and $\wV^{(j)}=V^{(j)}-(\lambda_jx+\lambda_j^2y)$. It holds
    \[
    |\wV^{(2)}-\wV^{(1)}|_{C^{0,1/4}_\unif(\RR)}\lesssim_{\lambda_1,\lambda_2} \|u\|_{\Hcrit}
    \]
    and the map $u\mapsto \wV^{(2)}-\wV^{(1)}$ is analytic with values in $\faktor{C^{0,\alpha}_\unif(\RR)}{\R}$.
\end{lemma}

\begin{proof}
    Set $\wv^{(j)}=v^{(j)}-\lambda_j$. We consider the special case $\lambda_1=1$ and $\lambda_2=-1$ without loss of generality as in the proof of part (b) of Corollary \ref{cor:estimates-for-wtv-part-2}, and we use the same notation $\wv^\pm$ of that proof. The analyticity of the map follows from the Banach fixed point theorem as in Lemma \ref{lemma:estimates-for-wtv}, so we focus on the bound. The functions $\wV^\pm$, given up to additive constants by Lemma \ref{lemma:existenceOfPrimitiveSolution}, satisfy
    \[
    \left\{\begin{aligned}
        & \wV^\pm_y-\wV^\pm_{xx}=\pm 2\wV^\pm_x+(\wV_x^\pm)^2-u,\\
        & \partial_x \wV^\pm=\wv^\pm.
    \end{aligned}\right.
    \]
    By subtracting the two equations, calling $W=\wV^+-\wV^-$, and noting that $\wv^\pm=\Gamma\de_x\Gamma^\mp ((\wv^\pm)^2-u)$, one has
    \als{
        W&=\Gamma((\wv^++\wv^-)(W_x+2))\\
        &=\Gamma((\wv^++\wv^-)W_x)+2S^-((\wv^+)^2-u)+2S^+((\wv^-)^2-u),
    }
    where
    \[
        S^\pm:=\Gamma\de_x\Gamma^\pm=-\fr12 [\Gamma-\Gamma^\pm],
    \]
    where the last identity is checked via the Fourier transform. The heat operator $\Gamma$ extends to a map
    \[
    \Gamma\colon L^2(\RR)\to \faktor{ C^{0,1/2}(\RR,d_{\rm p})}{\R},
    \]
    which is well-defined and bounded by Lemma \ref{lemma:heat_kernel_sends_Lp_to_Hoelder}, where $\R\subset C^{0,1/2}(\RR,d_{\rm p})$ is the subspace of constant functions. By Lemma \ref{cor:estimates-for-wtv-part-2}, $W_x=\wv^+-\wv^-\in L^6(\RR)$, so $(\wv^++\wv^-)W_x\in L^2(\RR)$. In particular,
    \[
    |\Gamma((\wv^++\wv^-)W_x)|_{\dot C^{0,1/2}(\RR,d_{\rm p})}\lesssim \|u\|_{\Hcrit}.
    \]
    For the remaining two terms, we use the fact that $S^\pm$ can be written in two different ways, as above. Combining Proposition \ref{prop:spacetime_estimate_heat_general} and Lemma \ref{lemma:heat_kernel_sends_Lp_to_Hoelder}, from the two above ways of writing $S^\pm$, we get respectively that
    \als{
    S^\pm &: L^{6/5}(\RR)\cap \partial_x L^2(\RR)\to C^{0,1/2}(\RR,d_{\rm p})\cap C^{0,1/2}(\RR,d_{\rm p,\mp 1}),\\
    S^\pm &
    : L^2(\RR)\to C^{0,1/2}(\RR,d_{\rm p})+ C^{0,1/2}(\RR,d_{\rm p,\mp 1}).
    }
    By interpolation,
    \[
    S^\pm : L^{3/2}(\RR)\cap \Hcrit\to C^{0,1/2}(\RR,d_{\rm p})+ C^{0,1/2}(\RR,d_{\rm p,\mp 1}).
    \]
    Putting together the above bounds, we have obtained that
    \[
    W\in C^{0,1/2}(\RR,d_{\rm p}) + C^{0,1/2}(\RR,d_{\rm p, 1}) + C^{0,1/2}(\RR,d_{\rm p, -1}),
    \]
    with bound from above by $\|u\|_{\Hcrit}$. The claim follows noting that for $|z|\leq 1$, it holds
    \[
    |z|_{\rm p}\sim |z|_{\rm p,1}\sim |z|_{\rm p,-1}\leq |z|^{1/2}.\qedhere
    \]
\end{proof}

As anticipated at the beginning of the subsection, we now construct new solutions of \eqref{eq:miuraEquation} by combining those coming from Corollary \ref{cor:estimates-for-wtv-part-2}, using the Cole--Hopf transformation. In the following, $\rho\in C^\infty_c(\RR)$ is again a standard mollifier. For a distribution $V\in\Dscr'(\RR)/\R$ defined up to an additive constant, by an abuse of notation, we will refer to the unique distribution $V\in\Dscr'(\RR)$ in its equivalence class satisfying the normalization condition
\begin{equation}\label{eq:normalization-condition}
    \int_{\RR} V(x,y)\rho(x,y) \,dx\,dy=0
\end{equation}
as a `normalization' of $V$.
\begin{proposition}\label{prop:kink-addition_map_eleV}
    Let $u\in\Hcrit$ small enough, $M\geq 1$, and $\vec \lambda\in \R^M$ such that $\lambda_1<\dots<\lambda_M$. Let $v^{(j)}\in L^3(\RR)+\lambda_j$, $1\leq j\leq M$ be the unique solutions to equation \eqref{eq:miuraEquation} with $\lambda=\lambda_j$, given by Corollary \ref{cor:estimates-for-wtv-part-2}. Let $V^{(j)}$ the corresponding solutions to equation \eqref{eq:primitiveEquation} as in Corollary \ref{cor:BMO-bounds-for-Vlambda}, normalized as in \eqref{eq:normalization-condition}. Given $\vec c:=(c_1,\dots,c_M)\in \R^M$, the functions
    \[
    \psi:=\frac{1}{\sum_{j=1}^M e^{c_j}}\sum_{j=1}^M e^{V^{(j)}+c_j},\qquad V:=\log\psi,\qquad v:=\partial_x V
    \]
    are all well-defined distributions, with $\psi,1/\psi\in L^6_\loc(\RR)$, $\partial_x\psi\in L^2_\loc(\RR)$, $V\in L^p_\loc(\RR)\,\forall p<\infty$, $v\in L^3_\unif(\RR)$. The function $\psi$ solves equation \eqref{eq:psi_is_in_the_kernel_of_Lax_operator}, and $v$ solves \eqref{eq:miuraEquation} distributionally. The map 
    \[
    \eleV^{\vec\lambda}\colon(u,\vec c)\mapsto v
    \]
    is uniformly continuous\footnote{Every topological vector space has a natural uniform structure by the translation invariance of the topology. The statement is equivalent to uniform continuity with values in $L^3(K)$ when restricting to any compact $K\subset\RR$.} from $B_{\eps_0}^{\Hcrit}(0)\times \R^M$ to $L^3_\loc(\RR)$.
\end{proposition}
\begin{proof}
    By Corollary \ref{cor:BMO-bounds-for-Vlambda}, $\psi\in L^6_\loc(\RR)$ and $\partial_x\psi=\frac{1}{\|e^{\vec c}\|_{\ell^1}}\sum_{j=1}^M v^{(j)}e^{V^{(j)}+c_j}\in L^2_\loc(\RR)$. The function $V=\log \psi$ is finite a.e., and it holds
    \[
    \min_j\{V^{(j)}\}\leq V \leq \max_j\{V^{(j)}\}.
    \]
    Since the functions $V^{(j)}$ are in $\PBMOx{\lambda_j}^0(\RR)$, they lie in $L^p_\loc(\RR)$ for any $p<\infty$ by Lemma \ref{cor:exponential-integrability-of-BMO-functions}. Thus, $V$ is also in $L^p_\loc(\RR)$, $p<\infty$, so one can set $v=\partial_x V$. Moreover, it holds
    \[
    v=\sum_{j=1}^M \zeta^{(j)}v^{(j)},\qquad \zeta^{(j)}:=\frac{e^{V^{(j)}+c_j}}{\sum_{k=1}^Me^{V^{(k)}+c_k}},
    \]
    in particular $v\in L^3_\unif(\RR)$, since $v^{(j)}\in L^3(\RR)+\lambda_j$ for all $j$, and $0\leq \zeta^{(j)}\leq 1$ a.e..

    Consider the map $(u,\vec c)\mapsto \zeta^{(j)}$. Rewrite $\zeta^{(j)}$ as
    \[
    \zeta^{(j)}:=\frac{1}{\sum_{k=1}^Me^{(V^{(k)}-V^{(j)})+(c_k-c_j)}}.
    \]
    By computing the differential explicitly, using the normalization \eqref{eq:normalization-condition} of the functions $V^{(j)}$, as well as the analyticity of the map $u\mapsto V^{(k)}-V^{(j)}$ with values in $C^{0,1/4}_\unif$ coming from Lemma \ref{lemma:holder_regularity_for_difference_of_primitive_solutions}, and the fact that $|\zeta^{(k)}|\leq 1$, it follows that the map $(u,\vec c)\mapsto \zeta^{(j)}$ is Lipschitz continuous from $B_{\eps_0}^{\Hcrit}(0)\times \R^M$ to $L^\infty(K)$ for any compact $K\subset \RR$. That is, the same map is uniformly continuous with values in $L^\infty_\loc(\RR)$. Combining this with the analyticity of the maps $u\mapsto v^{(j)}$ by Corollary \ref{cor:estimates-for-wtv-part-2} yields the uniform continuity of $\eleV^{\vec\lambda}$.
    
    Finally, by Corollary \ref{cor:BMO-bounds-for-Vlambda} and by linearity, it follows that $\psi$ solves equation \eqref{eq:psi_is_in_the_kernel_of_Lax_operator}, and by the Cole--Hopf transformation $v=\partial_x\log(\psi)$, $v$ solves equation \eqref{eq:miuraEquation} (this last step is true by direct computation when $u$ is smooth, and the statement extends to all small enough $u\in\Hcrit$ by continuity).
\end{proof}

\subsection{Proofs of Theorem \ref{theorem:theorem_1} and Corollary \ref{cor:almost-conservation-L2-norm-around-line-soliton}}

Now we are ready to give an explicit characterization of the solutions and of the map $\dtsmV$ in Theorem \ref{theorem:theorem_1}.
\begin{definition}\label{def:VcAndvc}
    Given a small $u\in\Hcrit$, let $v^{\pm}\in L^3(\RR)\pm 1$ the solutions of \eqref{eq:miuraEquation} given by Corollary \ref{cor:estimates-for-wtv-part-2} with $\lambda=\pm 1$. Define $V^\pm$ as the solutions of \eqref{eq:primitiveEquation} given by Lemma \ref{lemma:existenceOfPrimitiveSolution} corresponding to $v^\pm$, normalized by the condition \eqref{eq:normalization-condition}. Finally, for $c\in\R$, define
    $$ V^c:=\log\left( \frac{e^{V^+-c}+e^{V^-+c}}{e^c+e^{-c}} \right),\qquad v^c:=\partial_x V^c=\frac{v^+e^{V^+-c}+v^-e^{V^-+c}}{e^{V^+-c}+e^{V^-+c}}. $$
    In other words, $v^c=\eleV^{\vec \lambda}(u,\vec c)$, with $\eleV$ as in Proposition \ref{prop:kink-addition_map_eleV}, $\vec \lambda=(-1,1)$, and $\vec c=(c,-c)$.
\end{definition}

\begin{proposition}\label{prop:vcAreSolutions}
    The functions $\{v^c\}_{c\in\R}$ in Definition \ref{def:VcAndvc} solve equation \eqref{eq:miuraEquation} and satisfy the four assumptions \eqref{eq:v_minus_G_0_is_in_exponentially_weighted_L^2}--\eqref{eq:v_is_in_L^3_plus_tanh_sigma} stated before Corollary \ref{cor:uniquenessOfMiura}. In particular, $v^c$ satisfies the assumptions of both Corollary \ref{cor:uniquenessOfMiura} and Proposition \ref{prop:unconditional_uniqueness_miura}.
\end{proposition}

\begin{proof}
    The fact that $v^c$ satisfies equation \eqref{eq:miuraEquation} distributionally is a consequence of Proposition \ref{prop:kink-addition_map_eleV}. It is enough to show \eqref{eq:v_minus_G_0_is_in_exponentially_weighted_L^2}, \eqref{eq:v_minus_G_0_is_small_at_minus_infinity} by Corollary \ref{cor:uniquenessOfMiura}. Consider $\wV^\pm:=V^\pm-(\pm x+y)$, and set 
    \[
    \nu:=\aha (V^+-V^--2c),\quad \mu:=x-\nu.
    \]
    Let
    \[
    \tanh\hspace{-0.5pt}\circ\hspace{2pt}\nu(x,y):=\tanh(\nu(x,y)),\quad \eta^\pm\!\!\circ\nu(x,y):=\eta^\pm(\nu(x,y)).
    \]
    Then, one can write $v^c$ as
    \[
    v^c:=G_0+(\tanhcnu-\tanh)+(\etapcnu-\eta^+)\cdot(\wv^+-\wv^-).
    \]
    It is clear that $\partial_x\mu=-\fr12(\wv^+-\wv^-)$, which lies in $C_0(\R_y,L^2(\R_x))$ by Corollary \ref{cor:estimates-for-wtv-part-2}. In particular, it holds
    \begin{equation}\label{eq:z-001}
    |\mu(x_1,y)-\mu(x_2,y)|\leq C_y|x_1-x_2|^\aha,
    \end{equation}
    where $C_y=\fr12 \|\wv^+(\cdot,y)-\wv^-(\cdot,y)\|_{L^2(\R_x)} \lesssim \|u\|_{\Hcrit}$ $\forall y$. This immediately implies that
    \[
    (\tanhcnu-\tanh)(\cdot,y)\,\,\,,\,\,\,\,\, \Big((\etapcnu-\eta^+)\cdot(\wv^+-\wv^-)\Big)(\cdot,y)\,\,\,\in\, L^2(\R_x;\cosh^2(x)dx)\,\,\,\,\forall y\in\R.
    \]
    Moreover, by Lemma \ref{lemma:holder_regularity_for_difference_of_primitive_solutions}, $\mu$ is H\"older-continuous on the whole $\RR$, so by dominated convergence it follows that $v^c-G_0\in C(\R_y,L^2(\R_x;\cosh^2(x)dx))$. Finally, since $C_y\to 0$ as $y\to \pm\infty$, assumption \eqref{eq:v_minus_G_0_is_small_at_minus_infinity} is satisfied by choosing $\gamma=\gamma(y)$ such that $\gamma=\mu(\gamma,y)$, which always exists due to \eqref{eq:z-001}, and the same condition holds for $y\to+\infty$.
\end{proof}
Recall that, by Lemma \ref{lemma:expDecomposition} part \ref{item:expDec-005}, \eqref{eq:v_minus_G_0_is_small_at_minus_infinity} and \eqref{eq:v_minus_G_0_is_small_at_plus_infinity} are equivalent to
    \[
    \lim_{y\to\pm\infty}\|\cosh_{\alpha(y)}\omega(\cdot,y)\|_{L^2(\R_x)}=0.
    \]

In principle, we have many ways of parametrizing the family of solutions $\{v^{c}\}_{c\in\R}$: by the parameter $c$, by the phase shift $\alpha_0:=\alpha(0)$ at ordinate $y=0$, and by the quantity $\int v^\alpha-v^0dx$, where the first and the third are in a way more canonical but only up to an additive constant. Here we also introduce the parameter $\gamma_0$, which we used to formulate Theorem \ref{theorem:theorem_1}, and can be defined with very few assumptions on a general solution of \eqref{eq:miuraEquation}.
\begin{lemma}[Change of parameter]\label{lemma:change-of-variables-c-alpha_0-gamma_0}
    Let $u\in\Hcrit$ small, and consider $\{v^c\}_{c\in \R}$ as in Definition \ref{def:VcAndvc}. There exist invertible, smooth changes of variables on $\R$, $c\mapsto\alpha_0$, $c\mapsto\gamma_0$, determined uniquely by the conditions
    \begin{equation}\label{eq:definition_of_alpha_0}
        \int_\R v^c(x,0)-G_{\alpha_0}(x,0)dx=0,
    \end{equation}
    \begin{equation}\label{eq:definition-of-gamma_0}
        \int_{\RR} \rho(x-\gamma_0,y) v^c(x,y) dxdy=0.
    \end{equation}
    In particular, the map $c\mapsto v^c$ is injective. Moreover, it holds
    \begin{align}
        |\gamma_0-\alpha_0|&\lesssim \|u\|_\Hcrit,\label{eq:gamma_0_and_alpha_0_are_uniformly_close}\\
        |\gamma_0-c|+|\alpha_0-c|&\lesssim (1+|c|)\|u\|_{\Hcrit}.\label{eq:c_is_close_to_alpha_0_and_gamma_0_for_small_c}
    \end{align}
    Finally, the map $(u,c)\mapsto(u,\gamma_0)$ is bi-Lipschitz on bounded sets. In particular, the map $\dtsmV$ defined as
    \[
    \dtsmV(u,\gamma_0):=\eleV^{(-1,1)}(u,(c,-c))=v^c
    \]
    coincides with $\eleV^{(-1,1)}$, up to a homeomorphism and a projection from $\RR$ to $\{(c_1,c_2)\,|\,c_1+c_2=0\}$.
\end{lemma}
\noindent Note that for small $u$, estimate \eqref{eq:c_is_close_to_alpha_0_and_gamma_0_for_small_c} implies
\[
|\alpha_0-c|\lesssim (1+|\alpha_0|)\|u\|_{\Hcrit},\quad |\gamma_0-c|\lesssim (1+|\gamma_0|)\|u\|_{\Hcrit}.
\]
The proof is moved to Appendix \ref{subsec:proofs}.

\begin{remark}\label{rk:definition-of-gamma(y)}
    Similarly as above, we can define $\gamma(y)$ for all $y\in \R$ as the number determined uniquely by
    \[ \int_{\RR} \rho(x'-\gamma(y),y'-y) v^c(x',y') dx'dy'=0. \]
    By translation invariance in $y$, \eqref{eq:gamma_0_and_alpha_0_are_uniformly_close} implies
    \[ \|\gamma-\alpha\|_{L^\infty(\R_y)}\lesssim \|u\|_\Hcrit. \]
\end{remark}
\begin{corollary}\label{cor:log-bound-on-alpha}
    Let $v^c$ as in Definition \ref{def:VcAndvc}. Consider the associated phase shift $\alpha$ given by Lemma \ref{lemma:expDecomposition} thanks to Proposition \ref{prop:vcAreSolutions}. Then,
    \[ |\alpha(y_2)-\alpha(y_1)|\lesssim \log(2+|y_2-y_1|)\|u\|_\Hcrit. \]
\end{corollary}
\begin{proof}
    By Remark \ref{rk:definition-of-gamma(y)}, we can replace $\alpha$ with $\gamma$ in the above statement. We recall that $v^c$ can be written as
    \[
        v^c=\tanhcnu+(\etapcnu)\wv^++(\etamcnu)\wv^-,
    \]
    where $\nu(x,y):=\aha (V^+(x,y)-V^-(x,y)-2c)$.
    It is then a straightforward consequence of the bound from Lemma \ref{lemma:holder_regularity_for_difference_of_primitive_solutions} that, in order to show the logarithmic bound on $\gamma$, it is enough to prove the same bound for any function $\sigma=\sigma(y)$ such that $\nu(\sigma(y),y)=0$.
    
    Now, consider $\mu(x,y):=x-\nu(x,y)$. By Corollary \ref{cor:BMO-bounds-for-Vlambda} and Lemma \ref{lemma:holder_regularity_for_difference_of_primitive_solutions}, $\mu\in C^{0,\alpha}_\unif(\RR)\cap(\PBMOx 1(\RR)+\PBMOx {-1}(\RR))$. From this and Lemma \ref{lemma:BMO-logarithmic-growth}, it holds that $\mu$ grows at most logarithmically in $\RR$, that is,
    \[
    |\mu(x_1,y_1)-\mu(x_2,y_2)|\lesssim \log(2+|(x_1,y_1)-(x_2,y_2)|)\|u\|_{\Hcrit}.
    \]
    The logarithmic bound on $y\mapsto \sigma(y)$ is a direct consequence of the above bound.
\end{proof}

\begin{proof}[Proof of Theorem \ref{theorem:theorem_1}]
    The proof follows combining previous results. We first claim that the data to solution map $\dtsmV$ is the same we defined in Lemma \ref{lemma:change-of-variables-c-alpha_0-gamma_0}. In particular, we set $v=\dtsmV(u,\gamma_0)=\eleV^{(-1,1)}(u,(c,-c))$, with $c$ that depends univocally on $u$ and on $\gamma_0$ as in the Lemma. This $v$ indeed satisfies the localization condition
    \[
    \int_{\RR} \rho(x-\gamma_0,y) v(x,y) \,dx\,dy=0
    \]
    by the definition of the map $c\mapsto \gamma_0$ in Lemma \ref{lemma:change-of-variables-c-alpha_0-gamma_0}, and satisfies \eqref{eq:v_is_in_L^3_plus_tanh_sigma} by Proposition \ref{prop:vcAreSolutions}, so the existence statement is proved. The uniqueness is given by Proposition \ref{prop:unconditional_uniqueness_miura} and the bijectivity of the change of variables $\alpha_0\mapsto \gamma_0$ in Lemma \ref{lemma:change-of-variables-c-alpha_0-gamma_0}. The decomposition is given by Lemma \ref{lemma:expDecomposition} part \ref{item:expDec-003}, which can be applied since $v$ satisfies \eqref{eq:v_minus_G_0_is_in_exponentially_weighted_L^2} by Proposition \ref{prop:vcAreSolutions}. The same Proposition says that Corollary \ref{cor:uniquenessOfMiura} applies to $v$, and this proves the estimate on $\alpha,\omega$. The bounds on $v^\pm$ follow directly from Corollary \ref{cor:estimates-for-wtv-part-2}. The additional bound on $\alpha$ follows by Corollary \ref{cor:log-bound-on-alpha}. The continuity of the map and the explicit formula both follow from Proposition \ref{prop:kink-addition_map_eleV} and the bi-Lipschitz change of variables $(u,c)\mapsto(u,\gamma_0)$ from Lemma \ref{lemma:change-of-variables-c-alpha_0-gamma_0}.
\end{proof}

\begin{remark}
    From the decomposition in Theorem \ref{theorem:theorem_1}, it follows in addition that if $u\in L^2(\RR)$, then
    \[
    \|(v-\tanh_\alpha)_x\|_{L^2(\RR)}\lesssim \|u\|_{\Hcrit\cap L^2(\RR)}.
    \]
    In other words, the B\"acklund transform $\samB$ defined in \eqref{eq:backlund-transformation-definition-introduction} satisfies
    \[
    \|\samB(u,\gamma_0)-\ph_\alpha\|_{L^2(\RR)}\lesssim \|u\|_{\Hcrit\cap L^2(\RR)}.
    \]
    In fact, it holds
    \[
    (v-\tanh_{\alpha})_x=(\eta^+_\alpha \wv^+_x+\eta^-_\alpha \wv_x^-)+\aha\sech^2_\alpha\cdot\,(\wv^+-\wv^-)+\omega_x,
    \]
    and we can estimate all the three terms on the right hand side by the second and third estimate in \eqref{eq:estimates-for-wv-lambdaiszero-L2} using Corollary \ref{cor:estimates-for-wtv-part-2}, and by Theorem \ref{theorem:theorem_1} respectively.
\end{remark}

\begin{proof}[Proof of Corollary \ref{cor:almost-conservation-L2-norm-around-line-soliton}]
First of all, if $u\in L^2(\RR)\cap \Hcrit$ is small enough in both norms, the solution $v$ from Theorem \ref{theorem:theorem_1} coincides with the one given by Proposition \ref{proposition:existenceOfAnEternalSolution} and its proof, for some $\beta_0\in\R$. This is immediate if $u$ is compactly supported in $\RR$ thanks to the well-posedness of the initial-value problem for Burgers' equation, and it extends to all $u$ as above by approximation using the continuity of $\dtsmV$ (note that the solution in Proposition \ref{proposition:existenceOfAnEternalSolution} is constructed as a weak limit of solutions whose data are restrictions of the datum $u$ on a half plane $\{y>y_N\}$, with $y_N\to-\infty$). In particular, $v$ satisfies the estimates of the Proposition.

We have to prove two estimates. The estimate $|\wu|_{L^2_\ph(\RR)}\lesssim \|u\|_{L^2(\RR)}$ is an immediate consequence of the estimate $\|\beta_y\|_{L^2(\R_y)}+\|w_x\|_{L^2(\RR)}\lesssim \|u\|_{L^2(\RR)}$ of Proposition \ref{proposition:existenceOfAnEternalSolution}, where $v=w+\tanh_\beta$ is the decomposition of $v$ given by the Proposition, and of the identity $\wu=u-2w_x+\ph_\beta$.               
We now prove the converse estimate. Note that $\wu=u-2v_x$ satisfies the equation
\[
    v_y+v_{xx}=(v^2)_x-\wu_x.
\]
Let $\sigma\in C(\R_y)$ such that $\|\sigma_y\|_{L^2(\R)}<\infty$ and $\wu-\ph_\sigma\in L^2(\RR)$. Let $z:=v-\tanh_\sigma$, $g:=\wu-\ph_\sigma$.

We claim that $z\in C_0(\R_x,L^2(\R_x))$. For this, note that the three following decompositions hold:
\begin{equation*}
    v=w+\tanh_\beta=G_\alpha+\omega=z+\tanh_\sigma,
\end{equation*}
where $w,\beta$ are the functions from the decomposition from Lemma \ref{lemma:decompositionWAlpha} and $\omega, \alpha$ are as in the decomposition of Lemma \ref{lemma:expDecomposition}. We know that $v=v^c$ for some $c\in\R$ as in Definition \ref{def:VcAndvc}, and by Corollary \ref{cor:estimates-for-wtv-part-2}, $\|\wv^\pm\|_{C_0L^2}\lesssim \|u\|_{L^2(\RR)}$. Also, by Proposition \ref{prop:vcAreSolutions}, $\cosh_{\alpha}\omega\in C_0(\R_y,L^2(\R_x))$, so it follows $v-\tanh_{\alpha}\in C_0(\R_y,L^2(\R_x))$. By the definition of $w$ in the decompositon of Lemma \ref{lemma:decompositionWAlpha}, it follows immediately that $w\in C_0(\R_y,L^2(\R_x))$. Finally, since both $z_x$ and $w_x$ belong to $L^2(\RR)$, it has to hold $\beta-\sigma\in H^1(\R_y)$, so $z=w+(\tanh_\beta-\tanh_\sigma)\in C_0(\R_y,L^2(\R_x))$ and the claim is proved.

Now, the functions $z,\sigma$ satisfy
\[
z_y+z_{xx}-2(\tanh_\sigma z)_x=(z^2)_x-g_x-\sigma_y\sech^2_\sigma.
\]
By the usual energy estimates obtained multiplying the equation by $z$ and integrating in space, it holds for smooth enough functions
\[
\fr12 \fr{d}{dy}\left[\int_\R z^2dx\right]-\int_\R z_x^2dx-\int_\R \sech^2_{\sigma} z^2dx=\int_\R z_xg\,dx-\sigma_y\int_\R\sech^2_\sigma z\,dx.
\]
The last term can be controlled by $\sqrt 2\|\sigma_y\|_{L^2(\R_y)}\|\sech_\sigma z\|_{L^2(\RR)}$ when integrated in time, and the first term on the right-hand side is bounded by $\|z_x\|_{L^2(\RR)}\|g\|_{L^2(\RR)}$. By an approximation argument with smooth functions, using the fact that $z\in C_0(\R_y,L^2(\R_x))$, the above yields the bound
\[
\|z\|_{L^\infty L^2}^2+\|z_x\|_{L^2(\RR)}^2+\|\sech_\sigma z\|_{L^2(\RR)}^2\lesssim \|g\|^2_{L^2(\RR)}+\|\sigma_y\|^2_{L^2(\R_y)}.
\]
Since $u=\wu+2v_x=g+2z_x$, we obtain
\[
\|u\|_{L^2(\RR)}\lesssim \|\wu-\ph_\sigma\|_{L^2(\RR)}+\|\sigma_y\|_{L^2(\R_y)},
\]
and the bound is proved by taking the infimum over all $\sigma$.
\end{proof}

\section{The time-dependent B\"acklund transform}\label{sec:time}
We now go back to space-time equations. This time we will consider a time-dependent $u$ satisfying \eqref{eq:KP-II}, and instead of studying equation \eqref{eq:miuraEquation}, we will study the whole system \eqref{eq:systemForlittlev-miura+mKP-II}, as well as its relations with \eqref{eq:mKP-II} and the Lax system \eqref{eq:systemCompatibility}.\\
The goal of this section is to prove Theorem \ref{theorem:theorem_2}. To do this, we give definition and properties of the \emph{elementary solutions} of system \eqref{eq:systemForlittlev-miura+mKP-II} and prove a general nonlinear superposition principle in Proposition \ref{prop:eleVt_superposition-of-elementary-functions} that allows to combine such solutions.

\subsection{The well-posedness theory}
\subsubsection{Well-posedness around the zero solution}
We have a satisfactory well-posedness theory of the \eqref{eq:KP-II} equation on $\mathbb R^2$ in high regularity spaces (see \cite{hadac2008well-posednessKP-II} and references therein). The following Proposition can be seen as a special case of \cite{molinetSautTzvetkov2011non-localized}*{Theorem 1.2}.

\begin{definition}[\cite{molinetSautTzvetkov2011non-localized}]\label{def:Bourgain_type_spaces_X^s,b}
We define the spaces $X^{b,b_1,s}$ as the Banach spaces with norm
$$
\|u\|_{X^{b, b_1, s}}^2=\int_{\mathbb{R}_{\xi, \eta}^2 \times \mathbb{R}_\tau}\left\langle\frac{\langle\sigma(\tau, \xi, \eta)\rangle}{\langle\xi\rangle^3}\right\rangle^{2 b_1}(1+\xi^2+\eta^2)^s\langle\sigma(\tau, \xi, \eta)\rangle^{2 b}|\hat{u}(\tau, \xi, \eta)|^2 \, d\tau\, d\xi\, d\eta,
$$
where $\sigma(\tau, \xi, \eta)=\tau-4\pi^2\xi^3+3\eta^2 / \xi$. For any $T>0$, the norm in the localized version $X_T^{b, b_1, s}$ is given by
$$
\|u\|_{X_T^{b, b_1, s}}=\inf \left\{\|w\|_{X^{b, b_1, s}}\,|\; w(t)=u(t) \text { on }(0, T)\right\}.
$$
\end{definition}
Note that $X^{b,b_1,s}_T\hookrightarrow C([0,T],H^s(\RR))$ when $b>\aha$.

\begin{proposition}\label{prop:well-posedness-of-KP-II-L2}
    There exists an $\eps_0>0$ such that the following holds. Fix $k\in \Natural$, $\eps<\eps_0$, $1/4<b_1<3/8$. Let $u_0\in H^k(\RR)$. There exists a unique solution $u\in C([0,\infty),H^k(\RR))$ of \eqref{eq:KP-II} such that $u(0)=u_0$ and $u|_{[0,T]}\in X^{1/2+\eps,b_1,k}_T$ for all $T>0$. The data-to-solution map is analytic.
\end{proposition}
The next Lemma addresses the time regularity of the solutions and can be proved by directly looking at the Duhamel formulation of \eqref{eq:KP-II}.
\begin{lemma}\label{lemma:time-differentiability-}
     In the assumptions of Proposition \ref{prop:well-posedness-of-KP-II-L2}, if in addition $u_0\in \partial_x^j H^{k}(\RR)$ with $k-3j\geq 0$, it holds $u\in C^j([0,\infty),H^{k-3j}(\RR))$.
\end{lemma}

\begin{remark}\label{rk:time_smoothness_and_de_x_infty}
    Note that solutions of \eqref{eq:KP-II} are not necessarily smooth in time, even for $u_0\in H^\infty(\RR)$. This is true even for the linear flow, and can be checked by looking at the space-time Fourier transform of the solution of the linear KP equation. In order to have solutions that are smooth in space-time, we need $u_0\in\partial_x^\infty H^\infty(\RR):=\cap_{k\geq 0}\partial_x^kH^{2k}(\RR)$.
\end{remark}

The first well-posedness result in a scaling-critical space is due to Hadac--Herr--Koch \cite{hadacHerrKoch2009wellPosednessKP-IIinCriticalSpace}, who proved global well-posedness for small initial data in $\dot H^{-\aha,0}(\RR)$, as well as local well-posedness for data in the inhomogeneous version of the space. We state here a short version of the main theorem from the article, with the definition of the solution space in Appendix \ref{appendix:upvp}, Definition \ref{def:adapted-space-Zs}.
\begin{theorem}[{\cite{hadacHerrKoch2009wellPosednessKP-IIinCriticalSpace}*{Theorem 1.1}}]\label{theorem:global-well-posedness-KP-II-HHK09}
    Let $u_0 \in\Hcrit$ small enough. There exists a unique solution
$$
u \in \dot{Z}^{-\frac{1}{2}}((0, \infty)) \hookrightarrow C_b([0, \infty),\dot{H}^{-\frac{1}{2}, 0}(\mathbb{R}^2))
$$
of \eqref{eq:KP-II} on $[0, \infty)$. The data to solution map $u_0\mapsto u$ is analytic from a small ball in $\Hcrit$ centered at zero, to $\dot{Z}^{-\frac{1}{2}}((0, \infty))$.
\end{theorem}

\subsubsection{Well-posedness around the line soliton}
The well-posedness of \eqref{eq:KP-II} around the line soliton was first studied and proved in \cite{molinetSautTzvetkov2011non-localized}*{Theorem 1.2} by Molinet--Saut--Tzvetkov for data in $H^s(\RR)$, $s\in \N$ plus a non-modulated line soliton. Since the phase shift of the line solitons produced by our B\"acklund transform $\samB$ is not necessarily zero or vanishing rapidly at infinity, we need to prove a slightly modified version of Theorem 1.2 in \cite{molinetSautTzvetkov2011non-localized}.

Let $u$ be a solution to the KP-II equation with moving frame of reference
$$ u_t-cu_x+u_{xxx}-6uu_x+3\partial_x^{-1}u_{yy} = 0.$$
We fix the scaling parameter of the line soliton, so we set $c=4$. Using Notation \ref{notation:f_alpha}, we consider the ansatz
\[
u=v+\ph_\alpha,
\]
$\ph_\alpha:=\ph(x-\alpha(y))$, where $s\in\N$, $u\in H^s(\Real^2)$, $\alpha_y\in H^{s+1}(\R_y)$, and $\ph(x):=-2\sech^2(x)$ is the line soliton, which solves
$$ -4\ph_x+\ph_{xxx}-6\ph\ph_x = 0.$$
Note that by moving a low-regularity remainder inside $v$, it is always possible to assume $\alpha_y\in H^\infty(\R_y)$ with no harm to the following analysis. The function $v$ satisfies
\begin{equation}\label{eq:perturbedLineSoliton}
    v_t-4v_x+v_{xxx}-6vv_x-6(\ph_\alpha v)_x+3\partial_x^{-1}v_{yy} =(\alpha_y)^2\ph_{x,\alpha}-\alpha_{yy}\ph_\alpha.
\end{equation}

In the following, the definition of $X^{b,b_1,s}$ is modified by setting $\sigma(\tau,\xi,\eta)=\tau-4\xi-4\pi^2\xi^3+3\eta^2 / \xi$ to take into account the moving frame of reference. We first note that Strichartz estimates show that the solution of the linearized equation is in $X_T^{1/2+\eps,b_1,s}$ when the forcing is in $X_T^{-1/2+\eps,b_1,s}$. With this in mind, following the proof in \cite{molinetSautTzvetkov2011non-localized}*{Theorem 1.2} and adapting it to \eqref{eq:perturbedLineSoliton}, we see that:
\begin{enumerate}
    \item The nonlinearity can be treated in the same way as in the cited paper. In particular, for $\eps$ small enough and $1/4<b_1<3/8$, it holds
    \als{
        \|6vv_x\|_{X_T^{-1/2+\eps,b_1,s}}&\leq T^\nu \|6vv_x\|_{X_T^{-1/2+2\eps,b_1,s}}\\
        &\lesssim T^\nu \|v\|^2_{X^{1/2+\eps,b_1,s}_T},
    }
    for some $\nu>0$ (see \cite{molinetSautTzvetkov2011non-localized}*{Proposition 4.3 and equation (50)}).
    \item The term $(\ph_\alpha)_xv$ can be treated in the same way as in the paper (see \cite{molinetSautTzvetkov2011non-localized}*{Lemma 4.2}) with a slight modification:
    \begin{align*}
        \|(\ph_\alpha)_xv\|_{X_T^{0,0,s}}&\lesssim \|\partial_x\ph_\alpha\|_{W^{s,\infty}(\RR)}\|v\|_{L^\infty_T H^s(\RR)}\\
        &\lesssim\|\ph\|_{W^{s+1,\infty}(\R)}(1+\|\alpha_y\|_{ W^{s-1,\infty}(\R)})^s\|v\|_{X_T^{1/2+\eps,0,s}},
    \end{align*}
    where the estimate on $\|\partial_x\ph_\alpha\|_{W^{s,\infty}}$ can be checked directly for integer $s$. Note that the space $X_T^{0,0,s}$ works well here because $X_T^{0,0,s}\hookrightarrow X_T^{-1/2+\eps,b_1,s}$ if $b_1<1/2-\eps$.
    \item The terms $(\alpha_y)^2\ph_{x,\alpha}$ and $\alpha_{yy}\ph_\alpha$ are independent of time and lie in $H^s(\RR)$ assuming $\|\alpha_y\|_{H^{s+1}(\R)}<\infty$, in particular they belong to $X_T^{0,0,s}$ as well.
\end{enumerate}
The only term left to estimate is $\ph_\alpha v_x$. In the original paper, the term $\ph v_x$ is controlled thanks to a smoothing estimate for KP-II:
\als{
    \|\ph v_x\|_{X_T^{0,0,s}}&\lesssim (\sum_{k=1}^s\|\partial_x^k\ph\|_{L^2_xL^\infty_y})(\sum_{|\beta|\leq s}\|\partial_x\partial_{x,y}^\beta v\|_{L^\infty_x L^2_{y,t}})\\
    &\lesssim \|\ph\|_{H^s(\R)}\|v\|_{X_T^{1/2+\eps,0,s}}.
    }
In our case, for an arbitrary $\alpha$ such that $\alpha_y\in H^\infty(\R_y)$, the norm $\|\ph_\alpha\|_{L^2_xL^\infty_y}$ is infinite in general. We thus need a slight modification of local smoothing for the KP-II equation that takes into account the modulation of $\ph$ to adapt the result of Molinet--Saut--Tzvetkov to our case, where the soliton is modulated.

\begin{lemma}[Local smoothing with modulated weight]\label{lemma:modifiedLocalSmoothing}
    Let $c>0$, $u_0\in L^2(\RR)$, and $\alpha_y\in L^2(\R_y)$. It holds the estimate
    $$ c^{\fr14}\|\jap{c^\aha(x-\alpha)}^{-1}\partial_x e^{tS}u_0\|_{L^2_TL^2}+c^{\fr14}\|\jap{c^\aha(x-\alpha)}^{-1}\partial_x^{-1}\partial_y e^{tS}u_{0}\|_{L^2_TL^2}\lesssim L \|u_0\|_{L^2}, $$
    where $S=-\partial_x^3+c\partial_x-3\partial_y^2\partial_x^{-1}$, and $L=1+c^{3/4}\|\alpha_y\|_{L^2}^2$.
\end{lemma}
The case $\alpha\equiv 0$ yields the usual local smoothing estimate. The above is simply a modification that allows the level sets of the weight to be unbounded in $x$. The estimate for $S=-\partial_x^3-3\partial_y^2\partial_x^{-1}$ (with $c=1$ in the weights) holds with a constant that grows with the length of the time interval. The proof is moved to Appendix \ref{subsec:proofs}.

\begin{proposition}[Well-posedness of KP-II around a modulated line soliton]\label{prop:KP-IIglobalWellPosednessAroundSlackSoliton}
Fix $s\in \N$, and let $\eps,b_1$ as in Proposition \ref{prop:well-posedness-of-KP-II-L2}. Let $\alpha_y\in H^{s+1}(\R_y)$. For every $v_0\in H^s(\RR)$, there exists a unique global solution $v\in C([0,\infty),H^s(\RR))$ of equation \eqref{eq:perturbedLineSoliton} such that $v|_{t=0}=v_0$ and $u|_{[0,T]}\in X^{1/2+\eps,b_1,s}_T$ for all $T>0$. The data-to-solution map is analytic.
\end{proposition}

\begin{proof}
    By a standard argument for Bourgain type spaces, we can upgrade the modified local smoothing in Lemma \ref{lemma:modifiedLocalSmoothing} to the estimate
    \[
    \|\braket{x}_\alpha^{-1} (1-\Delta_{x,y})^{s/2}\partial_x v\|_{L^2_TL^2}\lesssim \|v\|_{X^{1/2+\eps,0,s}}.
    \]
    This immediately yields the estimate
    \als{
        \|\ph_\alpha \partial_x v\|_{X^{0,0,s}_T}&\sim\|\ph_\alpha \partial_x v\|_{L^2_TH^s(\RR)}\\
        &\lesssim \|v\|_{X^{1/2+\eps,0,s}}\\
        &\leq \|v\|_{X^{1/2+\eps,b_1,s}}.
        }
    The rest of the proof is analogous to that of \cite{molinetSautTzvetkov2011non-localized}*{Theorem 1.2}, with the use of the estimates summarized in this subsection and the above estimate. The globality of the solution in $H^s$ follows from the $L^2$ a priori estimate
    \[
    \|v(t)\|_{L^2(\RR)}^2\lesssim \exp(t\|\ph_x\|_{L^\infty(\R)})(1+\|v_0\|_{L^2}^2),
    \]
    and the $L^2$-subcriticality of the equation, analogously as in \cite{molinetSautTzvetkov2011non-localized}.
\end{proof}
\begin{remark}\label{rk:moving-low-regularity-remainder-in-v}
    The assumption $\alpha_y\in H^{s+1}(\R_y)$ is technical, and is only needed to close the fixed point argument in $X_T^{1/2+\eps,b_1,s}$. It is possible to relax the assumption to $\alpha_y\in H^{\max\{0,s-1\}}(\R_y)$ in the following way: first consider a standard regularization $\bar \alpha$ of $\alpha$ and note that the low-regularity remainder $r:=\ph_{\alpha}-\ph_{\bar\alpha}$ belongs to $H^s(\RR)$. Then, consider $\bar v_0:=v_0+r\in H^s(\RR)$ and let $\bar v\in X_T^{1/2+\eps,b_1,s}$ be the solution of \eqref{eq:perturbedLineSoliton} with $\alpha$ replaced by $\bar\alpha$ given by the Proposition, with initial datum $\bar v_0$. Then, $v:=\bar v-r$ is a solution of \eqref{eq:perturbedLineSoliton} with initial datum $v_0$. The price to pay, though, is that in general $v\not\in X_T^{1/2+\eps,b_1,s}$, although $v+r$ does indeed belong to the function space.
\end{remark}
\begin{remark}\label{rk:bourgain-space-norm-depends-monotonically-on-the-norm-of-alpha}
    Since the whole argument for the well-posedness of \eqref{eq:perturbedLineSoliton} only needs estimates from above on the $H^{s+1}$ norm of $\alpha_y$, for every time $T>0$ and $s\in\N$ it holds 
    \[
    \|v\|_{X_T^{1/2+\eps,b_1,s}}\leq C(\|v_0\|_{H^s},\|\alpha_y\|_{H^{s+1}(\R_y)},T,s),
    \]
    where $C$ is non-decreasing in the first three arguments. 
\end{remark}

Finally, we mention that solutions which are initially in $\partial_xL^2(\R^2)$ stay in that space for all times, and perturbations of the line soliton obey the same law.
\begin{proposition}\label{prop:the-KP-II-flow-preserves-D_xL^2}
    Let $u_0\in L^2(\RR)$, $u$ be the solution of \eqref{eq:KP-II} with initial datum $u_0$, and $v$ be the solution of \eqref{eq:KP-II} with initial datum $u_0+\ph(x)$. If in addition $u_0\in\partial_x L^2$, then $u,v-\ph(x-4t)\in C([0,\infty),\partial_xL^2(\RR))$, and it holds the estimate
    \[
    \|u(t)\|_{\partial_xL^2(\RR)}\leq \|u_0\|_{\partial_xL^2(\RR)}+C\sqrt t \|u_0\|_{L^2(\RR)}^2
    \]
    for some universal constant $C$.
\end{proposition}
The above Proposition is a refinement of \cite{mizumachi2018linesoli2}*{Lemma 3.1}.
\begin{proof}
    We consider $u$ and look at the Duhamel formulation
    \[
    u(t)=e^{tS}u_0-3\partial_x\int_0^t e^{(t-s)S}u^2(s)ds,
    \]
    where $S=-\partial_x^3-3\demu\partial_y^2$. By time translation invariance and the conservation of the $L^2$-norm, we have $\|u\|_{X^{1/2+\eps,b_1,0}_{[t,t+1]}}\lesssim \|u_0\|_{L^2(\RR)}$ for any $t\geq 0$, with $\eps,b_1$ as in Propositon \ref{prop:well-posedness-of-KP-II-L2}. By standard arguments involving Bourgain-type spaces and the Strichartz estimates for the group $e^{-t(\partial_x^3+3\partial_x^{-1}\partial_y^2)}$ (see \cite{koch-tataru-visan-2014-dispersive-equations-nonlinear-waves}), we have for every $t>0$
    \[
    \|u\|_{L^4_{[t,t+1]} L^4}\lesssim \|u\|_{X^{1/2+\eps,b_1,0}_{[t,t+1]}}\lesssim \|u_0\|_{L^2(\RR)}.
    \]
    In particular,
    \[
    \|u^2\|_{L^2_{[0,t]}L^2}\lesssim (1+\sqrt t)\|u_0\|_{L^2(\RR)}^2.
    \]
    After using H\"older's inequality in time, the above bound plugged into the Duhamel formulation yields
    \[
    \|u(t)\|_{\partial_xL^2(\RR)}\lesssim \|u_0\|_{\partial_xL^2(\RR)}+\sqrt t(1+\sqrt t) \|u_0\|_{L^2(\RR)}^2,
    \]
    which can be upgraded to the bound stated in the Lemma using the scaling symmetry \eqref{eq:scaling-symmetry-KP-II}.
    The statement for $v$ is proved analogously (cf. \cite{mizumachi2018linesoli2}*{Lemma 3.1}).
\end{proof}

\subsection{Elementary Lax-eigenfunctions and elementary solutions of mKP-II}

In this subsection we essentially replicate what we did in Subsection \ref{subsec:exact_solutions} to construct explicit solutions of equation \eqref{eq:miuraEquation}, adding a time dependence. This time we consider $u$ to be time-dependent and solving \eqref{eq:KP-II}, and the solutions of \eqref{eq:miuraEquation} that we are interested in will satisfy system \eqref{eq:systemForlittlev-miura+mKP-II}.

We motivate the definition of the elementary solutions as follows. The main difficulty in defining a solution of \eqref{eq:KP-II} via the B\"acklund transform is that the parameter $\gamma_0$ has to be chosen in a suitable way for all times $t>0$. Our strategy to solve this problem relies on two facts: first, while $\tanh$-like solutions of \eqref{eq:miuraEquation} as in Theorem \ref{theorem:theorem_1} are unique only up to 1 degree of freedom $\gamma_0$, solutions of \eqref{eq:miuraEquation} that are constant at infinity are well-defined with no further choice of parameters (see Lemma \ref{lemma:estimates-for-wtv} and Corollary \ref{cor:estimates-for-wtv-part-2}). Secondly, Lemma \ref{lemma:existenceOfPrimitiveSolutionTimeDependent} ensures that $x$-antiderivatives of these solutions are canonically well-defined up to a constant that depends neither on space nor on time. As a result, it follows that when using the Cole--Hopf transformation to combine the solutions of \eqref{eq:miuraEquation}, the parameters needed to interpolate those solutions are to be chosen once for all times. This is what will fix the parameter $\gamma_0=\gamma_0(t)$ in the statement of Theorem \ref{theorem:theorem_2}.

We start by showing that the function constructed in Lemma \ref{lemma:estimates-for-wtv} is a solution of \eqref{eq:mKP-II} if the datum $u$ is a time-dependent solution of \eqref{eq:KP-II}.
\begin{proposition}[Nonlinear existence of solutions of mKP-II]\label{prop:nonlinear-well-posedness-mKP-II-at-critical-regularity}
    Let $u_0\in\Hcrit$ be small enough and $v_0\in \dot H^{0,\fr14}(\RR)\cap \dot H^{\aha,0}(\RR)$ be the small solution of \eqref{eq:miuraEquation} given by Lemma \ref{lemma:estimates-for-wtv} with datum $u_0$. Let $u(t)$ be the solution of \eqref{eq:KP-II} given by Theorem \ref{theorem:global-well-posedness-KP-II-HHK09} and $v(t)$ the solution of \eqref{eq:miuraEquation} given by Lemma \ref{lemma:estimates-for-wtv}. Then, $v\in C_b([0,\infty),H^{0,\fr14}(\RR)\cap\dot H^{\aha,0}(\RR))$, $(u,v)$ solve system \eqref{eq:systemForlittlev-miura+mKP-II}, and $v$ is a limit of strong solutions of the \eqref{eq:mKP-II} equation from the well-posedness theory (\cite{kenigMartel2006mKP-IIwellPosedness}). Moreover, it holds the estimate
    \[
    \|v\|_{L^\infty_t (\dot H^{0,\fr14}(\RR)\cap \dot H^{\aha,0}(\RR))}+\|\demu v_y-v^2\|_{L^\infty_t\Hcrit }\lesssim \|u_0\|_\Hcrit
    \]
    and the map $u_0\mapsto v$ is continuous.
\end{proposition}

\begin{remark}
     The assumptions on $u_0,v_0$ can be rewritten as `Let $v_0\in \dot H^{0,\fr14}(\RR)\cap \dot H^{\aha,0}(\RR)$ small enough such that $\demu v_{0,y}-v_0^2\in \Hcrit$ is also small, and call $u_0:=v_{0,x}+v_0^2-\partial_x v_{0,y}$'. Note that the nonlinear term $\demu v_y-v^2$ is precisely one of the two terms appearing in the energy functional of the mKP-II equation
     \[
     E(v(t))=\int_{\RR}|\partial_xv(t)|^2+|\partial_x^{-1}\partial_y v(t)-v(t)^2|^2dx\,dy,
     \]
     which is formally conserved by the mKP-II flow (see \cite{kenigMartel2006mKP-IIwellPosedness}*{§1}). The function $v$ is a solution of the mKP-II equation in the sense that $(u,v)$ solve system \eqref{eq:systemForlittlev-miura+mKP-II}: the latter is related to the distributional mKP-II equation \eqref{eq:mKP-II-systemForm} as we noted in Section \ref{sec:preliminaries}.
\end{remark}
\begin{proof}
    The regularity and the bound on $v$ both follow from Lemma \ref{lemma:estimates-for-wtv} and the uniform-in-time smallness of $u$. The fact that $v$ is a limit of strong solutions of \eqref{eq:mKP-II} will be clear from the rest of the proof, so we only need to show that $v$ is a solution of system \eqref{eq:systemForlittlev-miura+mKP-II}. The map $u_0\mapsto v$ is continuous from $\Hcrit$ to $C_bL^3$, so by approximation it suffices to show the statement assuming $u_0\in\partial_x H^\infty(\RR)$. In particular, in these hypotheses we have $v_0,v_{0,x},\demu v_{0,y}\in H^\infty(\RR)$ by Lemma \ref{lemma:estimates-for-wtv}, so we fall in the range of applicability of the well-posedness theory of the mKP-II equation, as in \cite{kenigMartel2006mKP-IIwellPosedness}*{Theorem 1}. It follows that there exists a unique solution $\overline v \in C([0,\infty), H^\infty)$ to mKP-II with initial datum $v_0$ such that $\overline v_x,\demu \overline v_y \in C([0,\infty), H^\infty)$, and by the mapping property of the Miura map of Proposition \ref{prop:miura-maps-mKP-to-KP} and the uniqueness of the solution $u$ of KP-II (see also \cite{kenigMartel2006mKP-IIwellPosedness}*{Remark 1}) it holds
    \[
    -\demu \overline v(t)+\overline v^2(t)+\overline v_x(t)=u(t).
    \]
    By the uniqueness of the solution given by Lemma \ref{lemma:estimates-for-wtv}, it follows that $v\equiv \overline v$, so $v$ is a strong solution of \eqref{eq:mKP-II}. In particular, by the equivalence of systems \eqref{eq:systemForlittlev-miura+mKP-II} and \eqref{eq:mKP-II-systemForm}, it follows that $(u,v)$ solve system \eqref{eq:systemForlittlev-miura+mKP-II}.
\end{proof}
    Using symmetry \eqref{eq:mKP-II-symmetry} of system \eqref{eq:systemForlittlev-miura+mKP-II}, the content of Proposition \ref{prop:nonlinear-well-posedness-mKP-II-at-critical-regularity} can be easily extended to cover the cases where the initial data $(u_0,v_0)$ solve \eqref{eq:miuraEquation} and $v_0\in L^3(\RR)+\lambda$ for some constant $\lambda\in\R$.
\begin{corollary}\label{cor:v-(j)-is-a-solution-to-system-Miura+mKP}
    Let $u_0\in \Hcrit$ be small enough and let $u\in \dot Z^{-\aha}((0,\infty))$ be the solution of \eqref{eq:KP-II} given by Theorem \ref{theorem:global-well-posedness-KP-II-HHK09}. Fix $\lambda\in\R$. For all times $t\geq 0$, let $v(t)\in L^3(\RR)+\lambda$ be the solution of \eqref{eq:miuraEquation} given by Corollary \ref{cor:estimates-for-wtv-part-2}. Then, $v\in C([0,\infty),L^3(\RR))+\lambda$ and $(u,v)$ solves system \eqref{eq:systemForlittlev-miura+mKP-II}.
\end{corollary}

\begin{definition}[Elementary solutions]\label{def:elementary-solutions-of-system-Miura+mKP}
    Fix $u_0\in \Hcrit$ small enough and let $u\in \dot Z^{-\aha}((0,\infty))$ be as in Theorem \ref{theorem:global-well-posedness-KP-II-HHK09}. Let $\lambda_1\in \R$. The solution $v^{(1)}=\wv+\lambda_1$ of \eqref{eq:systemForlittlev-miura+mKP-II} as in Corollary \ref{cor:v-(j)-is-a-solution-to-system-Miura+mKP} is called the \emph{elementary solution} of \eqref{eq:systemForlittlev-miura+mKP-II} with parameter $\lambda_1$ associated to $u$.
\end{definition}
Next, given a solution $v$ of \eqref{eq:systemForlittlev-miura+mKP-II}, we construct a solution $V$ of \eqref{eq:systemForCapitalV} such that $V_x=v$ and establish bounds on $V$. We first prove that such a solution $V$ is unique up to an additive constant, which turns out to be independent of space and time.
\begin{lemma}\label{lemma:existenceOfPrimitiveSolutionTimeDependent}
    Let $\emptyset\neq I\subset \R_t$ be an open interval. Assume $u\in L^2_\loc(I\times\R^2)$, $v\in L^3_\loc(I\times\R^2)$, $w\in \Dscr'(I\times\R^2)$ are space-time distributions, such that $w_x=u_y$, $(u,v)$ solves the system \eqref{eq:systemForlittlev-miura+mKP-II}, and $(u,w)$ solves the KP-II equation, in the sense that
    \[
    u_t-6uu_x+u_{xxx}+3w_y=0.
    \]
    Then, there exists a unique $V\in\mathscr D'(I\times\R^2)$ up to an additive constant (independent of $t,x,y$) which solves system \eqref{eq:systemForCapitalV} with $\demu u_y=w$, and such that $V_x=v$. If $u,v,w$ are smooth, then $V$ is smooth.
\end{lemma}
The function $w$ is morally the term $\demu\partial_y u$ appearing in the KP-II equation, and we simply assume that it is well-defined.
\begin{proof}
    The proof for the uniqueness is analogous to the proof of Lemma \ref{lemma:existenceOfPrimitiveSolution}. For the existence, let $\wV$ be such that $\partial_x\wV=v$. Then, integrating system \eqref{eq:systemForlittlev-miura+mKP-II},
    \begin{equation}\label{eq:systemForCapitalV-distorted}
        \left\{
    \begin{aligned}
    &\wV_y-v_x=v^2-u+g\\
    &\wV_t+4v_{xx}+4v^3+12vv_x-6uv-3u_x-3w=h,
    \end{aligned}
    \right.
    \end{equation}
    where $g,h$ are distributions that are independent of $x$. Derivating the first equation of system \eqref{eq:systemForCapitalV-distorted} with respect to $t$ and the second equation with respect to $y$, after combining them and using the fact that $(u,w)$ solve KP-II, we find that $g_t=h_y$. This in turn implies, by applying Lemma \ref{lemma:distributional-antiderivative} twice, that $g=F_y$, $h=F_t$ for a third distribution $F$ independent of $x$. From here, the argument is analogous to the proof of Lemma \ref{lemma:existenceOfPrimitiveSolution}.
\end{proof}
Recall that for $u\in \dot Z^{-\aha}((0,\infty))$, the distribution $\demu u_y$ is well-defined thanks to Remark \ref{rk:Z-aha-is-in-L2-loc}.
\begin{lemma}\label{lemma:primitive-solutions-time-dependent-are-nice}
    Let $u_0\in \Hcrit$ be small and let $u\in \dot Z^{-\aha}((0,\infty))$ be the unique small solution of KP-II given by Theorem \ref{theorem:global-well-posedness-KP-II-HHK09}. Let $v\in C_b([0,\infty), L^3(\RR))+\lambda$ be the elementary solution of system \eqref{eq:systemForlittlev-miura+mKP-II} with parameter $\lambda\in\R$ associated to $u$, as in Definition \ref{def:elementary-solutions-of-system-Miura+mKP}. Then, the distribution $V$ solving system \eqref{eq:systemForCapitalV} and $V_x=v$, given (up to an additive constant) by Lemma \ref{lemma:existenceOfPrimitiveSolutionTimeDependent}, satisfies
    \[V(t,x,y)-(\lambda x + \lambda^2 y - 4\lambda^3 t)\in C([0,\infty),\PBMOx {\lambda}^0(\R^2)),\]
    and for all $t$, $V(t)$ coincides with the function obtained from $v(t)$ by Lemma \ref{lemma:existenceOfPrimitiveSolution} (up to a time-dependent additive constant). Moreover, it holds
    \begin{equation}\label{eq:z-030}
    \left\|\frac{d}{dt}\left(\int_{\RR} V\rho\,dx\,dy \right)+4\lambda^3\right\|_{L^2_\unif((0,\infty))}\lesssim \|u_0\|_{\Hcrit}.
    \end{equation}
\end{lemma}
We recall that here $\PBMOx \lambda^0(\RR)\hookrightarrow\mathscr S'(\RR)$ is simply the Banach space of all functions in $\PBMOx \lambda(\RR)$, to which one can equip the norm
\[
\|f\|_{\PBMOx{\lambda,\rho}^0(\RR)}=\|f\|_{\PBMOx \lambda(\RR)}+\left|\int_{\RR}f\rho\;dx\,dy\right|,
\]
where $\rho$ is a standard mollifier centered at the origin (see Definition \ref{def:spaces_of_homogeneous_type_and_BMO}). In particular, $\PBMOx \lambda^0$ embeds into $L^p_\loc(\RR)$ for all $p<\infty$ (see Corollary \ref{cor:exponential-integrability-of-BMO-functions}).
\begin{proof}
    First, by symmetry \eqref{eq:mKP-II-symmetry}, it suffices to show the statement for $\lambda=0$. Assume first that $u_0\in \partial_x^\infty H^\infty(\RR)$ (see Remark \ref{rk:time_smoothness_and_de_x_infty}). Then, $u$ is in $C^\infty([0,\infty),H^\infty(\RR))$. It follows from Lemma \ref{lemma:estimates-for-wtv} that $v\in C^\infty([0,\infty),H^\infty(\RR))$, and thus $V$ is smooth by the previous Lemma. Since $\partial_xV(t,\cdot,\cdot)=v(t,\cdot,\cdot)$, and since $V(t,\cdot,\cdot)$ solves equation \eqref{eq:primitiveEquation} with datum $u(t)$ for all $t\geq 0$, the function $V(t)$ agrees with the one given by Lemma \ref{lemma:existenceOfPrimitiveSolution} (with datum $u(t)$) for all $t\geq 0$ up to an additive constant which depends on $t$, by the uniqueness statement therein. By Corollary \ref{cor:BMO-bounds-for-Vlambda}, it holds the estimate
    \[
    \|V\|_{C_b\PBMO(\RR)}\lesssim \|u_0\|_{\Hcrit}.
    \]
    
    The above estimate does not give control on the time evolution of any additive constant, so we need an additional estimate. Consider the second equation in system \eqref{eq:systemForCapitalV}. Set $w:=\demu u_y$, which is well-defined by Remark \ref{rk:Z-aha-is-in-L2-loc}. Multiplying both sides by $\rho$ and integrating over space, we obtain the bound
    \begin{align*}
        \left|\frac{d}{dt}\int_{\RR} V\rho\,dx\,dy\right|&=\left|\int_{\RR} \rho\cdot\, (4v^3+4v_{xx}+12vv_x-6uv-3u_x-3w)dx\,dy\right|\\
        &\lesssim\|v\|_{L^3(B)}+\|v\|_{L^3(B)}^3+\|v\|_{L^3(B)}^2+\|u\|_{L^{3/2}(B)}\|v\|_{L^3(B)}\\
        &\QQQ+\left|\int_{\RR} \rho w\,\,dx\,dy\right|,
    \end{align*}
    where $B=B_1((0,0))$ is the support of $\rho$. The $L^3$ norm of $v$ is controlled by $\|u\|_\Hcrit$ by Lemma \ref{lemma:estimates-for-wtv} (remember that we assumed $\lambda=0$). For the norm of $u$ in $L^{3/2}(B)$, we note that from Theorem \ref{theorem:global-well-posedness-KP-II-HHK09} we have the bound
    \[
    \|u\|_{\dot Z^{-\aha}((0,\infty))}\lesssim\|u_0\|_\Hcrit,
    \]
    with $\dot Z^s$ as in Definition \ref{def:adapted-space-Zs}, and thanks to Remark \ref{rk:Z-aha-is-in-L2-loc} we can estimate $u$ locally in $L^2$ in space-time. The last term is also $L^2$-integrable in time by the bound in Remark \ref{rk:Z-aha-is-in-L2-loc}. It follows
    \begin{equation}\label{eq:z-029}
    \left\|\frac{d}{dt}\int_{\RR} V\rho\,dx\,dy\right\|_{L^2_\unif((0,\infty))}\lesssim \|u_0\|_{\Hcrit},
    \end{equation}
    which implies $V\in C([0,\infty),\PBMO^0(\RR))$ together with the above estimate. For general $u_0\in \Hcrit$, an approximation argument is enough to conclude thanks to the above a priori estimates.
\end{proof}

\subsection{The time-dependent B\"acklund transform - Proof of Theorem \ref{theorem:theorem_2}}

In this subsection we prove Theorem \ref{theorem:theorem_2}. First, we state a nonlinear superposition principle which allows to construct solutions of \eqref{eq:systemForlittlev-miura+mKP-II} from its elementary solutions associated to the same solution $u$ of \eqref{eq:KP-II}. This is a direct time-dependent analogue of Proposition \ref{prop:kink-addition_map_eleV}. We call the map below $\eleVt$ because its output is the forward time evolution of the output of $\eleV$ along the mKP-II flow.

\begin{proposition}[Nonlinear superposition of elementary solutions]\label{prop:eleVt_superposition-of-elementary-functions}
    Let $u_0\in\Hcrit$ be small enough, $M\geq 1$, and $\vec \lambda\in \R^M$ such that $\lambda_1<\dots<\lambda_M$. Let $u$ be the solution of \eqref{eq:KP-II} given by Theorem \ref{theorem:global-well-posedness-KP-II-HHK09} with $u|_{t=0}=u_0$. Let $v^{(j)}\in C_b([0,\infty),L^3(\RR))+\lambda_j$, $1\leq j\leq M$ be the corresponding elementary solutions of \eqref{eq:systemForlittlev-miura+mKP-II} associated to $u$. Let $V^{(j)}$ be the corresponding primitive solutions of system \eqref{eq:systemForCapitalV} given by Proposition \ref{lemma:existenceOfPrimitiveSolutionTimeDependent} and \ref{lemma:primitive-solutions-time-dependent-are-nice}, normalized as in \eqref{eq:normalization-condition} at $t=0$. Given $\vec c:=(c_1,\dots,c_M)\in \R^M$, the functions
    \begin{equation}\label{eq:nonlinear_superposition_elementary_sol_def_of_psi-v-wu}
        \psi:=\frac{1}{\sum_{j=1}^M e^{c_j}}\sum_{j=1}^M e^{V^{(j)}+c_j},\qquad V:=\log\psi,\qquad v:=\partial_x V,\qquad \wu:=u-2\partial_x v
    \end{equation}
    are well-defined and they satisfy\footnote{The `$b$' in $C_b$ this time refers to the notion of boundedness in Fr\'echet spaces. We are asking for the above functions to be bounded with values in $L^p(K)$ when restricted to any compact $K\subset \RR$.}
    \als{
        &\psi,1/\psi \in C([0,\infty), L^6_\loc(\RR)),\quad \psi_x\in C([0,\infty),L^2_\loc(\RR)),\\
        &V\in C_b([0,\infty),L^p_\loc(\RR))\quad\forall\,p<\infty,\\
        &v\in C_b([0,\infty),L^3_\loc(\RR))\cap L^6((0,\infty)\times\RR),\quad v_x\in L^2_\unif((0,\infty)\times\RR),\\
        &\wu\in L^2_\loc([0,\infty)\times \RR).
    }
    The function $v$ solves system \eqref{eq:systemForlittlev-miura+mKP-II}, and $\wu$ solves the KP-II equation in distributional form \eqref{eq:KP-II-distributional}. The map
    \[
    \eleVt^{\vec\lambda}:(u_0,\vec c)\mapsto v
    \]
    is continuous from $B_{\eps_0}^{\Hcrit}(0)\times \R^M$ to $C([0,\infty),L^3_\loc(\RR))\cap L^6_\loc([0,\infty)\times\RR)$ with $v_x\in L^2_\loc([0,\infty)\times \RR)$. It holds
    \begin{equation}\label{eq:eleVt_is_eleV_composed_with_a_continuous_curve_in_R^M}
    v(t)=\eleV^{\vec\lambda}(u(t),\vec {\bm c}(t))
    \end{equation}
    for all $t\geq 0$, for a $C^{0,\aha}$ curve $t\mapsto\vec {\bm c}(t)$, $\vec {\bm c}(0)=\vec c$. In particular, $\eleV^{\vec \lambda}=\eleVt^{\vec\lambda}|_{t=0}$.
\end{proposition}
\begin{proof}
    The main technicality is proving an estimate on $\wu$ to ensure it lies in $L^2_\loc([0,\infty)\times \RR)$, with continuous dependence on $u_0$. For this, by Theorem \ref{theorem:global-well-posedness-KP-II-HHK09}, Corollary \ref{cor:Z^aha_is_in_L6_de_x^aha_L3}, and Remark \ref{rk:Z-aha-is-in-L2-loc}, we know that
    \[
    u\in \dot Z^{-\aha}((0,\infty)) \hookrightarrow L^6((0,\infty),|\de_x|^\aha L^3(\RR))\cap L^2_\unif((0,\infty)\times \RR)
    \]
    Using the product estimate $\|fg\|_{L^6_tL^2_{x,y}}\lesssim \|f\|_{L^\infty_t L^3_{x,y}}\|g\|_{L^6_{x,y,t}}$ and the estimates
    \[
    \|\de_x\Gamma^{(c)}f\|_{L^6(\RR)}\lesssim \|f\|_{|\de_x|^\aha L^3(\RR)+L^2(\RR)},
    \]
    \[
    \|\de_x^2\Gamma^{(c)}f\|_{L^2(\RR)}\lesssim \|f\|_{L^2(\RR)},
    \]    
    $\Gamma^{(c)}=(\de_y-\de_x^2+c\de_x)^{-1}$ from Proposition \ref{prop:spacetime_estimate_heat_general}, it is straightforward to refine the fixed point argument in Lemma \ref{lemma:estimates-for-wtv} to show that for fixed $\lambda_1\in\R$, the map
    \[
    u_0\mapsto v^{(1)}
    \]
    as in Definition \ref{def:elementary-solutions-of-system-Miura+mKP} is analytic from a small ball $B_{\eps_0}(0)\subset\Hcrit$ to
    \[
    C_b([0,\infty),L^3(\RR))\cap L^6((0,\infty)\times\RR)+\lambda_1,
    \]
    with
    \[
    v^{(1)}_x\in C_b([0,\infty),\Hcrit)\cap L^2_\unif((0,\infty)\times\RR).
    \]
    The regularity of the functions defined in the statement and the continuity of the map follow analogously as in the proof of Proposition \ref{prop:kink-addition_map_eleV}, using in addition the analyticity of the above map and of the data to solution map in Theorem \ref{theorem:global-well-posedness-KP-II-HHK09} from $u_0\in\Hcrit$ to $u\in \dot Z^{-1/2}((0,\infty))$.
     
    For the remaining statements, except \eqref{eq:eleVt_is_eleV_composed_with_a_continuous_curve_in_R^M}, we can assume $u_0\in \partial_x^\infty H^\infty(\RR)$ by continuity, so all the functions appearing in the statement are smooth. The functions
    \[
    \psi^{(m)}:=e^{V^{(m)}}
    \]
    solve system \eqref{eq:systemCompatibility} as observed in Section \ref{sec:preliminaries}. By linearity,
    \[
    \psi:=\sum_{m=1}^M \psi^{(m)}
    \]
    solves system \eqref{eq:systemCompatibility}. Again, by the equivalences of Section \ref{sec:preliminaries}, $V:=\log(\psi)$ solves system \eqref{eq:systemForCapitalV}, and $v:=\partial_x V$ solves system \eqref{eq:systemForlittlev-miura+mKP-II}. The function $\wu$ solves the KP-II equation thanks to Proposition \ref{prop:miura-maps-mKP-to-KP}.

    Since the functions $v^{(m)}$, $V^{(m)}$ solve systems \eqref{eq:systemForlittlev-miura+mKP-II} and \eqref{eq:systemForCapitalV} respectively, they solve respectively \eqref{eq:miuraEquation} and \eqref{eq:primitiveEquation} with $u=u(t)$ for each $t\geq 0$, and $v^{(m)}(t)$ is the unique solution of \eqref{eq:miuraEquation} in $L^3(\RR)+\lambda_m$ as in Corollary \eqref{cor:estimates-for-wtv-part-2}. By Lemma \ref{lemma:existenceOfPrimitiveSolution}, called $\wt V^{(m)}(t)$ the unique solution of \eqref{eq:primitiveEquation} with $x$-derivative $v^{(m)}(t)$ normalized as in \eqref{eq:normalization-condition}, it holds
    \[
        V^{(m)}(t)=\wt V^{(m)}(t)+\int_{\RR}\rho \,V^{(m)}(t)\,\dxdy.
    \]
    In particular, by the definition of $\eleV$ in Proposition \ref{prop:kink-addition_map_eleV}, \eqref{eq:eleVt_is_eleV_composed_with_a_continuous_curve_in_R^M} holds with
    \[
    {\bm c}_m(t)=c_m+\int_{\RR}\rho\, V^{(m)}(t)\,\dxdy.
    \]
    Since $V^{(m)}$ is normalized as \eqref{eq:normalization-condition} at $t=0$, it holds $\vec{\bm c}(0)=\vec c$. The H\"older regularity of $\vec{\bm c}$ comes from estimate \eqref{eq:z-030} in Lemma \ref{lemma:primitive-solutions-time-dependent-are-nice}.
\end{proof}

Now we are ready to prove Theorem \ref{theorem:theorem_2}. We restate it here in a more detailed form.

\begin{customtheorem}{\ref{theorem:theorem_2}$'$}\label{theorem:theorem_2_reloaded}
    Let $u_0\in\Hcrit$ be small, $u\in C_b([0,\infty),\Hcrit)$ the global solution of \eqref{eq:KP-II} given by Theorem \ref{theorem:global-well-posedness-KP-II-HHK09}, and $\gamma_{0,0}\in\mathbb R$. Let $v_0:=\dtsmV(u_0,\gamma_{0,0})$ as in Theorem \ref{theorem:theorem_1}. There exists a continuous function $t\mapsto\gamma_0(t)$, $\gamma_0(0)=\gamma_{0,0}$ such that $(u,v)$ solves system \eqref{eq:systemForlittlev-miura+mKP-II}, where $v\in C([0,\infty),L^3_\loc(\RR))$ is defined as
    $$ v(t):=\dtsmV(u(t),\gamma_0(t)).$$
    Moreover, the curve $\wu(t):=\samB(u,\gamma_0(t))$ is a solution to the KP-II equation in distributional form and can be decomposed as $\wu(t)=\ph(x-\alpha(t,y))+u(t)+w(t)$, with the estimates
    \[ \sup_{t\geq 0}\left[\sup_{y_1,y_2\in\R}\frac{|\alpha(t,y_2)-\alpha(t,y_1)|}{\log(2+|y_2-y_1|)}+\|\alpha_y(t,\cdot)\|_{L^2_y}+\|w(t,\cdot,\cdot)\|_{H^{-\aha,0}(\RR)}\right]\lesssim \|u_0\|_{\Hcrit},
    \]
    \[
    \left\|\fr d{dt} \gamma_0-4\right\|_{L^2_\unif(0,\infty)}\lesssim\|u_0\|_{\Hcrit}.
    \]
    If in addition $u_0\in L^2(\RR)$ is small enough, then $\wu-\ph(x-\alpha(0,y)-4t)\in C([0,\infty),L^2(\RR))$ with the estimate
    \[
    |\wu(t)|_{L^2_\ph(\RR)}\sim|\wu(0)|_{L^2_\ph(\RR)}\sim \|u_0\|_{L^2(\RR)},\qquad t\geq 0,
    \]
    and $\wu$ is the solution of \eqref{eq:KP-II} coming from the well-posedness theory (see Proposition \ref{prop:KP-IIglobalWellPosednessAroundSlackSoliton}).
\end{customtheorem}

\begin{proof}
First, by Theorem \ref{theorem:theorem_1}, we can write $v_0$ as
\als{
v_0&=\partial_x\log(e^{V_0^+-c}+e^{V_0^-+c})\\
&=\eleV^{(-1,1)}(u_0,(c,- c)),
}
where $V_0^+$, $V_0^-$ are solutions of \eqref{eq:primitiveEquation} with datum $u_0$ as in Definition \ref{def:VcAndvc}, normalized as in \eqref{eq:normalization-condition}, and where $c\in\R$ depends bijectively on $\gamma_{0,0}$ as in Lemma \ref{lemma:change-of-variables-c-alpha_0-gamma_0}. Let $v^{(\pm 1)}=:v^\pm$ be the elementary solutions of system \eqref{eq:systemForlittlev-miura+mKP-II} with parameters $\lambda=\pm 1$ associated to $u$, as in Definition \ref{def:elementary-solutions-of-system-Miura+mKP}. Let $V^\pm$ be the corresponding solutions of \eqref{eq:systemForCapitalV} given by Lemmas \ref{lemma:existenceOfPrimitiveSolutionTimeDependent}, \ref{lemma:primitive-solutions-time-dependent-are-nice}, and normalize them as in \eqref{eq:normalization-condition} at $t=0$. By Proposition \ref{prop:eleVt_superposition-of-elementary-functions}, the function
\als{
v&=\partial_x\log(e^{V^+-c}+e^{V^-+c})\\
&=\eleVt^{(-1,1)}(u_0,(c,-c)),
}
solves system \eqref{eq:systemForlittlev-miura+mKP-II} with $v|_{t=0}=v_0$, and $\wu=u-2\partial_xv$ solves KP-II in distributional form. By Lemma \ref{lemma:primitive-solutions-time-dependent-are-nice}, it holds
\[
V^\pm|_{t=0}=V^\pm_0,
\]
and there exists a continuous $\bm c:[0,\infty)\to\R$ such that
\[
v(t)=\eleV^{(-1,1)}(u(t),(\bm c(t),-\bm c(t))).
\]
By the second change of variables in Lemma \ref{lemma:change-of-variables-c-alpha_0-gamma_0} and the bi-Lipschitz continuity result therein, there exists a continuous $\gamma_0=\gamma_0(t)$ with $\gamma_0(0)=\gamma_{0,0}$ such that $v(t)=\dtsmV(u(t),\gamma_0(t))$ for each time $t\geq 0$, where we defined $\dtsmV$ in Lemma \ref{lemma:change-of-variables-c-alpha_0-gamma_0} and coincides with the map in Theorem \ref{theorem:theorem_1}.

The curve $\wu:=u-2\partial_xv$ solves \eqref{eq:KP-II} distributionally by Proposition \ref{prop:eleVt_superposition-of-elementary-functions}. Taking $\alpha$ as in Theorem \ref{theorem:theorem_1} for each time, the decomposition and the estimates on $\alpha,w$ follow from the estimates of Theorem \ref{theorem:theorem_1}, since
\[
w=u-2(\omega+\eta^+_\alpha\wv^++\eta^-_\alpha\wv^-)_x.
\]

We now prove the continuity and the a priori bound on $\gamma_0(\cdot)$. By an approximation argument, using Ascoli--Arzelà theorem, the well-posedness of KP-II given by Theorem \ref{theorem:global-well-posedness-KP-II-HHK09} and the continuity of the map $\dtsmV$ in Theorem \ref{theorem:theorem_1}, we can assume $u_0\in\partial_x^\infty H^\infty(\RR)$, so that $u,v$ are smooth by Remark \ref{rk:time_smoothness_and_de_x_infty}. To show the a priori estimate, we differentiate with respect to $t$ the equation defining $\gamma_0$ \eqref{eq:definition-of-gamma_0} and obtain, after substituting the second equation in \eqref{eq:systemForlittlev-miura+mKP-II},
    \begin{align*}
    \fr d{dt}\gamma_0(t)\int_{\RR}\rho_{\gamma_0}v_x\,dx\,dy&=-\int_{\RR}\rho_{\gamma_0(t)}v_t\,dx\,dy\\
    &=\int_{\RR} \rho_{\gamma_0(t)}[(4v_{xx}+4v^3+12vv_x-6uv-3u_x)_x-3u_y]\,dx\,dy.
    \end{align*}
    Subtracting the number $4$,
    \[
    \frac d{dt}\gamma_0(t)-4=\frac{\int_{\RR} \rho_{\gamma_0(t)}[(4v_{xx}+4v^3+6(v^2)_x-4v-6uv-3u_x)_x-3u_y]\,dx\,dy}{\int_{\RR}\rho_{\gamma_0}v_x\,dx\,dy}.
    \]    
    As in the proof of Lemma \ref{lemma:change-of-variables-c-alpha_0-gamma_0}, we note that the denominator is positive and bounded away from zero, so it can be ignored. We now proceed as in the proof of \eqref{eq:z-029} in Lemma \ref{lemma:primitive-solutions-time-dependent-are-nice}. This time, $v$ is not small in $L^3$ of a given ball, because of the contribution from the modulated kink. Nevertheless, plugging $v=\tanh_\alpha$ yields
    \[
    4v_{xx}+4v^3+6(v^2)_x-4v=0,
    \]
    so after considering $v=\wv+\tanh_\alpha$, it is straightforward to show that
    \[
    |d\gamma_0/dt(t)-4|\lesssim \|\wv\|_{L^1(B)}+\|\wv\|_{L^2(B)}^2+\|\wv\|_{L^3(B)}^3+\|u\|_{L^{3/2}(B)}\|v\|_{L^3(B)}+\|u\|_{\Hcrit},
    \]
    where $B=B_1((\gamma_0(t),0))$ is the support of $\rho_{\gamma_0(t)}$. Since $\alpha$ is the same given by Theorem \ref{theorem:theorem_1}, the bounds in the Theorem imply that
    \[
    \|\wv(t)\|_{L^3(B)}\lesssim \|u(t)\|_{\Hcrit},
    \]
    and the bound follows analogously as for estimate \eqref{eq:z-029} by integrating in time, using that
    \[
    \|u\|_{L^2_\unif((0,\infty)\times \RR))}+\|u\|_{L^\infty_t\dot H^{-\aha,0}}\lesssim \|u_0\|_{\Hcrit}.
    \]
    
    Finally, we look at the last part of the Theorem. The $L^2_\ph$ estimate is a direct consequence of Corollary \ref{cor:almost-conservation-L2-norm-around-line-soliton} and the conservation of the $L^2$ norm for solutions of \eqref{eq:KP-II}:
    \[
    \|u(t)\|_{L^2(\RR)}=\|u_0\|_{L^2(\RR)}.
    \]
    Assume first that $u_0\in\partial_x H^\infty(\RR)$, so that $u\in C([0,\infty),\partial_x H^\infty(\RR))$, and is uniformly small in $\Hcrit$. We first show the following claim.
    
    \begin{claim}
    We have $v_x-\sech^2(x-4t-a)\in C([0,\infty),L^2(\RR))$ for some $a\in\R$. In particular, $\wu-\ph(x-4t-a)\in C([0,\infty),L^2(\RR))$.
    \end{claim}
    
    \begin{claimproof}[of the Claim]
    By Corollary \ref{cor:estimates-for-wtv-part-2}, since the curve $t\mapsto u(t)\in \partial_xL^2(\RR)$ is continuous, the corresponding solutions $v^\pm(t)\in L^3(\RR)\pm 1$ of equation \eqref{eq:miuraEquation} are such that
    \[
    \wv^\pm:=v^\pm\mp 1\in C_b([0,\infty),L^2(\RR)\cap \de_x L^6(\RR)).
    \]
    This means that for each $t\geq 0$, $\wv^\pm(t)$ have well-defined $x$-antiderivatives in $L^6(\RR)$. This is equivalent to $v^\pm(t)$ having well-defined antiderivatives in $L^6(\RR)+(\pm x+y)$, which have to solve equation \eqref{eq:primitiveEquation} with datum $u(t)$. We will call these primitives $W^\pm(t)$. By Lemma \ref{lemma:primitive-solutions-time-dependent-are-nice}, $W^\pm(t)-V^\pm(t)$ is a constant depending only on time. In particular, since $(u,V^\pm)$ solve system \eqref{eq:systemForCapitalV} by construction (note that the leading part of the second equation gives $V^\pm_t\approx -4(V^\pm_x)^3=-4(v^\pm)^3\approx \mp 4$), there must exist $a^\pm\in\R$ such that
    \[ V^\pm\mp c=W^\pm\mp4t+a^\pm. \]
    Now set $a=a^--a^+$. Call $\wW^\pm:=W^\pm-(\pm x+y)\in C([0,\infty),L^6(\RR))$. We have
    \begin{equation}\label{eq:z-006}
        \begin{aligned}
            v&=\de_x\log\left(e^{W^+-a-4t}+e^{W^-+a+4t}\right)\\
            &=\de_x\log\left(e^{\wW^++(x-4t-a)}+e^{\wW^--(x-4t-a)}\right)\\
            &=\tanh\circ\nu+(\eta^+\!\!\circ\nu)\wv^++(\eta^-\!\!\circ\nu)\wv^-,
        \end{aligned}
    \end{equation}
    where $\eta^\pm(x)=(1\pm\tanh(x))/2$ as usual, and
    \als{
    \nu&=\frac 12(V^+-V^-)-c\\
    &=\fr 12 (\wW^+-\wW^-)+(x-4t-a)
    }
    (note that $\de_x\log(y)=0$). Taking a further $x$-derivative, we obtain
    \begin{equation}\label{eq:z-005}
        \begin{aligned}
            \partial_xv-\sech^2(x-4t-a)&=(\sech^2\!\circ\hspace{2pt}\nu-\sech^2(x-4t-a))\\
            &\QQQ+(\etapcnu)\wv^+_x+(\etamcnu)\wv^-_x+\sech^2\!\circ\hspace{2pt}\nu\cdot\,(\wv^+-\wv^-).
        \end{aligned}
    \end{equation}
    Now, since $u\in C([0,\infty),\partial_x H^\infty(\RR))$, by the continuity of the data to solution map in Corollary \ref{cor:estimates-for-wtv-part-2} in all the involved function spaces, we deduce that $\nu\in C([0,\infty)\times\RR)$. We also have the estimate
    \begin{align*}
    |\sech^2(\nu(x,y))-\sech^2(x-4t-a)|&\leq \aha\left|\int_0^1 (\sech^2)_x(x-4t-a+s(\wW^+-\wW^-)/2)ds\right|\\
    &\QQQ\times|\wW^+-\wW^-|\\
    &\lesssim\braket{x-4t-a}^{-1}|\wW^+-\wW^-|.
    \end{align*}
   Since $\wv^\pm,\wv^\pm_x\in C([0,\infty),L^2(\RR))$, and since $\braket{x-4t-a}^{-1} \wW^\pm\in C([0,\infty),L^2(\RR))$ thanks to the weighted estimate in \eqref{eq:estimates-for-wv-lambdaiszero-H^-1} in Corollary \ref{cor:estimates-for-wtv-part-2}, the right-hand side of \eqref{eq:z-005} lies in $C([0,\infty),L^2(\RR))$ by the dominated convergence theorem. This concludes the proof of the claim.
   \end{claimproof}

    By the estimates of Lemma \ref{cor:estimates-for-wtv-part-2} and by repeatedly differentiating the explicit formula for $v$ in \eqref{eq:z-006}, one verifies that $v(t)-\tanh(x-4t-a)$ is bounded in $H^k(\RR)$ locally in time for all $k\geq 0$, which implies that $\wu-\ph(x-4t-a)\in C([0,\infty),H^\infty(\RR))$ by the above claim. From this, and since $\wu$ solves KP-II distributionally,
    \[
    \wu(t,x+4t,y)-\ph\in C([0,\infty),H^\infty(\RR))
    \]
    and solves equation \eqref{eq:perturbedLineSoliton} distributionally, with $\alpha\equiv 0$. Due to the high regularity, the above solution has to coincide with the solution given by the well-posedness theory of Proposition \ref{prop:KP-IIglobalWellPosednessAroundSlackSoliton} with $\alpha\equiv 0$, as it can be seen via a standard use of energy estimates for the difference of two solutions. In particular, by Proposition \ref{prop:KP-IIglobalWellPosednessAroundSlackSoliton}, it holds $\wu(t,x+4t,y)-\ph\in X_T^{1/2+\eps,b_1,0}$ for all $T>0$.\\
    In addition, as we noted in the proof of Corollary \ref{cor:almost-conservation-L2-norm-around-line-soliton}, $v_0$ coincides with one of the solutions in Proposition \ref{proposition:existenceOfAnEternalSolution}, with datum $u_0$. So, let $\beta=\beta(y)$ be the shift given by the Proposition, and let $\bar\beta:=\beta*\rho$ a regularization of $\beta$. It holds
    \als{
        \|\bar\beta_y\|_{H^2(\R_y)}+\|\beta-\bar\beta\|_{H^1(\R_y)}&\lesssim \|\beta_y\|_{L^2(\R_y)}\\
        &\lesssim \|u_0\|_{L^2(\RR)},\\[0.5em]
        \|\wu_0-\ph_\beta\|_{L^2(\RR)}&\leq \|u_0\|_{L^2(\RR)}+2\|(v-\tanh_\beta)_x\|_{L^2}\\
        &\lesssim \|u_0\|_{L^2(\RR)},
        }
    where the last inequalities in both estimates come from Proposition \ref{proposition:existenceOfAnEternalSolution}. Furthermore,
    \als{
        \|\wu_0-\ph_{\bar\beta}\|_{L^2(\RR)}&\leq \|\wu_0-\ph_{\beta}\|_{L^2(\RR)}+\|\ph_x\|_{L^2(\R_x)}\|\beta-\bar\beta\|_{L^2(\R_y)}\\
        &\lesssim \|u_0\|_{L^2(\RR)}.
        }
    The function $\wu(t,x+4t,y)-\ph_{\bar\beta}$ is a solution of \eqref{eq:perturbedLineSoliton} with $\alpha=\bar\beta$, and it also coincides with the one given by the well-posedness theory in this setting since $\bar\beta_y\in H^1(\R_y)$. By Remark \ref{rk:bourgain-space-norm-depends-monotonically-on-the-norm-of-alpha}, it holds the estimate
    \als{
    \|\wu(t,x+4t,y)-\ph_{\bar\beta}\|_{X_T^{1/2,b_1,0}}&\leq C(\|\wu_0-\ph_{\bar\beta}\|_{L^2(\RR)},\|\bar\beta_y\|_{H^1(\R_y)},T,0)\\
    &\leq C(\|u_0\|_{L^2(\RR)\cap \Hcrit},T)
    }
    for some function $C$ that is non-decreasing in both variables.
    
    For general $u_0\in\Hcrit\cap L^2(\RR)$ small, we consider an approximating sequence of data $\partial_xH^\infty(\RR)\ni u_{0,n}\to u_0$ in $\Hcrit\cap L^2(\RR)$ such that $\|u_{0,n}\|_{L^2}\leq 2\|u_0\|_{L^2}$. Let $\wu_n$ be the corresponding solutions of KP-II with initial data $\wu_{0,n}:=\wu_n|_{t=0}=\samB(u_{0,n},\gamma_0)$, with corresponding phase shifts at time zero $\beta_n$ given by Proposition \ref{proposition:existenceOfAnEternalSolution}, and note that $\|\bar\beta_{n,y}\|$ is uniformly bounded in $n$ thanks to the above estimates. Moreover, it is straightforward to show that $|\beta_n(0)|\lesssim 1+|\gamma_0|$. By the estimate \eqref{eq:c_is_close_to_alpha_0_and_gamma_0_for_small_c} in Lemma \ref{lemma:change-of-variables-c-alpha_0-gamma_0}, the sequence $(c_n)_{n}$ of real numbers such that $\wu_{0,n}=u_{0,n}-2\partial_x\eleV^{(-1,1)}(u_{0,n},(c_n,-c_n))$ is bounded. By reducing $(u_{0,n})$ to a subsequence, we can assume $c_n$ converges to $c\in\R$. By the continuity of the maps $\dtsmV$ in Theorem \ref{theorem:theorem_1}, $\eleV$ in Proposition \ref{prop:kink-addition_map_eleV}, and by the global well-posedness of \eqref{eq:KP-II} in $\Hcrit$ of Theorem \ref{theorem:global-well-posedness-KP-II-HHK09}, we have $\wu_{0}=u_{0}-2\partial_x\eleV^{(-1,1)}(u_{0},(c,-c))$ (that is, $c$ is the same constant as the one at the beginning of the proof), and
    \[
    \wu_n\to \wu\text{ in }\mathscr D'((0,T)\times\RR).
    \]
    Since $(\beta_{n,y})_n$ is bounded in $L^2(\R_y)$ and $\beta_n(0)$ is bounded in $\R$, up to extracting a new subsequence, we can assume $\beta_n\to \beta$ uniformly by Arzelà--Ascoli, so $\bar\beta_n\to \bar\beta$ in $L^\infty$. By continuity, this $\beta$ is the same shift assigned to the pair $(u_0,v_0)$ given by Proposition \ref{proposition:existenceOfAnEternalSolution}, although we do not need this fact here. The above implies
    \[
    \wu_n(t,x+4t,y)-\ph_{\bar\beta_n}\to \wu(t,x+4t,y)-\ph_{\bar\beta}\quad\text{in }\,\mathscr D'((0,T)\times\RR).
    \]
    By weak-$*$ compactness, the limit lies in $X_T^{1/2,b_1,0}$ due to the uniform bound
    \[
    \|\wu_n(t,x+4t,y)-\ph_{\bar\beta_n}\|_{X_T^{1/2,b_1,0}}\lesssim_T C(2\|u_0\|_{L^2(\RR)\cap \Hcrit},T)
    \]
    proved above for smooth solutions. By the uniqueness statement in Proposition \ref{prop:KP-IIglobalWellPosednessAroundSlackSoliton},
    \[
    \wu(t,x+4t,y)-\ph_{\bar\beta}
    \]
    also agrees with the solution given by the well-posedness theory as in the Proposition, with $\alpha=\bar\beta$.
\end{proof}

\section[The range of the soliton addition map]{The range of the soliton addition map}\label{sec:range}

By Theorem \ref{theorem:theorem_1}, we have constructed a B\"acklund transform
$$ \samB\colon (u,\gamma_0)\mapsto \wu $$
for the KP-II equation in Definition \ref{def:soliton_addition_map_and_L^2_phi(RR)}, which acts on small functions $u\in \Hcrit$ and a parameter $\gamma_0\in\R$ that determines the position of the line soliton at $y=0$. Our aim is to characterize a sufficiently large subset of the image of $\samB$ through $B_{\eps_0}^{\dot H^{-\aha,0}}(0)\times \R$, where $\eps_0$ is the smallness constant. In this section we will prove Theorem \ref{theorem:range_of_samB_contains_manifold_of_cod_1}, which characterizes the intersection between the range and a small ball in a suitable weighted space at $L^2$ regularity. The Theorem implies codimension-1 modulational stability of the line soliton.

\subsection{Premise}\label{subsec:change-of-var-w-psi}

Given a datum $u$ and the output $\wu$ of the B\"acklund transform, the two functions solve the system of equations given by the Miura map(s)
\begin{equation}\label{eq:system_miura_maps_backlund_section_range}
    \left\{\begin{aligned}
        &v_y-v_{xx}=(v^2)_x-u_x,\\
        &v_y+v_{xx}=(v^2)_x-\wu_x,\\
        &u=\wu+2v_x,
    \end{aligned}\right.
\end{equation}
a rigorous way of writing the system
\begin{equation*}
    \left\{\begin{aligned}
        M_+(v)=\wu,\\
        M_-(v)=u.
    \end{aligned}\right.
\end{equation*}
In this section, we want to solve for $v$, given $\wu$ a suitable perturbation of the line soliton, to derive sufficient conditions under which $\wu$ lies in the range of the soliton addition map. We set
\[
g:=\wu-\ph,\qquad w:=v-\tanh(x),
\]
where $\ph(x)=-2\sech^2(x)$ is the non-modulated line soliton, and perform some reductions. The second equation in \eqref{eq:system_miura_maps_backlund_section_range} becomes \begin{equation}\label{eq:equation_for_w_unvariabled_section_range}
w_y+w_{xx}-2(\tanh w)_x=(w^2)_x-g_x,
\end{equation}
and since we removed the leading parts $\ph$ and $\tanh$, we want to find solutions $w$ that approach zero at infinity, for given perturbations $g$ that are localized and smooth enough. Recall the reflection symmetry
\[
\Rcal f(x,y):=f(x,-y).
\]
By the change of variables
\begin{equation*}
    \begin{aligned}
        w&\mapsto -\Rcal w=:\w,\\
        g&\mapsto \Rcal g=:\g,
    \end{aligned}
\end{equation*}
the above equation is equivalent to
\begin{equation}\label{eq:perturbed-line-soliton-adjoint-Miura-map}
    \w_y-\w_{xx}+2(\tanh \w)_x=(\w^2)_x-\g_x.
\end{equation}
We will thus look for a solution to \eqref{eq:perturbed-line-soliton-adjoint-Miura-map} with $\g\in\Hcrit$, that is equivalent to $g\in\Hcrit$. Since $u=g+2w_x$, we will look for additional assumption on $g$ to make $w_x\in\Hcrit$ small. By the Cole--Hopf transformation $\w\mapsto e^{\int \w \,dx}=:\psi$, we can reduce the problem to that of finding positive solutions to
\begin{equation}\label{eq:psi-from-g}
    \psi_y-\psi_{xx}+2\tanh\psi_x=-\g\psi.
\end{equation}

\begin{remark}[The codimension-1 condition and the Lax eigenfunctions]
Before we continue, here is a brief explanation on why we expect that the range of $\samB$ has codimension 1 in a suitable vector space. By the property of the Cole--Hopf transformation, the function
\[
\Psi(x,y):=e^y\cosh(x)\psi^{-1}(x,-y)
\]
is a $0$-eigenfunction of the Lax operator $L_u=\de_y-\de_x^2+u$ with potential $u$, since
\als{
\de_x\log(e^y\cosh\Rcal\psi^{-1})&=\tanh-\Rcal \w\\
&=v
}
is a solution of \eqref{eq:miuraEquation}. Since $\de_x\log(\Psi)=v$, $L_u\Psi=0$, and since $\Psi$ has the expected asymptotics at infinity, it is reasonable to expect that $\Psi$ coincides, up to a positive scalar, with the eigenfunction of $L_u$ in Theorem \ref{theorem:theorem_1} (called `$\psi$' there), if $v$ is a solution that comes from the Theorem. By looking at the properties of $v$ in Theorem \ref{theorem:theorem_1}, in combination with Corollary \ref{cor:estimates-for-wtv-part-2}, it can be seen that if $u\in \Hcrit$ is localized and regular enough, it holds
\[
\Psi(x,y)=e^y\cosh(x)m(x,y),
\]
where $m$ converges to a positive constant at infinity in all directions. On the other hand, it is evident that for localized $g$, solutions of equation \eqref{eq:psi-from-g} with $\g=\Rcal g$ that approach a constant $a^->0$ at $y\to-\infty$, in general will converge to a different constant $a^+>0$ as $y\to+\infty$ due to the transport term. The condition that allows $g$ to lie in the range of $\samB$ is precisely that $a^-=a^+$.
\end{remark}

\subsection{Linearization of the problem in the Cole--Hopf variables}\label{subsec:kernel-of-linearized-psi-from-g}

To study equation \eqref{eq:psi-from-g}, we look at the linearized equation around the constant solution $\psi= 1,\g=0$:
$$ \psi_y-\psi_{xx}+2\tanh\psi_x=f. $$
By Proposition \ref{prop:explicit-kernels}, a solution $\psi$ is given by $\psi=Tf+1$, where the operator $T$ is an integral operator with explicit kernel that inverts $\partial_y-\partial_{xx}+2\tanh\partial_x$:
\begin{equation}\label{eq:operator-T-transport-outwards}
T=K_{\rm tr+}:=\Gamma^+\Mult_{\eta^+}+\Gamma^-\Mult_{\eta^-}+\frac{1}{2}\partial_x^{-1} (\Gamma^--\Gamma^+)\Mult_{\sech^2},
\end{equation}
where we recall that $\Gamma^{\pm}=(\partial_y\pm 2\partial_x-\partial_x^2)^{-1}$, $\Mult_{\eta^\pm}$ is the multiplication by $\eta^\pm(x)$, and $\Mult_{\sech^2}$ is the multiplication by $\sech^2(x)$, as in Definition \ref{def:heat-operators-in-appendix}. The operator $T$ is therefore a sum of four terms: two (tilted) heat operators composed with multiplication operators by bounded smooth functions, and two antiderivatives of (tilted) heat operators composed on the right with the multiplication operator by $\sech^2(x)$. The difference $\partial_x^{-1}(\Gamma^+-\Gamma^-)$ is a convolution operator with the function
\begin{align}
\demu(\Gamma^--\Gamma^+)(x,y)&=\demu\Gamma(x+2y,y)-\demu\Gamma(x-2y,y)\\
&=\int_{x-2y}^{x+2y} \Gamma(x',y)dx,
\end{align}
whose kernel decays to zero for fixed $y$ and is identically zero for negative $y$, but it converges to the constant $1$ for positive $y$.

\subsection{Solving the equation \texorpdfstring{$M_+(v)=\wu$}{M+(v)=u}}

We turn to the study of equation \eqref{eq:psi-from-g}.

\begin{proposition}\label{prop:well-posedness-of-psi-to-g}
    Let $\g\in L^1(\R^2)\cap L^2(\R^2)$. There exists a unique solution $\psi\in L^\infty(\RR)$ to equation \eqref{eq:psi-from-g} such that $\|\psi-1\|_{L^\infty(\R\times (-\infty,y])}$ goes to zero\footnote{The kernel of the operator $\partial_y-\partial_{xx}+2\tanh\partial_x$ in $L^\infty$ coincides with the space of constant functions, this is why we need such a condition. Since the equation is linear, the Proposition shows that the solution is unique up to a multiplicative constant.} as $y\to-\infty$. It holds $\psi\in C_b(\RR)$, $\inf_{x,y\in\R}\psi(x,y)>0$, and
    \[
    \|\psi\|_{L^\infty(\RR)}+\|1/\psi\|_{L^\infty(\RR)}\lesssim \exp\left(C\|\g\|_{L^1(\RR)\cap L^2(\RR)}\right),
    \]
    \[
    \|\psi-1\|_{L^\infty(\RR)}\lesssim \exp\left(C\|\g\|_{L^1(\RR)\cap L^2(\RR)}\right)\|\g\|_{L^1(\RR)\cap L^2(\RR)}.
    \]
    For a universal $C>0$. The data-to-solution map is analytic.
\end{proposition}
\begin{proof}
    Let $T=(\partial_y-\partial_x^2+2\tanh\partial_x)^{-1}$ be the integral operator described above. By Lemma \ref{lemma:operator_l_tr_+_g_bounds}, the operator $T$ is bounded from $Z:=L^1(\RR)\cap L^2(\RR)$ to $L^\infty(\RR)$, and the Banach space $X:=T(Z)\hookrightarrow L^\infty(\RR)$ is such that
    \[
    f\in X\quad\implies\quad f|_{\R\times (-\infty,y_0]}\in C_0(\R\times(-\infty,y_0])\,\,\forall y_0\in\R
    \]
    (we defined $C_0$ in Subsection \ref{subsec:notations}). The operator $T$ is a right inverse of $L=\partial_y-\partial_x^2+2\tanh\partial_x$, whose kernel in $L^\infty(\RR)$ is exactly the $1$-dimensional space of constant functions. Our assumptions then imply $\psi\in T(Z)+1$, and equation \eqref{eq:psi-from-g} is equivalent to
    \[
    \psi-1=-T(\psi \g).
    \]
    Since, as we recall later, the operator $L_{\g}:=\partial_y-\partial_x^2+2\tanh\partial_x+\g$ is invertible from $T(Z)$ to $Z$, the uniqueness follows by necessity, since by \eqref{eq:psi-from-g}, it has to hold
    \[
    \psi=-(L_{\g})^{-1}\,\g+1.
    \]
    
    Now, let
    \[N:=\left\{\phi\in T(Z)\,|\, \inf_{\RR}\phi>-1\right\},\] 
    which is a convex, open subset of $T(Z)$. Consider the map 
    \begin{equation*}
        \begin{gathered}
            F\colon N\to Z\\
            \phi\mapsto -\frac{1}{1+\phi}(\de_y-\de_x^2+2\tanh\de_x)\phi.
        \end{gathered}
    \end{equation*}
    It is well-defined and analytic due to the restriction $\inf_{\RR}\phi>-1$ and the boundedness of $T$ from $Z$ to $L^\infty(\RR)$. The Proposition will follow after establishing the invertibility and additional properties of the map $F$ by setting $\psi:=F^{-1}(\g)+1$.
    
    The differential of $F$ is
    \begin{align*}
        D F(\phi)\cdot \dot\phi&=-\frac{1}{1+\phi}(\de_y-\de_x^2+2\tanh\de_x)\,\dot\phi+\frac{\dot\phi}{(1+\phi)^2}(\de_y-\de_x^2+2\tanh\de_x)\,\phi=\\
        &=-\frac{1}{1+\phi}(\de_y-\de_x^2+2\tanh\de_x+\g)\,\dot\phi,
    \end{align*}
    where in the last equality, $\g:=F(\phi)$ is meant as a multiplication operator. By Lemma \ref{lemma:operator_l_tr_+_g_bounds} part (b), $T_{\g}:=(\partial_y-\partial_x^2+2\tanh\partial_x+\g)^{-1}$ is well-defined and bounded from $Z$ to $T(Z)$, in particular $F$ has invertible differential everywhere. We claim that $F$ is also surjective, and a preimage of $\g\in Z$ is given by
    \[
    F^{-1}(\g)=-T_{\g}\,\g.
    \]
    In fact, $-T_{\g}\,\g\in T(Z)$ by Lemma \ref{lemma:operator_l_tr_+_g_bounds}, and we claim that $\inf \phi >-1$. For that it is enough to consider $\psi:=1+\phi$, and note that the function $\wt\psi:=1/\psi$ solves
    \[
    \wt\psi_y-\wt\psi_{xx}+2\tanh\wt\psi_x=-2\psi\wt\psi^2+\g\wt\psi\leq \g\wt\psi.
    \]
    Since $\psi\in T(Z)+1$, the functions $\psi,\wt\psi$ converge to the constant $1$ uniformly as $y\to-\infty$, and since the kernel of $T$ is non-negative, we can follow the same steps as in \eqref{eq:z-008} replacing the $L^\infty$ norm with the supremum of $\wt\psi$, so after an approximation and a continuity argument, the a priori estimate
    \[
    \sup_{\RR}1/\psi\lesssim \exp\left(C\|\g\|_{L^1(\RR)\cap L^2(\RR)}\right)
    \]
    is proved. Since $\inf \psi >0$, it follows that $\psi$ solves $\psi=-T(\psi \g)+1$, i.e., $F(\psi-1)=\g$. Thus, $F$ is an analytic diffeomorphism.
    
    The same a priori estimate for $\psi$ is proved in complete analogy, since $\psi$ solves equation \eqref{eq:psi-from-g}. Finally, the estimate on $\psi-1$ follows from Lemma \ref{lemma:operator_l_tr_+_g_bounds} part (b) and the identity $\psi-1=-T_{\g}\g$.
\end{proof}
    We recall that setting $w:=-\Rcal\de_x\log(\psi)$, we get a solution of \eqref{eq:equation_for_w_unvariabled_section_range}, and we want to understand what to assume on $g=\Rcal\g$ to have $w_x\in\Hcrit$.

\subsection{The functional \texorpdfstring{$\Phi$}{Phi}}

There is one mechanism that prevents $w_x$ from being in $\Hcrit$, even for extremely localized $g$: \emph{mass unbalance between the left and right regions} delimited by the line soliton. To intuitively illustrate this phenomenon, assume for simplicity that $g\in\partial_x\mathscr D(\RR)$ (so that it also lies in $\Hcrit$). Let $\supp(g)\subset [-R,R]^2$. The diffusive and transport natures of equation \eqref{eq:psi-from-g} (with $\g=\Rcal g$) suggest that, for large $y\gg R$, the solution $\psi(\cdot,y)$ will converge to a constant on a growing time interval $[-2y+O(\sqrt y),2y+O(\sqrt y)]$, or more explicitly,
$$ \psi(x,y)\approx 1+ c\int_{-\infty}^x \Gamma(x'+2y,y)-\Gamma(x'-2y,y) \,dx' $$
for some $c\in\R$ (up to translations in the $y$ variable). If this constant is not zero, we argue that $w\not\in L^3(\RR)$. The reason is that from the above heuristics, $\psi_x$ will behave like
$$ \psi_x\approx c\,(\Gamma(x+2y,y)-\Gamma(x-2y,y)), $$
which cannot lie in $L^3(\RR)$ because the heat kernel simply doesn't (it belongs to $L^{3,\infty}(\RR)$). Then we simply note that $w=\frac{\psi_x}{1+\psi}$, so the same holds for $w$. It is then not possible to have $w_x\in\Hcrit$, since we already know from section \ref{sec:space} that we would have $w\in L^3(\RR)$. These heuristics suggest that in order to aim at $w_x\in \Hcrit$, we need some cancellation condition to ensure that $c=0$ in the above asymptotics. In other words, $\psi$ should decay to zero as $y\to+\infty$.

To rephrase this condition in a convenient way, we note that the equation
$$ \psi_y-\psi_{xx}+2\tanh \psi_x=0 $$
preserves the hyperplane $\{\sech^2\}^\perp$ of $L^2(\R)$, so a natural condition is to impose that 
$$\aha\int \psi(x,y)\sech^2(x)\,dx\to 1$$
as $y\to+\infty$. For localized $\g$ and solutions $\psi$ to \eqref{eq:psi-from-g}, by testing the equation against $\sech^2(x)$, the latter condition is equivalent to
\begin{equation}\label{eq:orthogonality-condition-g}
    -\aha\int_{\RR} \sech^2(x)\g(x,y)\psi(x,y)\,dx\,dy=0.
\end{equation}
The above serves as a motivation for the following definition and the subsequent analysis.
\begin{definition}\label{def:Phi}
We define the functional
$$ \Fun:L^1(\RR)\cap L^2(\RR)\to \R, $$
\als{
\Fun(\g):=&-\aha\int_{\RR} \sech^2(x)\g(x,y)\psi(x,y)\dxdy\\
=&\lim_{y\to+\infty} \aha\int_{\R}\sech^2(x)\psi(x,y)\dx -1,
}
where $\psi\in L^\infty(\RR)$ is the unique solution to $\eqref{eq:psi-from-g}$ that converges uniformly to $1$ as $y\to -\infty$ given by Proposition \ref{prop:well-posedness-of-psi-to-g}. The second equality is verified by integrating \eqref{eq:psi-from-g} against $\sech^2(x)$.
\end{definition}

\begin{remark}\label{rk:Phi_is_invariant_under_reflections}
The above functional is analytic on $L^1\cap L^2(\RR)$. Its differential at $\g= 0$ is
$$ D\Phi(0)\cdot z=-\aha\int_{\RR}\sech^2(x)z(x,y) \dxdy,$$
(so the requirement $\Phi(\g)=0$ is somewhat transversal to the requirement $\g\in\Hcrit$). It is also invariant under the reflection $\Rcal$. In fact, if $\psi$ solves \eqref{eq:psi-from-g} and $\phi$ solves
\[
\phi_y-\phi_{xx}+2\tanh\phi_x=-(\Rcal h)\phi,
\]
and are both given by Proposition \eqref{prop:well-posedness-of-psi-to-g}, it is easily verified that
\[
\frac{d}{dy} \left[ \aha \int_{\R}\sech^2 \psi(\Rcal\phi)\dx \right]=0,
\]
which implies
\[
\lim_{y\to+\infty} \aha\int_{\R}\sech^2(x)\psi(x,y)\dx = \lim_{y\to+\infty} \aha\int_{\R}\sech^2(x)\phi(x,y)\dx
\]
since both solutions converge uniformly to $1$ as $y\to-\infty$.
\end{remark}

As remarked in the introduction, the above functional appears naturally in the scattering transform of \eqref{eq:KP-II} for perturbations of the line soliton.

It is clear from the above, although not rigorously proved, that if $g$ is a `good perturbation' that falls in the image of our B\"acklund transform, it must hold $\Phi(\Rcal g)=0$, that is, $\Phi(g)=0$. In the following, we look for additional conditions on $g$ to prove the reverse implication.

\subsection{Estimates in a parabolic Hardy space}
The property $\Phi(\g)=0$ appears naturally when writing $\psi$ as an integral operator applied to $-\g\psi$. Assume that we have
\[
\g\in L^1(\RR)\cap L^2(\RR)\cap\Hcrit
\]
such that $\Phi(\g)=0$. For the solution $\psi\in L^\infty(\RR)$ to equation \eqref{eq:psi-from-g} given by Proposition \ref{prop:well-posedness-of-psi-to-g}, it holds
$$ \psi=T(-h(1+\psi))+1, $$
with $T=(\partial_y-\partial_x^2+2\tanh\partial_x)^{-1}$ as in \eqref{eq:operator-T-transport-outwards}. Since we want $w_x=-\Rcal(\psi_x/\psi)_x\in\Hcrit$, by fractional chain rule, we need $|D_x|^{3/2}\psi\in L^2(\RR)$. When applying $T$ to $-h\psi\in L^1\cap L^2(\RR)$, the contribution from the first two terms in \eqref{eq:operator-T-transport-outwards} have the desired bounds by the estimate
\[
|D_x|^{3/2}\Gamma^\pm \colon L^{3/2}(\RR)\mapsto L^2(\RR),
\]
which is covered by Proposition \ref{prop:spacetime_estimate_heat_general}. To control the remaining term
\[
-\aha|D_x|^{1/2}(\Gamma^--\Gamma^+)(\sech^2 h\psi)
\]
in $L^2$, scaling suggests that the argument of $\Gamma^\pm$ must be in a space that scales like $L^1(\RR)$. This introduces a complication, because it means that we need $\sech^2 h\psi$ to lie in a Hardy space $\Hcal^1$ adapted to the operators $\Gamma^\pm$. In particular we need it to have zero mean, hence one further motivation for the condition $\Phi(g)=0$.

What follows is a brief (pre -)dual treatment to that of the BMO spaces in Subsection \ref{subsec:exact_solutions}, and all the results can be found in the same references.

\begin{definition}[\cite{coifman-weiss-1977-extension-hardy-spaces}]\label{def:Hardy_space_H1}
    Let $(X,d,\mu)$ be a doubling metric measure space of homogeneous type.
    \begin{itemize}[itemsep=0.5ex,topsep=0ex]
        \item Let $1<q\leq\infty$. A $q$-atom $a:X\to \R$ is a measurable function such that\footnote{When $\mu(X)<\infty$, one often assumes the constant $a=\mu(X)^{-1}$ to be an atom as well, which results in adding the constant functions to the space $\Hcal^1(X)$. As we did in the definition of $\BMO$, we give a definition that does not depend on whether $\mu(X)$ is finite or infinite.}:
        \setlist{nolistsep}
        \begin{enumerate}[noitemsep]
            \item the support of $a$ is contained in a ball $\Bcal_r(x_0)$,
            \item $ \mu(\Bcal_r(x_0))^{1/q'}\|a\|_{L^q(X)}\leq 1 $,
            \item $\int_X a\,d\mu=0$.
        \end{enumerate}
        
        \item The Hardy space $\Hcal^1(X)$ is the vector space of functions $f$ such that there exists a sequence $\lambda\in \ell^1(\N)$ and a sequence of $\infty$-atoms $a_0,a_1,\dots$ such that it holds a.e.
        \[f=\sum_{j\in\N} \lambda_j a_j.\]
        \item We equip $\Hcal^1(X)$ with the norm
        \[
        \|f\|_{\Hcal^1(X)}:=\inf_{\lambda,(a_j)_j} \|\lambda\|_{\ell^1},
        \]
        where the infimum is taken over all representations of $f$ as in the previous point.
    \end{itemize}
    Consider the space $(X,d,\mu)=(\RR,d_\lambda,\mu)$ with the tilted parabolic metric $d_\lambda$ as in Definition \ref{def:parabolic_BMO}, and $\mu$ the Lebesgue measure. We denote the associated Hardy space $\Hcal^1(X)$ by $\PHux\lambda(\RR)$.
\end{definition}

\begin{remark}\label{rk:q-atoms,H1-BMO_duality}
    It follows from the definition that $\Hcal^1(X)\subset L^1(X)$ with continuous embedding. It is well-known that $\Hcal^1(X)$ is a Banach space, and $\BMO(X)$ is the dual of the Hardy space $\Hcal^1(X)$, where the pairing is given by the integral of the product (extended by density). Finally, we will use the nontrivial fact (see \cite{coifman-weiss-1977-extension-hardy-spaces}*{Theorem A}) that in the definition of $\Hcal^1(X)$ we can equivalently consider $q$-atoms instead of $\infty$ atoms, yielding the same vector space and the same norm up to equivalences.
\end{remark}

We will need two simple Lemmas.

\begin{lemma}[Decay + integrability + zero mean, imply $\Hcal^1$]\label{lemma:Lp+weight/decay+zero_mean_implies_Hardy_H1}
    Let $(X,d,\mu)$ a doubling metric measure space with doubling constant $A$, and let $1<p\leq\infty$, $\eps>0$, and $x_0\in X$. It holds
    \[
    \|f\|_{\Hcal^1(X)}\lesssim_{\mu,d,x_0,p,\eps} \|wf\|_{L^p(X)}
    \]
    for all $f$ such that $\int_X f\,d\mu=0$, where $w(x)=(1+ d(x_0,x))^{D/p'+\eps}$, $D=\log_2A$.
    
\end{lemma}
If a function has zero mean, decays slightly better than $L^1$, and is slightly more integrable than $L^1$, then it lies in $\Hcal^1$. Note that the norm on the right-hand side controls the $L^1$ norm, so the integral of $f$ is well-defined. Note as well that it holds $D\geq 1$ as long as $X$ contains more than one point, see \cite{Soria_Tradacete_2019_The_least_doubling_constant}*{Theorem 3.1}.

\begin{lemma}\label{lemma:heat_kernel_and_Hardy_space_H1}
    The operator $\Gamma^{(c)}:=(\partial_y+c\de_x-\partial_x^2)^{-1}$ satisfies
    \[
    \||\de_x|^{-1/2}\Gamma^{(-2\lambda)}f\|_{L^6(\RR)}+\||\partial_x|^{1/2}\Gamma^{(-2\lambda)}f\|_{L^2(\RR)}\lesssim \|f\|_{\PHux\lambda(\RR)}.
    \]
\end{lemma}
The proofs use classical arguments, and we move them to Appendix \ref{subsec:proofs}.

\subsection{Proof of Theorem \ref{theorem:range_of_samB_contains_manifold_of_cod_1}}
Recall the definition of the parabolic norm $|z|_{\rm p,\lambda}$ in Definition \ref{def:parabolic_BMO}.
\begin{proof}[Proof of Theorem \ref{theorem:range_of_samB_contains_manifold_of_cod_1}]
    The map $\Phi$ is given by Definition \ref{def:Phi}. The property $\Phi(0)=0$ is immediate. Its differential is
    \[
    D\Phi(\g)\cdot \dot \g=-\aha\int_{\RR} \sech^2\cdot\,(\dot\g\psi+\g\dot\psi),
    \]
    \[
    \dot\psi:=-T_{\g}(\dot \g\psi)
    \]
    with $\psi$ as in Proposition \ref{prop:well-posedness-of-psi-to-g} (the operator $T_{\g}:=(\de_y-\de_x^2+2\tanh\de_x+\g)^{-1}$ is well-defined on $L^1\cap L^2$ by Lemma \ref{lemma:operator_l_tr_+_g_bounds} part (b)). The non-degeneracy of the differential at any $\g$ can be seen by choosing $\dot \g$ as sign-definite and supported on $y>M$, with $M$ large enough so that $\g$ is small on $y>M$ and $\dot\psi$ is identically zero for $y<M$. The property
    \[
    \Phi(\g)=\Phi(\Rcal \g)
    \]
    is shown in Remark \ref{rk:Phi_is_invariant_under_reflections}.
    
    Fix $\eps>0$. We first show that $\Phi(g)=0$ implies that $g$ is in the range of $\samB$. Let $g\in Y_\eps(\RR)$ be small enough. Consider $\psi\in C_b(\RR)$ the unique solution of
    \[
    \psi_y-\psi_{xx}+2\tanh\psi_x=-(\Rcal g)\psi
    \]
    converging to the constant $1$ uniformly as $y\to-\infty$ given by Proposition \ref{prop:well-posedness-of-psi-to-g}. By the change of variables discussed in Subsection \ref{subsec:change-of-var-w-psi}, the function
    \[
    w:=-\Rcal\left(\partial_x\log\psi\right)=-\Rcal\left(\frac{\psi_x}{\psi}\right)
    \]
    solves equation \eqref{eq:equation_for_w_unvariabled_section_range}, and the system \eqref{eq:system_miura_maps_backlund_section_range} is solved by
    \[
    \wu:=g+\ph,\qquad u:=g+2w_x,\qquad v:=w+\tanh(x).
    \]
    In particular, the pair $(u,v)$ solves \eqref{eq:miuraEquation}. For the claim to hold, it is enough to show that $w\in L^3(\RR)$, and that $w_x\in \Hcrit$ is small enough. In fact, the latter implies that $u\in \Hcrit$ is small, since $g$ is small in $Y_\eps(\RR)\subset \Hcrit$. Moreover, by the uniqueness statement in Theorem \ref{theorem:theorem_1}, the former implies that we indeed have $v=\dtsmV(u,\gamma_0)$ for some $\gamma_0\in\R$, and thus $\wu=\samB(u,\gamma_0)$ by the definition of $\samB$.

    We focus on proving $w_x\in\Hcrit$ small, since the other condition follows similarly. By the definition of $w$, with sufficient regularity we have
    \[
    w_x=\Rcal\left( \frac{\psi_x^2}{\psi^2}-\frac{\psi_{xx}}{\psi} \right),
    \]
    so the claims follows by fractional Leibniz and fractional chain rule, after establishing for small $g$
    \[
    \||\de_x|^{3/2}\psi\|_{L^2(\RR)}+\|\partial_x\psi\|_{L^3(\RR)}+\||\de_x|^{1/2}\psi\|_{L^6(\RR)}\lesssim \|g\|_{Y_\eps(\RR)}.
    \]
    The second term is bounded by the other two terms by interpolation. For the other terms, writing $h:=\Rcal g$ for brevity, we decompose $\psi$ into
    \als{
        \psi&=-T(h\psi)+1\\
        &=\left[-\Gamma^+(\eta^+ h\psi)-\Gamma^-(\eta^-h\psi)\right]-\aha\demu(\Gamma^--\Gamma^+)(\sech^2 h\psi)+1\\
        &=:\rm{I}+\rm{II}+1,
        }
    where we recall that $T$ is defined in \eqref{eq:operator-T-transport-outwards}, and the first equation is in the proof of Proposition \ref{prop:well-posedness-of-psi-to-g}. The contribution from $\rm I$ satisfies all the three bounds above, since by Proposition \ref{prop:spacetime_estimate_heat_general} one has
    \[
    \||\partial_x|^s\Gamma f\|_{L^p(\RR)}\lesssim \|f\|_{L^{3/2}(\RR)},\quad s\in(0,2],\;\frac{1}{p}=\frac{s}{3},
    \]
    and we can assume $\|\psi\|_{L^\infty}\leq 2$ by the estimates of Proposition \ref{prop:well-posedness-of-psi-to-g} and the smallness of $g$. For the term $\rm II$, it is enough by Lemma \ref{lemma:heat_kernel_and_Hardy_space_H1} to show that $\sech^2 h\psi$ belongs to the intersection of parabolic Hardy spaces $\PHux1(\RR)\cap \PHux{-1}(\RR)$. By Lemma \ref{lemma:Lp+weight/decay+zero_mean_implies_Hardy_H1}, this follows when $\sech^2h\psi$ has mean zero, which is granted by the condition $\Phi(g)=\Phi(h)=0$, and by the weighted estimates
    \[
    \||(x,y)|_{\p,\pm 1}^{3/p'+\delta}\sech^2g\psi\|_{L^p}\lesssim_{p,\eps,\delta} \|g\|_{L^2\cap L^1_{\sech^2,\eps}}.
    \]
    This estimate is true for $p$ close enough to $1$ and for $\delta>0$ small enough. In fact, by interpolation of weighted spaces it holds
    \[
    \|w^{\theta}f\|_{L^{p}(\RR)}\leq \|f\|_{L^1(\RR)}^{1-\theta}\|wf\|_{L^2(\RR)}^\theta,\quad\frac{1}{p}=\frac{1-\theta}{1}+\frac{\theta}{2},
    \]
    with $w(x,y)=\sech^2(x)(1+|y|)^\eps$. For $\theta\in (0,1)$ with $1>\theta>\frac{3}{3+2\eps}\iff \delta:=\theta\eps-3/p'>0$,
    \begin{align*}
    w^\theta&=\sech^2(x) \cosh^{2(1-\theta)}(x)(1+|y|)^{\theta\eps}\\
    &\gtrsim \sech^2(x) (1+|x|+|y|)^{\theta\eps}\\
    &\geq \sech^2(x)(1+|x|+|y|)^{3/p'+\delta}\\
    &\gtrsim \sech^2(x)|(x,y)|^{3/p'+\delta}_{\rm p,\pm 1}.
    \end{align*}

For the second part of the proof, assume $\Phi(g)\neq 0$. We focus first on what we know about $w$. Since $g\in Y_\eps(\RR)$, as we have just shown, $\sech^2 h\psi$ satisfies the assumptions of Lemma \ref{lemma:Lp+weight/decay+zero_mean_implies_Hardy_H1}, except for its integral being non-zero. It can thus be written as a function in the Hardy space $\PHux1(\RR)\cap \PHux{-1}(\RR)$ plus a scalar multiple of an arbitrary test function. A quick study of the kernel of $\Gamma^--\Gamma^+$ shows that because of this, $\rm {II}$ does not have an $x$-derivative in $L^3(\RR)$, unlike $\rm I$. In particular, derivating in $x$, we have $\psi_x,w\not\in L^3(\RR)$. However,
\[
\psi_x,w\in L^{3,\infty}(\RR)\setminus L^3(\RR),
\]
since $\Gamma^\pm\in L^{3,\infty}(\RR)$ and $\sech^2h\psi\in L^1(\RR)$. Moreover, it can be checked that
\[
\|\Gamma^\pm(x,\cdot)\|_{L^2(\R_y)}\lesssim \log^{1/2}_-(|x|)+\jap{x}^{-1/4},
\]
which gives for any $\eps>0$ the estimate
\begin{equation}\label{eq:z-026}
\|\sech^\eps(x)\Gamma^\pm(\sech^{\eps}f)\|_{L^2(\RR)}\lesssim_{\eps} \|f\|_{L^1(\RR)}.
\end{equation}
Considering again the equation $\psi=-T(h\psi)+1$, this estimate, together with estimate \eqref{eq:z-024} of Lemma \ref{lemma:weighted-estimates-heat} with $s=1$, and the identity $w=-\Rcal(\psi_x/\psi)$, implies that
\[
\sech(x)w\in L^2(\RR).
\]
We also know that $w$ satisfies \eqref{eq:equation_for_w_unvariabled_section_range}, as noted at the beginning of the proof.

Now assume by contradiction that $\wu$ is in the range of $\samB$. That is,
\als{
\wu&=\samB(u',\gamma_0')\\
&=u'-2\de_x\dtsmV(u',\gamma_0')
}
for some $u',\gamma_0'$ as in the assumptions of Theorem \ref{theorem:theorem_1}. Call $v'=\dtsmV(u',\gamma_0')$, and $w':=v'-\tanh$. Then, as discussed at the beginnning of the section, $w'$ solves equation \eqref{eq:equation_for_w_unvariabled_section_range}. We also have $w'\in L^3(\RR)$. In fact, let $v'=w''+\tanh_{\alpha}$ be the decomposition with $\alpha$ as in Theorem \ref{theorem:theorem_1}. In particular, $\alpha_y\in L^2(\R)$, $w''\in L^3(\RR)$, $w_x''\in H^{-\fr12,0}(\RR)$. Since as we said $u'-2v'_x=\wu=\ph+g$, and $-2v_x=-2w''_x+\ph_\alpha$, it holds
\[
\ph_\alpha-\ph=g-u'+2w''_x\in H^{-\aha,0}(\RR),
\]
which implies $\alpha\in L^2(\R)$ knowing that $\alpha_y\in L^2(\R)$. This implies
\[
w'=w''+(\tanh_\alpha-\tanh)\in L^3(\RR).
\]
Similarly, by the estimate of Theorem \ref{theorem:theorem_1} and the weighted estimates on the functions $v^\pm$ therein coming from Corollary \ref{cor:estimates-for-wtv-part-2} part (a), we have immediately $\sech_{\alpha} w''\in L^2(\RR)$, which by the same argument above implies
\[
\sech(x) w'\in L^2(\RR).
\]

To recap, we have $w,w'\in L^{3,\infty}(\RR)\cap \cosh(x)L^2(\RR)$, and they both solve \eqref{eq:equation_for_w_unvariabled_section_range}. We claim that they must coincide.

\begin{claim}
    Let $w^1,w^2\in L^{3,\infty}(\RR)\cap \cosh(x)L^2(\RR)$ solve equation \eqref{eq:equation_for_w_unvariabled_section_range} with $g\in \mathscr D'(\RR)$. Then, $w^1= w^2$.
\end{claim}

\begin{claimproof}[of the Claim]
    We first verify that for any $\eps>0$,
    \[
    \de_xT: L^{3/2,\infty}(\RR)\cap \cosh^{2-\eps}(x)L^1(\RR)\to L^{3,\infty}(\RR)\cap \cosh^\eps(x)L^2(\RR)
    \]
    is well-defined and bounded, with $T$ as in \eqref{eq:operator-T-transport-outwards}. In particular, $\de_xT$ is the integral operator
    \[
    \de_x T= \de_x\Gamma^+\Mult_{\eta^+}+\de_x\Gamma^-\Mult_{\eta^-}+\fr 12 (\Gamma^--\Gamma^+)\Mult_{\sech^2}.
    \]
    The statement is true for the operators $\de_x\Gamma^\pm$: first, it holds
    \[
    \|\sech^\eps(x)\de_x\Gamma^\pm f\|_{L^2(\RR)}\lesssim_\eps \|f\|_{L^{3/2,\infty}(\RR)}
    \]
    by interpolation between estimate $\|\de_x\Gamma^\pm f\|_{L^2(\RR)}\lesssim \|f\|_{L^{6/5}(\RR)}$ from Proposition \ref{prop:spacetime_estimate_heat_general}, and estimate \eqref{eq:z-024} in Lemma \ref{lemma:weighted-estimates-heat} with $s=1$; second, their kernels belong to $L^{3/2,\infty}(\RR)$, and $L^{3/2,\infty}*L^{3/2,\infty}\subset L^{3,\infty}$ \cite{o_neil-1963-convolution_operators_and_Lpq}. The statement is also true for the operators $\Gamma^\pm\Mult_{\sech^2}$: in fact, since $\Gamma^\pm\in L^{3,\infty}(\RR)$, the two respective convolution operators map $L^1(\RR)$ to $L^{3,\infty}(\RR)$, and they map $\sech^{\eps}(x)L^1(\RR)$ to $\cosh^\eps(x)L^2(\RR)$ by \eqref{eq:z-026}. So the estimate for $\de_x T$ is proved by combining the above ones.
    
    Consider now $w:=w^1-w^2$. After a reflection in the $y$ variable, using the same names to denote the reflected versions of the respective functions, $w$ solves
    \[
    w_y-w_{xx}+2(\tanh w)_x=-((w^1+w^2)w)_x.
    \]
    By density and the uniqueness of solutions of the linear heat equation with prescribed initial data, a solution $z\in L^{3,\infty}(\RR)$ of the above equation with the right hand side equal to zero lies in $C(\R_y,L^1+L^\infty(\R))$, and satisfies the a priori estimate
    \[
    \|z|_{y=y_1}\|_{L^1+L^\infty(\R)}\lesssim \|z|_{y=y_0}\|_{L^1+L^\infty(\R)}
    \]
    with $y_0<y_1$ due to the diffusion and the vector field $\de_y+2\tanh(x)\de_x$ having positive divergence. In particular, $z$ must be zero by sending $y_0$ to $-\infty$. Since $w\in L^{3,\infty}(\RR)\cap \cosh(x)L^2(\RR)$, this implies that
    \[
    w=\de_xT((w^1+w^2)w).
    \]

    Calling $X:=L^{3,\infty}(\RR)\cap\cosh^\eps(x)L^2(\RR)$, $Y:=L^{3/2,\infty}(\RR)\cap\cosh^{2-\eps}(x)L^1(\RR)$, we have the estimate
    \als{
    \|w\|_X&=\|\de_xT((w^1+w^2)w)\|_X\\
    &\lesssim \|(w^1+w^2)w\|_Y\\
    &\lesssim \|w^1+w^2\|_X\|w\|_X.
    }
    This estimate holds when restricting all functions on any half plane $\R\times (-\infty,M)$, $y_0\in\R$, due to the fact that all the convolution kernels appearing in $\de_x T$ are supported for positive $y$. When $M$ is smaller than a suitable $M_0$ such that $\mathbbm 1_{\{y<M_0\}}(w^1+w^2)$ is small enough in $X$, the estimate implies
    \begin{equation}\label{eq:z-027}
    \mathbbm 1_{\{y<M\}}w=0.
    \end{equation}
    Equation \eqref{eq:z-027} then implies
    \begin{equation}\label{eq:z-028}
        w_y-w_{xx}+2(\tanh w)_x=-(\mathbbm 1_{\{y<M\}}(w^1+w^2)w)_x.
    \end{equation}
    The same argument above, with a bootstrap argument involving \eqref{eq:z-027} and \eqref{eq:z-028}, imply that \eqref{eq:z-027} holds for all $M\in\R$. The claim is thus proved.
\end{claimproof}
As we said, we also have $w'\in L^3(\RR)$, $w\in L^{3,\infty}(\RR)\setminus L^3(\RR)$, which yields a contradiction with the assumption that $\wu$ is in the range of $\samB$. This concludes the proof of the Theorem.
\end{proof}

\subsection{A conjecture on the range of the soliton addition map}\label{subsec:conjecture_on_Phi_and_h}

In \cite{mizumachi2019phase}*{Theorem 1.5}, Mizumachi proves that polinomially localized perturbations of the line soliton induce a finite, well-defined shift $h\in\R$ of the position of the soliton along the $x$ axis in a co-moving frame. Specifically, if $u$ is the solution of \eqref{eq:KP-II} with initial datum $u_0=\ph+g$ such that $\left\|\langle x\rangle(\langle x\rangle+\langle y\rangle) g\right\|_{H^1\left(\mathbb{R}^2\right)}$ is small enough, there exists $h\in\R$ such that suitable modulation parameters $x=x(t,y)$ and $\lambda=\lambda(t,y)$ describing the modulations of the line soliton ($u= \ph^{\lambda(t,y)}(x-x(t,y))+\OO_{L^2}(\|g\|)$ for a suitable norm $\|\cdot\|$) satisfy $\sup _{t \geq 0, y \in \mathbb{R}}\left|x(t, y)-4 t\right| \lesssim \left\|\langle x\rangle(\langle x\rangle+\langle y\rangle) g\right\|_{H^1\left(\mathbb{R}^2\right)}$, and
\[
\left
\{\begin{aligned}
&\lim _{t \rightarrow \infty}\left\|\lambda(t, \cdot)-1\right\|_{L^{\infty}(\R)}=0 \\
&\lim _{t \rightarrow \infty}\left\|x(t, \cdot)-4t-h\right\|_{L^{\infty}\left(|y| \leq(4-\delta) t\right)}=0 \\
&\lim _{t \rightarrow \infty}\left\|x(t, \cdot)-4t\right\|_{L^{\infty}\left(|y| \geq(4+\delta) t\right)}=0
\end{aligned}\right.
\]
for any $\delta>0$.

We conjecture that when $g$ is small enough in $(\langle x\rangle(\langle x\rangle+\langle y\rangle))^{-1}H^1(\RR)$ and $Y_\eps(\RR)$, the spaces involved in Theorem \cite{mizumachi2019phase}*{Theorem 1.5} and our Theorem \ref{theorem:range_of_samB_contains_manifold_of_cod_1} and Corollary \ref{cor:codimension_1_stability} respectively, it holds
\[
h=0\quad\iff\quad \Phi(g)=0,
\]
where $\Phi$ is the functional in Theorem \ref{theorem:range_of_samB_contains_manifold_of_cod_1}. More generally, we conjecture that $h$ is a function of $\Phi$.\\
As discussed in the introduction, the codimension-1 condition is natural, linked to the integrable structure, and of qualitative type. It is reasonable to suspect that the manifold contained in the range of $\samB$ follows special dynamics along the KP-II flow. The above equivalence says that this manifold corresponds to the set of perturbations such that the line soliton converges back to the non-perturbed soliton $\ph(x-4t)$ locally in space, along a co-moving frame. All other perturbations will converge to a soliton shifted by a finite, non-zero amount along the $x$ axis. Intuitively, if the perturbed soliton was the image of a localized solution of KP-II through $\samB$, the behavior of $x(t,y)$ described above would look too special to be compatible with the fact that general small solutions of KP-II in $\Hcrit$ scatter, unless the constant $h$ is zero. We did not find a proof of this before finishing this article, so we leave it as an open problem.

\section{A multisoliton addition map}\label{sec:six}

Recall the definitions of $\eleV$, $\eleVt$ in Propositions \ref{prop:kink-addition_map_eleV}, \ref{prop:eleVt_superposition-of-elementary-functions}.

\begin{definition}
    We define the (upgraded) \emph{B\"acklund transform for $(M-1,1)$-solitons} as follows. Let $u_0\in\Hcrit$ small enough, $M\geq 1$, and $\vec \lambda\in \R^M$ such that $\lambda_1<\dots<\lambda_M$.  For $\vec c\in \R^M$, we define the transformation for fixed time
    \[
    \eleB^{\vec\lambda}(u_0,\vec c):=u_0-2\de_x \eleV^{\vec\lambda}(u_0,\vec c),
    \]
    and the transformation that includes the time evolution of the image through the KP-II flow,
    \[
    \eleBt^{\vec\lambda}(u_0,\vec c):=u-2\de_x\eleVt^{\vec\lambda}(u_0,\vec c),
    \]
    where $u$ is the solution of KP-II with initial datum $u_0$.
\end{definition}

From the definition of multisolitons in Section \ref{sec:preliminaries} and from the aforementioned Propositions, we have the following results:

\begin{enumerate}
    \item Adding a scalar multiple of the vector $(1,\dots,1)$ to $\vec c$ leaves the image of $\eleB$, $\eleBt$ unchanged.
    \item For $M=2$, $\eleB^{(-1,1)}$ coincides with the soliton addition map $\samB$ we constructed in Definition \ref{def:soliton_addition_map_and_L^2_phi(RR)} up to a homeomorphic change of variables of the domain, and up to the symmetry of the previous point (this follows from Lemma \ref{lemma:change-of-variables-c-alpha_0-gamma_0}).
    \item Given $M\geq 2,\vec\lambda,\vec c$, the function $\eleB^{\vec\lambda}(0,\vec c)$ is a $(M-1,1)$-multisoliton (with $\eleBt^{\vec\lambda}(0,\vec c)$ being its time evolution along the KP-II flow). For fixed $M\geq 2$, the map $(\vec\lambda,\vec c)\mapsto \eleB^{\vec\lambda}(0,\vec c)$ is a bijective parametrization of the set of $(M-1,1)$-multisolitons, up to rescaling $\vec c$ as in the first point.
    \item The function $\wu:=\eleBt^{\vec\lambda}(u_0,\vec c)$ belongs to $L^2_\loc([0,\infty)\times\R^2)$ and solves the KP-II equation distributionally with initial datum $\eleB^{\vec\lambda}(u_0,\vec c)$ (the symbol `$\sto$' refers to the forward time evolution).
    \item For each $\vec\lambda,u_0,\vec c$, there exists a continuous function $\vec{\mathbf c}=\vec{\mathbf c}(t)$ with $\vec{\mathbf c}(0)=\vec c$ such that
    \[
    \eleBt^{\vec\lambda}(u_0,\vec c)(t)=\eleB^{\vec\lambda}(u(t),\vec{\mathbf c}(t)).
    \]
    \item The maps are continuous with values in suitable low regularity spaces.
\end{enumerate}
    
The above is the subclass of the `tree-shaped' multisolitons, with $N_+=1$, the number of outgoing solitons at $y=+\infty$. These are the multisolitons in which the size $N$ of the matrix in \eqref{eq:wronskian-determinant} equals $1$. For instance, choosing $M=3$, the map $\eleB$ allows the construction of solutions close to a modulated $Y$-shaped multisoliton.

We note that the transformation makes sense for $M=1$ and yields a nontrivial B\"acklund transform for solutions of KP-II without solitons. It is an immediate consequence of Corollary \ref{cor:estimates-for-wtv-part-2} that $\eleB^{\lambda}(\cdot,1)$, for $\lambda\in\R$, leaves the space $\Hcrit$ invariant for small data. This map can be seen as a limit of the map for $M=2$ where the one of the coordinates of the vector $\vec c$ goes to $+\infty$, which morally corresponds to adding a line soliton at $x=\infty$ or $x=-\infty$.

Finally, the B\"acklund transforms can be conjugated by the reflection symmetry \eqref{eq:reflection-symmetry-KP-II}. The conjugated maps add solitons that have one soliton at $y\to-\infty$ and $M-1$ solitons at $y\to\infty$.

\appendix

\section{Linear estimates and parabolic equations}\label{appendix:heat-equation}

\subsection{Linear operators, kernels, and estimates for the heat equation}
\begin{definition}\label{def:heat-operators-in-appendix}
    We define the following operators that act on suitable functions on $\Real^2$:
    \begin{enumerate}[label=-]
        \item $H:=\partial_y-\partial_x^2$
        \item $H^{(c)}:=\partial_y+c \partial_x-\partial_x^2$
        \item $L_{{\rm tr}\pm}:=\partial_y-\partial_x^2\pm 2\tanh\partial_x$
        \item $L_{{\rm co}\pm}:=\partial_y-\partial_x^2\pm 2\partial_x\tanh$
        \item $L_{{\rm tr}-}\partial_x^{-1}=\partial_x^{-1}L_{{\rm co}-}=\partial_x^{-1}\partial_y-\partial_x-2\tanh$,
    \end{enumerate}
where `$\tanh$' denotes the multiplication operator by the function $\tanh(x)$. Furthermore, let 
\[
G_t(x):=\mathbbm 1_{\{t>0\}} \frac{1}{\sqrt{4\pi t}} e^{-\frac{x^2}{4t}}
\]
be the heat kernel in 1 space dimension. We define the heat kernels
\begin{equation}\label{eq:heat-kernels-definition}
\begin{aligned}
&\Gamma(x,y):=G_{y}(x),&&\demu\Gamma(x,y):=\int_{0}^x G_y(x')dx',&&\Gamma^\pm:=\Gamma^{(\pm 2)},\\
&\Gamma^{(c)}(x,y):=\Gamma(x-cy,y),&&\demu\Gamma^{(c)}(x,y):=\demu\Gamma(x-cy,y),&& \demu\Gamma^\pm:=\demu\Gamma^{(\pm 2)}.
\end{aligned}
\end{equation}
For each of the above kernels $K=K(x,y)$, we will use the same symbol to denote the associated convolution operator
\[
K f(x,y):=\int_{\RR} K(x-x_0,y-y_0)f(x_0,y_0) \,dx_0\,dy_0.
\]
More generally, in what follows we will consider integral operators that are not translation-invariant:
\[
K f(x,y):=\int_{\RR} K(x,y;x_0,y_0)f(x_0,y_0) \,dx_0\,dy_0,
\]
with kernels $K=K(x,y;x_0,y_0)$.
Finally, let $\Mult_{\eta^\pm}=:\Mult^\pm$, $\Mult_{\sech^2}$ be the multiplication operators by the functions $(x,y)\mapsto \eta^\pm(x):=\frac{1}{1+e^{\pm 2x}}$, $(x,y)\mapsto\sech^2(x)$ respectively.
\end{definition}
\begin{proposition}[Explicit kernels]\label{prop:explicit-kernels}
	The operators in Definition \ref{def:heat-operators-in-appendix} admit respectively the following operators as right inverses:
            \begin{enumerate}[label=-]
                \item $\Gamma,$
                \item $\Gamma^{(c)},$
                \item $K_{{\rm tr}+}:=\Gamma^+\Mult^++\Gamma^-\Mult^- +\aha (\demu\Gamma^--\demu\Gamma^+) \Mult_{\sech^2},$
                \item $K_{{\rm tr}-}:=\Mult^+\Gamma^-+\Mult^- \Gamma^+,$
                \item $K_{{\rm co}+}:=\Gamma^+\Mult^++\Gamma^-\Mult^-,$
                \item $K_{{\rm co}-}:=\Mult^+\Gamma^-+\Mult^-\Gamma^+ +\aha \Mult_{\sech^2}(\demu\Gamma^--\demu\Gamma^+),$
                \item $DK_{{\rm tr}-}:=\Mult^+\partial_x \Gamma^-+\Mult^-\partial_x\Gamma^++\aha \Mult_{\sech^2}(\Gamma^--\Gamma^+).$
            \end{enumerate}
The above are all integral operators, the kernels of which will be called with the same symbols.
\end{proposition}
\noindent The proof is straightforward. The kernels can be deduced from the heat kernel using the following relations
\begin{align}\label{eq:conjurel}
	&L_{\rm tr-}=\Mult^{-1}H\Mult,&&L_{\rm co-}\partial_x=\partial_x L_{\rm tr-},&&(\partial_x+2\tanh)L_{\rm co-}=L_{\rm tr-}(\partial_x+2\tanh(x)),
\end{align}
where $\Mult$ is the multiplication operator by the function $e^y\cosh(x)$, and the fact that some operators are adjoint to others after a reflection in the $y$ variable (for instance, $K_{{\rm tr}-}(x,y;x_0,y_0)=K_{{\rm co}+}(x_0,y;x,y_0)$). Note also that it holds $DK_{{\rm tr}-}=\partial_x K_{{\rm tr}-}=-\partial_{x_0} K_{{\rm co}-}$.

\begin{proposition}[Estimates for the heat equation with forcing]\label{prop:spacetime_estimate_heat_general}
    Let $s\in[0,2]$, $p,q,r,\sigma\in[1,\infty]$ satisfying
    \begin{equation*}
	\left(\frac{2}{r}+\frac{1}{\sigma}\right)=\left(\frac{2}{p}+\frac{1}{q}\right)-2+s.
    \end{equation*}
Consider the integral operator $\Gamma$ as before:
$$ \Gamma f(x,y):=\int_{\RR}\Gamma(x-x_0,y-y_0)f(x_0,y_0)\,dx_0\,dy_0, $$ 
where $\Gamma(x,y)=\mathbbm 1_{\{y>0\}} \frac{1}{\sqrt{4\pi y}} e^{-\frac{x^2}{4y}}$. The estimate
\begin{equation*}
	\||\partial_x|^s \Gamma f\|_{L_y^rL_x^\sigma}\lesssim\|f\|_{L_y^pL_x^q}
\end{equation*}
holds whenever the right hand side is finite, in the following cases:
\begin{enumerate}
    \item $s\in[0,2)$, $1<p<r<\infty$ and  $1\leq q\leq \sigma\leq \infty$,
    \item $s=2$,  $1< p,q<\infty$, $(p,q)=(r,\sigma)$,
    \item $s\in[0,1]$, $(r,\sigma)=(\infty,2)$, $1\leq p,q,\leq 2$, $(p,q,s)\neq(2,1,1/2)$\\ (in this case, it holds $|\partial_x|^s\Gamma f\in C_0(\R_y,L^2_x)$).
\end{enumerate}
\end{proposition}
All these estimates hold for $\Gamma^{(c)}$ as well with uniform constants in $c$, thanks to the change of coordinates $(x,y)\mapsto(x-cy,y)$.

\begin{proof}
    It is straightforward to verify that for fixed $y>0$, all fractional $x$-derivatives of $\Gamma(\cdot,y)$ of non-negative order are bounded and in $L^1$ (they actually lie in the Hardy space $\mathcal H^1$ for $s>0$). Thus, simply by the scaling symmetry in the $y$ variable, it follows that
    \begin{equation}
        \||\partial_x|^s\Gamma(\cdot,y)\|_{L_x^p}\lesssim_s |y|^{-\aha(1+s)+\frac{1}{2p}}, \quad 1\leq p\leq\infty,\,s\in[0,\infty).
    \end{equation}
    By Young's convolution inequality, this in turn implies the following estimates for the fractional derivatives of the heat kernel:
    \[ \||\de_x|^s e^{y\partial_x^2} f\|_{L_x^\sigma}\lesssim|y|^{-\frac{k}{2}-\frac{1}{2}(\frac{1}{q}-\frac{1}{\sigma})}\|f\|_{L_x^q},\qquad y>0, \qquad s\geq 0.\]
    Part (1) is then a consequence of the Hardy--Littlewood--Sobolev inequality and the above $L^q-L^\sigma$ estimates of the heat propagator $e^{y\partial_x^2}$.
    
    A proof of part (2) is contained in \cite{Lemarie_Rieusset_2002_recent_developments_in_the_Navier_Stokes_problem}*{Chapter 7}.
    
    For part (3), consider the heat equation
    \[
    u_y-u_{xx}=|\partial_x|^sf.
    \]
    The standard energy estimates yield immediately
    \[
    \|u\|_{L^\infty_y L^2_x}^2+\|u_x\|_{L^2_yL^2_x}^2\leq \||\partial_x|^su\|_{L^{p'}_yL^{q'}_x}\|f\|_{L^p_yL^q_x}.
    \]
    The estimate then follows by the interpolation estimate
    \[
    \||\partial_x|^su\|_{L^{p'}_yL^{q'}_x}\lesssim \|u\|_{L^\infty_y L^2_x}^{1-\theta}\|u_x\|^\theta_{L^2_yL^2_x}
    \]
    with $\theta=s+(1/2-1/q')$ and the inequality $2ab\leq Ca^2+C^{-1}b^2$. Note that the above interpolation inequality fails precisely at the endpoint $s=1/2$, $(p',q')=(2,\infty)$ (for which $\theta=1$). By approximation with smooth functions, it holds $|\partial_x|^s\Gamma f\in C_0(\R_y,L^2_x)$.
\end{proof}

\begin{remark}\label{rk:estimates_for_the_heat_operator}
    With the same methods as in part (1), estimates with mixed derivatives are proved. For example, we will use the following one:
    \begin{equation}\label{eq:z-011}
    \||\de_y|^{\frac{1}{4}}\de_x \Gamma f\|_{L^2_{x,y}}\lesssim\|f\|_{L^{3/2}_{x,y}}.
    \end{equation}
\end{remark}
We refer to the definition of $d_{\rm p,\lambda}$ in Definition \ref{def:parabolic_BMO}.
\begin{lemma}\label{lemma:heat_kernel_sends_Lp_to_Hoelder}
    The heat operator $\Gamma^{(-2\lambda)}=(\de_y-\de_x^2-2\lambda\de_x)^{-1}$ extends to a bounded map
    \[
        \Gamma^{(-2\lambda)}:L^p(\RR)\to \faktor {C^{0,\alpha}(\RR,d_{{\rm p},\lambda})}{\R}\,\,, \qquad 1-\fr \alpha 3=\frac{1}{p}+\frac{1}{3},\quad p\in(3/2,3),
    \]
    where for $0<\alpha<1$ and $(X,d)$ metric space, we denote the semi-normed H\"older space by
    \[
    C^{0,\alpha}(X,d):=\left\{f\in C(X,d)\,\bigg|\,|f|_{ C^{0,\alpha}}:=\sup_{z_1,z_2\in X}\frac{|f(z_1)-f(z_2)|}{d(z_1,z_2)^{\alpha}}<\infty\right\},
    \]
    and where $\R\subset C^{0,\alpha}(X,d)$ is the subspace of constant functions.
\end{lemma}
\begin{proof}
    By the change of variables $(x,y)\mapsto (x+2\lambda y,y)$, we can assume $\lambda=0$. Set $|\cdot|_{\rm p}=|\cdot|_{\rm p,0}$ as in Definition \ref{def:parabolic_BMO}, and let $z=(x,y)$, $w=(x',y')$. The heat kernel $\Gamma$ satisfies
    \begin{equation}\label{eq:z-012}
        |\Gamma(z)|\lesssim |z|_{\rm p}^{-1},\qquad |\partial_x\Gamma(z)|\lesssim |z|_{\rm p}^{-2},\qquad |\partial_y\Gamma(z)|\lesssim |z|_{\rm p}^{-3}.
    \end{equation}
    As a consequence (see the proof of Lemma \ref{lemma:heat_kernel_and_Hardy_space_H1}), we have the estimate
    \[
    |\Gamma(z+w)-\Gamma(z)|_{\rm p}\lesssim
    \left\{
    \begin{aligned}
        &\frac{1}{|z+w|_{\rm p}}+\frac{1}{|z|_{\rm p}}&&\text{for }|z|_{\rm p}\leq 2|w|_{\rm p},\\
        &\frac{|w|_{\rm p}}{|z|^2_{\rm p}}&&\text{for }|z|_{\rm p}\geq 2|w|_{\rm p},\\
    \end{aligned}
    \right.
    \]
    By taking the $L^2$ norm and splitting the integral on the two regions, using that $\int_{r\leq|z|_{\rm p}\leq R}|z|_{\rm p}^{-s}dz\lesssim R^{3-s}-r^{3-s}$, we find
    \[
        \|\Gamma(\cdot+w)-\Gamma(\cdot)\|_{L^q{(\RR_z)}}\lesssim |w|_{\rm p}^{\alpha},\qquad 0<\alpha<1,\quad \fr 1q=\fr{1-\alpha}3 +\fr{2\alpha }{3}.
    \]
    We have $\Gamma f(w)-\Gamma f(0)=\int_{\RR}(\Gamma(z+w)-\Gamma(z))f(-z)dz$, $w\in \RR$, for any test function $f$. Choosing $q=p'$ in the previous estimate, the right hand side is bounded by
    \[
    \left|\int_{\RR}(\Gamma(z+w)-\Gamma(z))f(-z)dz\right|\lesssim |w|_{\rm p}^{\alpha} \|f\|_{L^p}
    \]
    for $\alpha$ as in the statement of the Lemma. The previous equation allows thus to extend $\Gamma$ to $f\in L^p$, with $\Gamma f$ being well-defined up to an additive constant. The $C^{0,\alpha}$ bound is given by
    \[
    |\Gamma f(w_1)-\Gamma f(w_2)|\lesssim |w_1-w_2|^{\alpha}_{\rm p}\|f\|_{L^p(\RR)},
    \]
    which follows by the previous bound by translation invariance.
\end{proof}

\begin{lemma}\label{lemma:weighted-estimates-heat}
    Let $c\in\R\setminus\{0\}$ and $\alpha\in C(\R_y)$ such that $\alpha_y\in L^2(\R_y)$. The following bounds hold:
    \begin{align}
        \|\braket{(x-\alpha(y))/L}^{-1}|D_x|^s\Gamma^{(c)}u\|_{L^2(\RR)}&\lesssim |c|^{-1} \|u\|_{L^p(\RR)},\quad s\in [0,1], \frac{1}{p}=\frac{(1-s)}{6/5}+\frac{s}{2},\label{eq:z-024}\\
        \|\braket{(x-\alpha(y))/L}^{-1}\partial_x\Gamma^{(c)}u\|_{L^2(\RR)}&\lesssim |c|^{-\aha} \|u\|_{\Hcrit},\label{eq:z-025}
    \end{align}
    where $\Gamma^{(c)}=(\partial_y-\partial_x^2+c\partial_x)^{-1}$ (see Definition \ref{def:heat-operators-in-appendix}), and $L=|c|^{-1}(1+\|\alpha_y\|_{L^2})^{-1}$.
\end{lemma}
\noindent Estimate \eqref{eq:z-025} shows that estimate \eqref{eq:z-024} is not optimal in terms of regularity. Note that the inequalities are invariant under the scaling $(x,y)\mapsto(\lambda x,\lambda^2 y)$, which leaves the quantity $\|\alpha_y\|_{L^2}$ unchanged.
\begin{proof}
    We can assume $c>0$ without loss of generality by symmetry. Let $\rho\in C^\infty_c(\R)$ a standard mollifier, $\rho_\eps(x):=\eps^{-1}\rho(\eps^{-1}x)$, $\eps>0$, and let $\alpha_\eps:=\rho_\eps*\alpha$. It is straightforward to show
    \[
    \|\alpha_{\eps,y}\|_{L^\infty}\lesssim \eps^{-\aha} \|\alpha_y\|_{L^2},\qquad \|\alpha_{\eps,y}-\alpha\|_{L^\infty}\lesssim \eps^{\aha}\|\alpha_y\|_{L^2},
    \]
    and choosing $\eps=\Big(\frac{6C_1\|\alpha_y\|_{L^2}}{c}\Big)^2$, where $C_1$ is the implicit constant in the first inequality, we have
    \[
    \|\alpha_{\eps,y}\|_{L^\infty}\leq\frac{c}{6},\qquad \|\alpha_\eps-\alpha\|_{L^\infty}\lesssim c^{-1}\|\alpha_y\|_{L^2}^2,
    \]
    and in particular,
    \[
    \braket{(x-\alpha(y))/L}^{-1}\lesssim \braket{(x-\alpha_\eps(y))/c^{-1}}^{-1}.
    \]
    This means that it is enough to prove the two estimates for $L=c^{-1}$, and assuming that $\|\alpha_y\|_{L^\infty}\leq \frac{c}{6}$.
    
    Consider the first estimate. Assume first $s=0$. We know by Proposition \ref{prop:spacetime_estimate_heat_general} that $v\in C_{0,y} L^2_x$. Let $a(x,y):=2-\arctan(3c(x-\alpha(y)))$. Considering the heat equation with transport
    \[
    (\partial_y-\partial_x^2+c\partial_x)v=u,
    \]
    multiplying by $a$ and integrating in $dx$ and by parts, we obtain the energy estimate
    \[
    \aha\fr{d}{dy}\int av^2dx+\int av_x^2dx-\frac{c}{2}\int a_xv^2dx+\fr12\alpha_y\int a_x v^2dx -\aha\int a_{xx}v^2dx=\int avf\,dx.
    \]
    The third term is non-negative due to $a$ being non-increasing, and by the definition of $a$ and the estimate $\|\alpha_y\|_{L^\infty}\leq \frac{c}{6}$, it is at least $3$ times larger than the absolute value of the fourth and fifth terms. Integrating in $y$, we thus obtain
    \als{
        \|\sqrt a v\|^2_{L^\infty L^2}+\|\sqrt a v_x\|^2_{L^2L^2}+\frac{c^2}{6}\|\braket{3c(x-\alpha)}^{-1}v\|^2_{L^2L^2}&\leq \left|\iint avf\,dx\,dy\right|\\
        &\leq C^{-1}\|\sqrt a v\|_{L^6}^2+C\|f\|_{L^{6/5}}^{2},
        }
    which proves the estimate for $C$ large enough by the interpolation inequality
    \[
    \|f\|_{L^{6}L^6}\lesssim\|f\|_{C_0L^2}^{2/3}\|f_x\|^{1/3}_{L^2L^2}.
    \]
    The estimate for $s=1$ follows with the same tools, this time with the right hand side of the energy estimate being
    \[
    \int avf_xdx=-\int a_xvf\,dx-\int av_xf\,dx.
    \]
    The estimate for $s\in(0,1)$ follows by interpolation.

    For the second estimate, we let $v=\partial_x\Gamma^{(c)} u$, and set $\U:=|\partial_x|^{-1/2}u\in L^2(\RR)$, $\V:=\partial_x\Gamma^{(c)}\U\in C_0L^2$. Clearly, it holds $\V:=|\partial_x|^{-1/2}v$.
    By the previous estimate with $s=1$ and Proposition \ref{prop:spacetime_estimate_heat_general}, which can be applied after a linear change of coordinates, we can establish the bounds
    $$ \|\mathcal V\|_{C_0 L^2}+\|\mathcal V_x\|_{L^2L^2}+c\|\jap{c(x-\alpha)}^{-1}\mathcal V \|_{L^2L^2}\lesssim\|\mathcal U\|_{L^2}=\|u\|_{\dot H^{-\aha,0}}. $$
    The last two terms on the left hand side control the quantity $c^{1/2}\|\mathbbm 1_{|x-\alpha(y)|\leq c^{-1}}v\|_{L^2L^2}$, and from that we can achieve the full bound stated in \eqref{eq:z-025} by a summation trick and by translation invariance.
\end{proof}

\begin{lemma}[Mapping properties of $(\de_y-\de_x^2+2\tanh\de_x+h)^{-1}$]\label{lemma:operator_l_tr_+_g_bounds}
    Consider $T:=K_{{\rm tr}+}$, the integral operator as in Proposition \ref{prop:explicit-kernels} that inverts the operator $L_{{\rm tr}+}=\partial_y-\partial_x^2+2\tanh\partial_x$ .
    \begin{enumerate}[label=(\alph*)]
        \item The operator $T$ is well-defined from $L^1(\RR)\cap L^2(\RR)$ to $L^\infty(\RR)$, and it holds
        \begin{equation}
            \|Tu\|_{L^\infty(\RR)}\lesssim \|u\|_{L^1(\RR)\cap L^2(\RR)}.
        \end{equation}
        The range lies in the subspace $C_0^-(\RR):=\{\psi\in L^\infty\,|\,\psi\in C_0(\R\times(-\infty,y]) \;\,\forall y\in\R\}$.
        \item Let $h\in L^1(\RR)\cap L^2(\RR)$. Define the operator
        $$ L_{{\rm tr}+,h}:=\partial_y-\partial_x^2+2\tanh\partial_x +h. $$
        The operator is invertible from the space $T(L^1(\RR)\cap L^2(\RR))$ to $L^1(\RR)\cap L^2(\RR)$. Denoting by $T_h:=(L_{{\rm tr}+,h})^{-1}$, we have for a universal constant $C>0$
        \[
        \|T_h u\|_{T(L^1(\RR)\cap L^2(\RR))}\lesssim \exp\left(C\|h\|_{L^1(\RR)\cap L^2(\RR)}\right) \|u\|_{L^1(\RR)\cap L^2(\RR)}.
        \]
    \end{enumerate}
\end{lemma}

\begin{proof}
    Recall that the operator $T$ is the integral operator on $\RR_{x,y}$ with kernel
        \begin{align*}\label{eq:kernel-of-linearized-psi-from-g}
            K(x,y;x_0,y_0)&=\Gamma^+(x-x_0,y-y_0)\eta^+(x_0)+\Gamma^-(x-x_0,y-y_0)\eta^-(x_0)\\
            &\QQQ+\frac{1}{2}\partial_x^{-1}(\Gamma^-(x-x_0,y-y_0)-\Gamma^+(x-x_0,y-y_0))\sech^2(x_0),
        \end{align*}
    or equivalently, following the notation in Definition \ref{def:heat-operators-in-appendix},
    \[ T=\Gamma^+\Mult_{\eta^+}+\Gamma^-\Mult_{\eta^-}+\aha (\demu\Gamma^--\demu\Gamma^+) \Mult_{\sech^2}. \]
    The estimate in part (a) is then a consequence of the estimates
    \[
    \|\Gamma^\pm u\|_{L^\infty}\lesssim \|u\|_{L^1\cap L^2},\qquad \|\demu\Gamma^\pm u\|_{L^\infty}\lesssim \|u\|_{L^1},
    \]
    which come from Young's convolution inequality and the fact that the convolution kernels $\Gamma^\pm$, $\demu \Gamma^\pm$ belong respectively to $L^{3,\infty}(\RR)\subset L^2+L^\infty$, $L^\infty(\RR)$. The statement on the range is true for $u\in C^\infty_c(\RR)$ as the integral kernel is identically zero for $y<y_0$, and extends by density to all $u$.
    
    For part (b), we argue perturbatively. The statement for small $h$ follows by Neumann series inversion: the estimate
    \als{
    \|h\cdot Tu\|_{L^1\cap L^2(\RR)}&\leq \|h\|_{L^1\cap L^2(\RR)}\|Tu\|_{L^\infty(\RR)}\\
    &\lesssim \|h\|_{L^1\cap L^2(\RR)}\|u\|_{L^1\cap L^2(\RR)}
    }
    implies that the operator $L_{{\rm tr}+,h} T={\rm Id}+hT$ is invertible on $L^1\cap L^2(\RR)$.\\
    For large $h$, we can repeat the same argument on a subset $\R\times(-\infty,y_0]\subset \RR$, with $y_0=y_0(h)\ll 0$ so that the norm $\|h\|_{L^1\cap L^2(\R\times(-\infty,y_0])}$ is small enough, and then extend $\psi=(L_{{\rm tr}+,h})^{-1}u$ to the whole $\RR$ by solving the initial value problem of the linear PDE
    \[
    \left\{
    \begin{aligned}
        &\psi_y-\psi_{xx}+2\tanh \psi_x=-h\psi+u,\\
        &\psi(\cdot,y_0)=\psi_0
    \end{aligned}
    \right.
    \]
    with $\psi_0\in C_0(\R_x)$. The global bounds follow from the a priori estimate 
    \begin{equation}\label{eq:z-008}
    \begin{aligned}\|\psi(y)\|_{L^\infty(\R_x)}&\leq \|\psi_0\|_{L^\infty(\R_x)}+\|Tu\|_{L^\infty(\RR)}\\
    &\QQQ+2\int_{y_0}^{\min\{1,y\}} \|\Gamma(\cdot,y-s)\|_{L^{2}(\R_x)} \|h(s)\|_{L^2(\R_x)}\|\psi(s)\|_{L^\infty(\R_x)}ds\\
    &\QQQ+2\int_1^{\max\{1,y\}} \|\Gamma(\cdot,y-s)\|_{L^{\infty}(\R_x)} \|h(s)\|_{L^1(\R_x)}\|\psi(s)\|_{L^\infty(\R_x)}ds\\
    &\QQQ+\int_{y_0}^y \|\demu\Gamma(\cdot,y-s)\|_{L^\infty(\R_x)} \|h(s)\|_{L^1(\R_x)}\|\psi(s)\|_{L^\infty(\R_x)}ds
    \end{aligned}
    \end{equation}
    given by the Duhamel formulation of the problem involving the kernel $K$, from the estimates
    \[
    \|\Gamma^\pm (\cdot,y)\|_{L^p(\R_x)}\lesssim |y|^{-\aha(1-1/p)},\quad \|\demu\Gamma^\pm (\cdot,y)\|_{L^\infty(\R_x)}\lesssim 1,
    \]
    and from Gronwall's inequality.
\end{proof}

\subsection{Well-posedness for parabolic equations}
Here we state a well-posedness result for the initial value problem (on positive sub-intervals of $\R_y$) associated to equation \eqref{eq:miuraEquation}.

\begin{lemma}\label{lemma:miura-eq-GWP-in-y-L2}
Let $v_0\in L^2(\Real_x)+\tanh(x)$ and $u\in L^2(\R\times(0,\infty))$. There exists a unique solution $v\in C([0,\infty), L^2(\Real_x))+\tanh(x)$ to equation \eqref{eq:miuraEquation} such that $v(\cdot,y)=v_0$. The map $(v_0,u)\mapsto v$ is continuous (equipping the codomain with the compact-open topology).
\end{lemma}

\begin{proof}
The proof relies on a standard fixed point argument. First, calling $z(x):=v(x)-\tanh(x)$, we equivalently show the global well-posedness of equation
\begin{equation}\label{eq:miuraEquationZ}
    z_y-z_{xx}-2(\tanh(x)z)_x-(z^2)_x=-u_x.
\end{equation}
By the uniqueness properties of the linear heat equation, we can equivalently look for solutions satisfying the integral equation
\[
z(y)=e^{y(\partial_x^2+2\partial_x\tanh)}z_0+\int_0^ye^{(y-s)(\partial_x^2+2\partial_x\tanh)}\partial_x(z^2(s)-u(s))\,ds,
\]
where $e^{y(\partial_x^2+2\de_x\tanh)}$ and $e^{y(\partial_x^2+2\de_x\tanh)}\partial_x$ are the integral operators with kernels $K_{\rm{co}-}(\cdot,y)$ and $DK_{\rm{tr}-}(\cdot,y)$ defined in Proposition \ref{prop:explicit-kernels}. Consider the map
$$ \Phi(v)=e^{y(\partial_x^2+2\partial_x\tanh)}z_0+\int_0^ye^{(y-s)(\partial_x^2+2\partial_x\tanh)}\partial_x(v^2(s)-u(s))\,ds.$$
We note that, up to multiplications by bounded functions, $e^{y(\partial_x^2+2\tanh)}\partial_x$ is essentially a sum of heat kernels and derivatives of heat kernels, while $e^{y(\partial_x^2+2\tanh)}$ is a sum of heat kernels and a term whose $L^1$ norm grows linearly in $y$. Thus, we can easily obtain the bounds
$$\|\Phi(v)\|_{L_T^\infty L^2}\lesssim (1+T)\|z_0\|_{L^2}+(1+T)\|u\|_{L^2_{x,y}} + (T^{1/4}+T^{3/4})\|v\|_{L^\infty_TL^2}^2,$$
$$
\|\Phi(v)-\Phi(w)\|_{L_T^\infty L^2}\lesssim (T^{1/4}+T^{3/4})\|v+w\|_{L^\infty_TL^2}\|v-w\|_{L^\infty_T}.
$$
This is enough to prove local well-posedness using the Banach fixed point theorem and standard arguments. The global well-posedness follows from the standard energy estimate of equation \eqref{eq:miuraEquationZ},
\[
\|z\|_{L^\infty_TL^2}+\|z_x\|_{L^2_TL^2}\lesssim T^\aha \|z\|_{L^\infty_T L^2}+\|u\|_{L^2L^2},
\]
and Gronwall's inequality.
\end{proof}

\begin{lemma}\label{lemma:localWellPosednessExponentiallyLocalized}
Let $R>0$. There exists $T=T(R)>0$ such that, letting $z_0\in L^2(\Real_x;\cosh^2(x)dx)$, $V,f,h\in L^3(\Real_x\times(0,T)) $, $g\in L^3(\Real_x\times(0,T))$ with norms bounded by $R>0$, the Cauchy problem
\begin{equation*}
    \left\{\begin{aligned}
        &z_y-z_{xx}-2((\tanh(x)+U)z)_x=(z^2)_x+\sech^2(x)(f+g)+(\sech^2(x)h)_x\\
        &z(x,0)=z_0(x)
    \end{aligned}\right.
\end{equation*}
admits a unique solution $ z\in C([0,T],L^2(\Real;\cosh^2(x)dx)) $. Moreover, the map
$$ (z_0,U,f,g,h)\mapsto z $$
is continuous.
\end{lemma}

\begin{proof}
    Consider $w:=\cosh(x)z$. The equation for $w$ becomes
    \begin{align*}
        w_y-w_{xx}+w-2\sech^2 w-2\cosh \cdot\,(\sech Uw)_x&=\sech \cdot\, (w^2)_x-2\sech\tanh w^2\\
        &\QQQ+\cosh\cdot\,(\sech^2\cdot\,(f+g)+(\sech^2 h)_x)
    \end{align*}
    and the proof follows the lines of a classical fixed point argument as in Lemma \ref{lemma:miura-eq-GWP-in-y-L2}.
\end{proof}

\section{The Miura map with \texorpdfstring{$L^2$}{L2} data}\label{sec:existence_of_solutions_with_L2_data}
This appendix, originally a first attempt in the construction of solutions to \eqref{eq:miuraEquation}, describes the situation of small data $u\in L^2(\RR)$. This is a simplified setting which can be instructive for the reader, and we will rely to some extent on this subsection for the proof of Corollary \ref{cor:almost-conservation-L2-norm-around-line-soliton}. The problem with using exclusively this approach is that we lack a uniqueness theorem for the solutions of \eqref{eq:miuraEquation} with generic $L^2$ data.

The main idea is to prove a monotonicity estimate for the initial value problem of equation \eqref{eq:miuraEquation} to obtain global solutions with uniform bounds on any interval $[a,\infty)$, and then let $a\to-\infty$ to obtain a solution defined on the whole $\RR$ by compactness. We start with the following Lemma, which gives a simple decomposition of a function in $L^2(\R_x)+\tanh(x)$. Concerning this decomposition, we will use the letter $\beta$ to distinguish this shift parameter from the shift $\alpha$ in the rest of the paper, although they are both quantities that represent the positions of the kink.

We make use of Notation \ref{notation:f_alpha} throughout all this subsection: $f_\beta(x):=f(x-\beta)$ if $\beta\in\R$, and $f_\beta(x,y):=f(x-\beta(y))$ if $\beta:\R\to\R$. We define the quantity
\[
\vertiii{v_0}_{L^2(\R)}:=\min_{\gamma\in\R} \|v_0-\tanh_\gamma\|_{L^2(\R)}.
\]
\begin{lemma}\label{lemma:decompositionWAlpha}
    There exists $\theta_0>0$ and an analytic decomposition map $v\mapsto (w,\beta)$ from $A_{\theta_0}\subset L^2(\R)+\tanh(x)$ to $L^2\times\Real$, where $A_{\theta_0}:=\{v\,|\,\vertiii{v}_{L^2}<\theta_0\}$, such that $v=\tanh_\beta+w$ and
    \[
    \int_{\R} w\sech^2_\beta\,dx=0.
    \]
    Moreover, $(w,\beta)$ is uniquely determined by the above properties under the hypothesis that $v\in A_{\theta_0}$.
\end{lemma}
\begin{proof}
    The first part is a consequence of the implicit function theorem. The map
    $$
    \begin{aligned}
    F&:(L^2(\R)+\tanh(x))\times\Real\to\Real,\\
    F&(v,\beta)=\int_{\R}(v-\tanh_\beta)\sech^2_\beta\dx
    \end{aligned}
    $$
    is such that $F(\tanh_\gamma,\gamma)=0$, $\gamma\in\R$, and has $\beta$-derivative at $(v,\beta)=(\tanh_\gamma,\gamma)$ equal to
    \[
    \de_\beta F(\tanh_\gamma,\gamma)=\|\sech^2_\gamma\|_{L^2(\R)}^2>0.
    \]
    By the implicit function theorem, there exists $\eps_1>0$ and an analytic function
    \[
    \beta\colon B_{\eps_1}^{L^2}(\tanh_\gamma)\to L^2\times \Real
    \]
    such that $F(v,\beta(v))=0$. Moreover, $\eps_1$ does not depend on $\gamma$ by translation invariance, and any two such maps agree where they overlap.
    
    For the second part, assume without loss of generality that $v=w+\tanh(x)$, $\int_{\R} w\sech^2(x)\dx=0$. Then,
     \[ \int_{\R} (w+\tanh-\tanh_\beta)\sech^2_\beta\dx=w*\sech^2(\beta)+\tanh*\sech^2(\beta). \]
    It is then clear that $w*\sech^2(0)=0$ and, for $\|w\|_{L^2}$ smaller than a suitable $\eps_2>0$, the function $w*\sech^2$ will have an $L^\infty$ norm and a Lipschitz constant so small that it will never be equal to $-\tanh*\sech^2(\beta)$ besides at $\beta=0$. The claim follows by choosing $\theta_0=\min\{\eps_1,\eps_2\}$.
\end{proof}
\begin{lemma}\label{lemma:global-bounds-L2-small-data}
Let $u\in L^2(\Real\times (0,\infty))$ small, and $v_0\in L^2(\R)+\tanh(x)$ be such that $\vertiii{v_0}_{L^2(\R)}$ is small enough. The solution $v\in C([0,\infty),L^2(\R))+\tanh(x)$ of equation \eqref{eq:miuraEquation} with initial datum $v|_{y=0}=v_0$ given by the global well-posedness theory (see Lemma \ref{lemma:miura-eq-GWP-in-y-L2}) satisfies
$$ \sup_{y\in[0,\infty)}\vertiii{v(y)}_{L^2(\R)}\lesssim \vertiii{v_0}_{L^2(\R)}+\|u\|_{L^2L^2}. $$
Moreover, called $w$ and $\beta$ the decomposition given by Lemma \ref{lemma:decompositionWAlpha}, we have the bounds
\[ \|w\|_{L^\infty L^2}+\|w_x\|_{L^2L^2}+\|\sech_\beta w\|_{L^2L^2}+\|\beta_y\|_{L^2}\lesssim\|u\|_{L^2_{x,y}}+\vertiii{v_0}_{L^2}.\]
Finally, called $h:=w-k$, with $k\in C([0,\infty),L^2(\R))$ being the unique solution to $k_y-k_{xx}=-u_x$ with $k(\cdot,0)\equiv 0$, we have the bounds
\[
\|h_y\|_{L^{3/2}+L^2(\Real\times(0,\infty))}+\|h_{xx}\|_{L^{3/2}+L^2(\Real\times(0,\infty))}\lesssim \|u\|_{L^2_{x,y}}+\vertiii{v_0}_{L^2}.
\]
\end{lemma}
\begin{proof}
    Assume first that $v_0$ and $u$ are test functions. The solution $v$ given by Lemma \ref{lemma:miura-eq-GWP-in-y-L2} is then a classical solution in $C^\infty([0,\infty),H^\infty(\R)+\tanh)$, and for a maximal time $T^*>0$, the quantity $\vertiii{v(y)}_{L^2}$ remains small in $[0,T^*)$. It follows by Lemma \ref{lemma:decompositionWAlpha} that there exists a decomposition $w\in C^\infty L^2$, $\beta\in C^\infty([0,T^*))$ such that
    \[
    v=w+\tanh_\beta,\qquad\int_{\R} w\sech^2_{\beta}dx=0\quad\forall y\in[0,T^*).
    \]
    The equation for $w$ then reads
    $$w_y-w_{xx}-2(\tanh_\beta w)_x=(w^2)_x-u_x+\beta_y\sech^2_\beta.$$
    We multiply by $w$ and integrate integrating in $x$. Using Lemma \ref{lemma:lowerBoundOnTheQuadraticForm} with the orthogonality condition $\int \sech^2_\beta w\,dx=0$, integrating then in $y$ and using Cauchy--Schwarz, we obtain the estimate
    $$\|w\|^2_{L_{[0,T^*)}^\infty L^2}+\|w_x\|_{L_{[0,T^*)}^2L^2}^2+\|\sech_\beta w\|_{L_{[0,T^*)}^2L^2}^2\lesssim\|u\|_{L_{[0,T^*)}^2L^2}^2+\vertiii{v_0}_{L^2}^2.$$
    By the smallness assumptions, we have an a priori uniform bound on the quantity
    $\vertiii{v(y)}_{L^2},$
    which implies $T^*=+\infty$ by a continuity argument.
    
    The only term remained to estimate is the derivative of $\beta$. From the equation, we get
    \begin{align}\label{equationForTheDerivativeOfAlpha}
        0&=\frac{d}{dy}\int_{\R}w\sech^2_\beta\dx\\
        &=\int_{\R}w_y\sech^2_\beta\dx-\beta_y\int w\,(\sech^2_\beta)_x\dx\\
        &=\int_{\R}\left((\sech^2_\beta)_{xx}-2\tanh\cdot\,(\sech^2_\beta)_x\right)w\,dx+ \int_{\R} \sech^2_\beta (w^2)_x\,dx\\
        &\QQQ+\frac{4}{3}\beta_y+\int_{\R} u\sech^2_\beta dx-\beta_y\int_{\R}{w(\sech^2_\beta)_x}\,dx.
    \end{align}
    If $\|w\|_{L^\infty L^2}$ is small enough, we thus obtain
    \als{
        |\beta_y(y)|&\lesssim \int_{\R} (|w|+|w_x|)\sech^2_\beta\,dx+\int_{\R} |w||w_x|\sech^2_\beta\,dx+\int_{\R} |u|\sech^2_\beta\,dx\\
        \implies \|\beta_y\|_{L^2(\R_y)}^2&\lesssim \|v_0\|_{L^2(\R_x)}^2+\|u\|_{L^2(\RR_{x,y})}^2.
        }
    For general data $v_0$ and $u$, we can argue by density using the $C_T L^2$ continuity and the continuity of the decomposition map in Lemma \ref{lemma:decompositionWAlpha} to obtain the same conclusion and the desired bounds on $w$, $w_x$, $\beta_y$.
    
    Finally, we consider $h$ as defined in the statement. Then, $h$ satisfies the equation
    \[
    h_y-h_{xx}=(2\tanh_\beta w+w^2)_x+\beta_y\sech^2_\beta.
    \]
    Using the previous estimates, we note that the right hand side of the equation lies in $L^{3/2}_{x,y}+L^2_{x,y}.$ Thus, by the $L^p$ boundedness of the operator $\partial_x^2(\partial_y-\partial_x^2)^{-1}$ from Proposition \ref{prop:spacetime_estimate_heat_general} and the uniqueness properties of the heat operator, this same bound is transferred on $h_y$ and $h_{xx}$, hence the last statement.
\end{proof}
We now want to build an eternal solution, that is, a solution of \eqref{eq:miuraEquation} living in $\Real^2$. We will use the a priori bounds proved for the solution.
\begin{proposition}\label{proposition:existenceOfAnEternalSolution}
Let $u\in L^2(\Real^2)$ be small enough, and $\beta_0\in \R$. There exists an eternal solution $v\in C_0(\R_y,L^2(\R_x))$ of equation \eqref{eq:miuraEquation} such that
$$ \int_{\R}\sech^2(x-\beta_0)v(x,0)\,dx=0. $$
Moreover, it holds $\sup_{y\in\Real}\vertiii{v(y)}_{L^2}<\theta_0$ as in Lemma \ref{lemma:decompositionWAlpha}, and the unique decomposition $v=\tanh_\beta+w$ given for each $y\in\R$ by the Lemma satisfies $w\in C(\R_y,L^2(\R_x))$ and the bounds
$$ \|w\|_{L^\infty L^2}+\|w_x\|_{L^2(\RR)}+\|\sech_\beta w\|_{L^2(\RR)}+\|\beta_y\|_{L^2(\R_y)}\lesssim\|u\|_{L^2(\RR)}\quad. $$
\end{proposition}
\begin{proof}
By translation invariance, we assume $\beta_0=0$. By Lemma \ref{lemma:global-bounds-L2-small-data}, we can find solutions $v^N\in L^\infty L^2+\tanh(x)$ to \eqref{eq:miuraEquation} with initial times $y_N\to-\infty$ and initial data $v_0^N(x)=\tanh(x-\beta_0^N)$, with uniform bounds on the decompositions given by Lemma \ref{lemma:decompositionWAlpha}, which we will call $w^N$ and $\beta^N$. Using the estimate on $\beta^N_y$ in Lemma \ref{lemma:global-bounds-L2-small-data} and by continuity of the decomposition map of Lemma \ref{lemma:decompositionWAlpha}, we can choose $\beta_0^N$ such that $\beta^N(0)=0$, i.e.,
$$ \int_{\R}\sech^2(x)v^N(x,0)\,dx=0.$$
Now, we can use the uniform estimates for $v^N$ given by Lemma \ref{lemma:global-bounds-L2-small-data} to have that, up to subsequences and using a diagonal construction, $\beta^N$ converges locally uniformly to a function $\beta$ with $\beta(0)=0$ and $\|\beta_y\|_{L^2}\lesssim \|f\|_{L^2_{x,y}}$. The sequence $w^N$ converges $*$-weakly in $L^\infty L^2$ to a function $w\in L^\infty L^2$, the derivatives $w^N_x$ converge weakly to $w_x$ in $L^2(\Real^2)$, and the uniform bounds of $w^N$ are carried over to $w$. We will then call $v:=w+\tanh(x-\beta)$.\\
Moreover, by the previous Lemma, we can assume weak convergence of $(w^N-k^N)_y$ and $(w^N-k^N)_{xx}$ in $L^{3/2}+L^2(\RR)$ after removing $k^N:=-(\de_y-\de_x^2)^{-1}(\mathbbm 1_{\{y>y_N\}}u_x)$, which converges to $k:=-(\de_y-\de_x^2)^{-1}u_x$ in $L^6(\RR)$ by the estimates of Proposition \ref{prop:spacetime_estimate_heat_general}. By Rellich's compactness Theorem, $w^N$ converges strongly in $L^2_{\loc}(\RR)$ to $w$, which immediately implies that $v^N\to v$ in $L^2_\loc(\RR)$, and the same happens for the first order $x$-derivatives. It follows by continuity that $v$ satisfies the equation \eqref{eq:miuraEquation} distributionally. By the uniqueness properties of the heat equation and the a priori estimate, $v$ satisfies the Duhamel formulation of the equation, so the continuity in $y$ holds due to Proposition \ref{prop:spacetime_estimate_heat_general}. It is a consequence of the strong convergence that $(w,\beta)$ coincides with the decomposition given by Lemma \ref{lemma:decompositionWAlpha}.
\end{proof}

\section{\texorpdfstring{$U^p$}{Up} and \texorpdfstring{$V^p$}{Vp} spaces}\label{appendix:upvp}
We refer to \cite{candy-herr-2018-on-the-division-problem-wave-maps,hadac2008well-posednessKP-II,koch-tataru-visan-2014-dispersive-equations-nonlinear-waves} for the definitions in this appendix.

\subsection{Definition of the spaces \texorpdfstring{$U^p$}{Up} and \texorpdfstring{$V^p$}{Vp}}
Let $I=(a,b)\subset \R$ an open, possibly unbounded interval. Denote by $\bar I:=I\cup\{a,b\}\subset \R\cup\{\pm\infty\}$. Denote by $\mc R$ the set of $L^2(\RR)$-valued regulated functions on $I$, that is, bounded functions which admit left and right limits at any given point of the domain (and admit right limit and left limit at $a$, $b$ respectively). Let $\mc R_\rc\subset \mc R$ the subset of all right continuous functions $u$ such that $\lim_{t\to a^+}u(t)=0$. The spaces $\mc R, \mc R_\rc$ are Banach spaces when equipped with the supremum norm. Moreover, $\mc R_\rc$ embeds naturally into $\mathscr S'(I\times \RR)$.

Define the set of all partitions of $I$
\[\mb{P}= \{ \tau=(t_j)_{j=1}^{N} \mid N \in \N, \,\, t_j \in I, \,\, t_j< t_{j+1} \},\] corresponding to the decompositions of $I$ into subintervals. We say that $u$ is a \emph{step function} if there exists $\tau \in \mb{P}$ and $f_1,\dots,f_{|\tau|}\in L^2(\RR)$ such that
\[u(t) = \sum_{j=1}^{|\tau|} \mathbbm 1_{[t_j,t_{j+1})}(t) f_j,
\]
where we adopt as a convention $t_0=a,t_{|\tau|+1}=b$. We denote the set of step functions as $\mc S_\rc$ and note that $\mc S_\rc\subset \mc R_\rc$. A step function $u\in \mc S_\rc$ is a \emph{$U^p$-atom} if in the above definition the vectors $(f_j)_{j=1}^{|\tau|}$ satisfy
    $$ \Big( \sum_{j=1}^{|\tau|} \| f_j \|_{L^2(\RR)}^p \Big)^{\frac{1}{p}} = 1.$$
The space $U^p$ is defined to be the space
    $$ U^p := \left\{ u=\sum_{j=1}^\infty c_j u_j \s \Big| \s (c_j)_{j\geq 1} \in \ell^1(\N_+,\C),\s\text{$(u_j)_{j\geq 1}$ $U^p-$atoms } \right\},$$
with the norm
    $$ \| u \|_{U^p} = \inf\left\{\sum_{j=1}^\infty |c_j|\s\Big|\s\exists\; (u_j)_{j\geq 1} \text{ $U^p$-atoms}\s: \s u = \sum_{j\in \N} c_j u_j \right\}.$$
Since the $U^p-$atoms belong to $\mc R_\rc$ with bounded norm, and since the $U^p$ norm controls the supremum norm, the above sum is well-defined and one has $U^p\hookrightarrow \mc R_\rc$.

We define the spaces $V^p$ for completeness, although we will not use them. Define the $p-$variation of a function $v: I \rightarrow L^2$ as the seminorm
        $$ |v |_{V^p} = \sup_{(t_j)_{j=1}^{|\tau|} \in \mb{P}} \Big( \sum_{j=1}^{|\tau|-1}\| v(t_{j+1}) - v(t_j) \|_{L^2}^p \Big)^\frac{1}{p}$$
and $\| v \|_{V^p} =\|v\|_{L^\infty_t L^2_{x,y}}+|v |_{V^p}$. Note that all functions whose $V^p$ seminorm is finite belong to $\mc R$, but may contain functions which are identically zero outside a countable set. Let $V^p$ be the space of all functions $v\in\mc R_\rc$ such that $| v |_{V^p} < \infty$. On the space $V^p$, the seminorm $|\cdot|_{V^p}$ is in fact a norm, due to the condition at $-\infty$, and it is equivalent to the norm $\|\cdot\|_{V^p}$.

\begin{definition}[Adapted function space, \cite{hadacHerrKoch2009wellPosednessKP-IIinCriticalSpace}]\label{def:adapted-space-Zs}
    We define $U_S^p$ as the space of functions of the form $e^{tS}f(t)$, $f\in U^p$, with norm $\|u\|_{U_S^p}=\left\|e^{-t S} u\right\|_{U^p}$, where $S=-\partial_x^3-3\demu\partial_y^2$. Define $\dot Z^s(I)$ as the closure of all $u \in C(I,H^{1,1}(\mathbb{R}^2)) \cap U_S^2$ such that
$$
\|u\|_{\dot{Z}^s}:=\left(\sum_{\lambda\in 2^\Z} \lambda^{2 s}\left\|P^x_\lambda u\right\|_{U_S^2}^2\right)^{\frac{1}{2}}<\infty,
$$
in the space $C_b(I,\dot{H}^{s, 0}(\mathbb{R}^2))$ with respect to the $\dot Z^s-$norm.
\end{definition}

\begin{lemma}\label{lemma:smoothing-estimate-for-U^2_S}
    The following estimates hold:
    \[
    \|u\|_{L^p_tL^q_{x,y}}\lesssim_p \|u\|_{U^2_S},\quad \fr 1p+\fr 1q=\aha,\,\,\,p\in(2,\infty],
    \]
    \[
    \|\partial_x u\|_{L^\infty_xL^2_{t,y}}+\|\demu \partial_y u\|_{L^\infty_xL^2_{t,y}}\lesssim \|u\|_{U^2_S}.
    \]
\end{lemma}
\begin{proof}
    By the definition of $U^2$ and arguing on $U^2$-atoms first (see \cite{hadacHerrKoch2009wellPosednessKP-IIinCriticalSpace}*{Corollary 2.18}), the estimates follow from the corresponding estimates for the linear flow $t\mapsto e^{tS}$. For a solution $u=e^{tS}u_0$, the usual Strichartz estimates hold
    \[
    \|u\|_{L^p_tL^q_{x,y}}\lesssim_p\|u_0\|_{L^2(\RR)},\quad \fr 1p+\fr 1q=\aha,\,\,\,p\in(2,\infty]
    \]
    (see \cite{koch-tataru-visan-2014-dispersive-equations-nonlinear-waves}), which imply the first estimate. Analogously, the smoothing estimate for linear solutions of KP-II,
    \[
    \|\de_xe^{tS}u_0\|_{L^\infty_xL^2_{t,y}}+\|\demu\partial_x e^{tS}u_0\|_{L^\infty_xL^2_{t,y}}\lesssim \|u_0\|_{L^2(\RR)},
    \]
    which is proved in \cite{kenig-ziesler-2005-LWP-for-mKP-II}*{Lemma 3.2}, implies the second estimate (the first part of the second estimate is actually already proved in \cite{hadacHerrKoch2009wellPosednessKP-IIinCriticalSpace}*{Corollary 2.18}).
\end{proof}
\begin{corollary}\label{cor:Z^aha_is_in_L6_de_x^aha_L3}
    It holds $\dot Z^{-1/2}((0,\infty)) \hookrightarrow C_b([0,\infty),\Hcrit)\cap L^6((0,\infty),|\de_x|^\aha L^3(\RR))$.
\end{corollary}
\begin{proof}
    The first estimate is immediate by the definition of $U^2$ and Minkowski integral inequality. Note that a function in $\dot Z^{-\aha}$ is well-defined at $t=0$ because of the definition of $U^2$. For the second one, we first use Lemma \ref{lemma:smoothing-estimate-for-U^2_S} to estimate the quantity
    \[
    \left(\sum_{\lambda\in 2^\Z} \lambda^{-1}\left\|P^x_\lambda u\right\|_{L^6_tL^3_{x,y}}^2\right)^{\frac{1}{2}}\lesssim \|u\|_{\dot Z^{-\aha}}.
    \]
    Minkowski's integral inequality is then used multiple times to bring the summation on $\lambda$ inside the norm, and finally the square function characterization of the $L^p$ norm
    \[
    \|(-\de_x^2)^{s/2}f\|_{L^3(\R)}\sim \|(\sum_{\lambda\in 2^{\Z}} \lambda^s|P_\lambda f|^2)^\aha\|_{L^3(\R)}
    \]
    is enough to conclude.
\end{proof}

\begin{remark}\label{rk:Z-aha-is-in-L2-loc}
    It holds the embedding $\dot Z^{-\aha}((0,\infty))\hookrightarrow L^2_\unif((0,\infty)\times \RR_{x,y})$. In fact, combining the smoothing estimate of Lemma \ref{lemma:smoothing-estimate-for-U^2_S} with the embedding $U^2\hookrightarrow L^\infty L^2$ yields
    \begin{align*}
        \sum_{\lambda\in 2^\Z}\lambda^{-1}\|P^x_\lambda u\|^2_{U^2_S}&\gtrsim \|P^x_{\leq 1}u\|_{L^\infty_t L^2_{x,y}}^2+\sum_{\lambda\geq 1}\lambda \|P^x_\lambda u\|^2_{L^\infty_x L^2_{t,y}}\\
        &\gtrsim \|P^x_{\leq 1}u\|_{L^\infty_t L^2_{x,y}}^2+\|P^x_{\geq 1}u\|^2_{L^\infty_xL^2_{t,y}}.
    \end{align*}
    By interpolating between the two estimates in Lemma \ref{lemma:smoothing-estimate-for-U^2_S}, with the same argument, it is possible to show that $u\in L^p_\unif((0,\infty)\times\RR)$, $p<8/3$.\\    
    Moreover, the operator $\demu \partial_y$ extends to a bounded operator from $\dot Z^{-\aha}((0,\infty))$ to some Banach space of tempered distributions, since by Lemma \ref{lemma:smoothing-estimate-for-U^2_S} it holds
    \[
    \left(\sum_{\lambda\in 2^\Z}\lambda^{-1}\|P^x_\lambda \partial_x^{-1}\partial_y u\|^2_{L^\infty_x L^2_{t,y}}\right)^\aha\lesssim \|u\|_{\dot Z^{-\aha}}.
    \]
    The same statements hold for the space $X_T^{1/2+\eps,b_1,0}$ as in Definition \ref{def:Bourgain_type_spaces_X^s,b}, with different estimates. The spaces $X_T^{1/2+\eps,b_1,0}$, $0<T<\infty$ and $\dot Z^{-\aha}((0,\infty))$ contain all solutions of \eqref{eq:KP-II} with initial data in $L^2(\RR)$, $\Hcrit$ respectively (see Proposition \ref{prop:well-posedness-of-KP-II-L2} and Theorem \ref{theorem:global-well-posedness-KP-II-HHK09} and Proposition): in particular all solutions from the well-posedness theory with the above data lie in $L^2_\loc((0,\infty)\times\RR)$ and the operator $\partial_x^{-1}\partial_y$ is well-defined on those solutions. This means that all terms in \eqref{eq:KP-II} are well-defined.
\end{remark}

\section{Supplementary lemmas and proofs}

\subsection{Miscellaneous results}\label{sec:miscellaneous}

\begin{lemma}[Lower bound on a quadratic form]\label{lemma:lowerBoundOnTheQuadraticForm}
For every $w\in H^1(\R)$, the inequality
$$ \int_\R (w_x)^2\,dx-2\int_\R \sech^2(x)w^2\,dx\geq 0$$
holds, assuming one of the orthogonality conditions
$$ \braket{w,\sech}_{L^2}=0,\quad\text{or}\quad \braket{w,\sech^2}_{L^2}=0. $$
\end{lemma}
\begin{proof}
    From classical operator theory \cite{praQuantMechFlugge1999}, we know explicitly the negative energy states of the P\"osch--Teller type Schr\"odinger operators $H_n=-\de_x^2-n(n-1)\sech^2(x)$, $n\geq 2$. In particular, the operator
    \[ H_2=-\partial_x^2-2\sech^2(x) \]
    has its ground state $w_0=\sech(x)$ as the only negative energy state, with eigenvalue $-1$, while the operator
    \[ H_3=-\partial_x^2-6\sech^2(x) \]
    has two bound states, $v_0=\sech^2(x)$ and $v_1=\tanh(x)\sech(x)$, with energies that are respectively $-4$ and $-1$.
    The goal is to prove that the quadratic form $Q(\cdot):=\braket{\cdot,H_2\cdot}$ is non-negative on the hyperplanes $w_0^\perp$ and $v_0^\perp$.
    
    The first statement is immediate, since $w_0$ is the only negative eigenvector of $H_2$.
    Concerning the orthogonal of $v_0$, we argue as follows. Since $w_0$ is even, the quadratic form is positive on the subspace of odd functions. By the fact that the even and odd subspaces are invariant under $H_2$, we just need to prove that the form is positive on the space $M$ of even functions which are orthogonal to $v_0$.
    For that, we look at $H_3$. The state $v_1$ is odd, so functions in $M$ are orthogonal to both $v_0$ and $v_1$. Thus, the form $\braket{\cdot,H_3\cdot}$ is positive on $M$, which implies the same for $Q$ by monotonicity.
\end{proof}

\begin{lemma}[Any distribution admits an antiderivative]\label{lemma:distributional-antiderivative}
    Let $u\in\Dscr'(\R^n)$ and $x$ be one of the coordinates of $\R^n$. There exists $U\in \Dscr'(\R^n)$ such that $\partial_x U=u$.
\end{lemma}
\begin{proof}
    Denote as $x'\in\R^{n-1}$ the remaining coordinates, and call $e_1$ the vector with coordinates $x=1$, $x'=0$. Let $\chi\in C^\infty(\R)$ be a smooth non-decreasing cutoff function such that $\chi(x)=0$ for $x\leq -1$ and $\chi(x)=1$ for $x\geq 1$. Consider $u^+:=\chi u$, $u^-:=(1-\chi)u$, where $\chi$ is considered a function on $\R^n$ depending only on the $x$ variable. Consider the Heaviside function $H:=\mathbbm 1_{[0,\infty)}$ in the variable $x$, and let $\delta_0^{n-1}\in \Dscr'(\R^{n-1})$ be the Dirac delta. For $f\in L^1_\loc(\R)$, let $f\otimes \delta_{0}^{n-1}\in\mathscr D'(\R^n)$ be defined as
    \[
    f\otimes \delta_{0}^{n-1}(\wphi):=\int_\R f(t)\wphi(te_1)dt.
    \]
    Then by a direct verification, the convolutions $U^+:=(H\otimes \delta_{0}^{n-1})*u^+$, $U^-:=((H-1)\otimes\delta_{0}^{n-1})*u^-$ are well defined, they lie in $\mathscr D'(\R^n)$, and they satisfiy $\partial_x U^\pm=u^\pm$. It follows that $U:=U^++U^-$ satisfies $\partial_x U=u$.
\end{proof}

\subsection{Some detailed proofs}\label{subsec:proofs}

\begin{proof}[Proof of Lemma \ref{lemma:change-of-variables-c-alpha_0-gamma_0}]
    As usual, we will use Notation \ref{notation:f_alpha} and write $f_\alpha(x,y)=f(x-\alpha(y),y)$ for $\alpha=\alpha(y)$ (we use the same notation when $f$ and/or $\alpha$ are independent of $y$). Recall that $v^c$ is defined as
    \als{
    v^c=&\frac{v^+e^{V^+-c}+v^-e^{V^-+c}}{e^{V^+-c}+e^{V^-+c}}\\
    =&(\etapcnu)v^++(\etamcnu)v^-,
    }
    where we set
    \begin{equation}\label{z-018}
        \nu(x,y):=\aha(V^+(x,y)-V^-(x,y))-c.
    \end{equation}
    The function $c\mapsto v^c|_{y=0}$ is a curve in $L^2(\R_x;\cosh^2(x)dx)+G_0|_{y=0}$ by Lemma \ref{prop:vcAreSolutions} and can be easily verified to be smooth, so by Lemma \ref{lemma:expDecomposition} parts \ref{item:expDec-003} and \ref{item:expDec-004}, the map $c\mapsto \alpha_0$ is well-defined and smooth. Next, we differentiate \eqref{eq:definition_of_alpha_0} in the variable $c$ to get
    \[
    \de_c\alpha_0(c)=\frac{\int_\R \sechscnu(\cdot,0)\cdot(v^+-v^-)|_{y=0}\,\, dx}{\int_\R\sech^2_{\alpha_0}\cdot\,(v^+-v^-)|_{y=0}\,\,dx}.
    \]
    
    \begin{claim}[1] If $x_0\in\R$ is such that $\nu(x_0,y_0)=0$, then
    \[
    \|\sechscnu|_{y=y_0}-\sech^2_{x_0}\|_{L^1\cap L^\infty(\R)}+\|\etapmcnu|_{y=y_0}-\eta^\pm_{x_0}\|_{L^1\cap L^\infty(\R)}\lesssim \|u\|_{\Hcrit}.
    \]
    \end{claim}
    \begin{claimproof}[of Claim 1]
    We fix $y=0$ for simplicity and focus on the $\sech^2$ case, since the other is analogous. It holds $\partial_x\nu(x,y)=1+\frac{\wv^+-\wv^-}{2}$, so by Corollary \ref{cor:estimates-for-wtv-part-2} part (b) we have
    \[ |\nu(x_2,y)-\nu(x_1,y)-(x_2-x_1)|\lesssim \|u\|_{\dot H^{-\aha,0}(\RR)}|x_2-x_1|^\aha. \]
    In particular, for some $C>0$,
    \begin{align}
    |\nu(x,0)-(x-x_0)|&\leq C\|u\|_{\Hcrit}|x-x_0|^{1/2},\label{eq:z-013}\\
    |\nu(x,0)|&\geq |x-x_0|-C\|u\|_{\Hcrit}|x-x_0|^{1/2}.\label{eq:z-014}
    \end{align}
    Moreover, $|\frac{d}{dx}\sech^2(x)|=|-2\tanh(x)\sech^2(x)|\leq 8e^{-2|x|}$. So, calling $\sigma:=C\|u\|_{\Hcrit}$,
    \als{
    |\sech^2(\nu(x,0))-\sech^2(x-x_0)|&= \left|\int_0^1 (\sech^2)_x((x-x_0)+s(\nu(x,0)-(x-x_0)))\,ds\right|\\
    &\QQQ\cdot|\nu(x,0)-(x-x_0)|\\
    \text{\eqref{eq:z-013}, \eqref{eq:z-014}}\longrightarrow &\lesssim \min\{1,e^{-2(|x-x_0|-\sigma|x-x_0|^{1/2})}\}\sigma|x-x_0|^{1/2}\\
    &\lesssim e^{-2(|x-x_0|-|x-x_0|^{1/2})_+}|x-x_0|^{1/2}\|u\|_{\Hcrit}\\
    &=:\tau(x-x_0)\|u\|_{\Hcrit},
    }
    where the last inequality holds for small $u$. The claim is proved since $\tau\in L^1\cap L^\infty(\R)$.
    \end{claimproof}
    
    Let $x_0$ be any point such that $\nu(x_0,0)=0$. By Corollary \ref{cor:estimates-for-wtv-part-2}, it holds
    \begin{equation}\label{eq:z-021}
        \|(v^+-v^-)-2\|_{C_{0,y}L^2_x}\lesssim \|u\|_{\Hcrit}.
    \end{equation}
    This fact and Claim 1 imply that the numerator and denominator in the expression of $\de_c\alpha_0$ are uniformly bounded from above and away from zero by the smallness of $(v^+-v^-)|_{y=0}-2$ in $L^2(\R_x)$ (as already noted in the proof of Proposition \ref{lemma:expDecomposition}), and that
    \[
    \sup_{c\in\R}|\de_c\alpha_0(c)-1|\lesssim \|u\|_{\Hcrit}.
    \]
    In particular, $c\mapsto\alpha_0$ is a $C^1$-diffeomorphism of $\mathbb R$.
    
    Concerning $\gamma_0$, the map
    \begin{equation}\label{loc:map-to-use-IFT-c-gamma_0}
        (c,\gamma_0)\mapsto F(c,\gamma_0)=\int_{\RR}\rho_{\gamma_0} \frac{v^+e^{V^+-c}+v^-e^{V^-+c}}{e^{V^+-c}+e^{V^-+c}} dx\,dy
    \end{equation}
    is well-defined and smooth with $\gamma_0-$derivative
    \[
    \partial_{\gamma_0}F(c,\gamma_0)=\int_{\RR}\rho_{\gamma_0} \partial_x v^c dx\,dy.
    \]
    From Proposition \ref{prop:vcAreSolutions} and the estimates of Corollary \ref{cor:uniquenessOfMiura}, we know that $v^c=\tanh_{\alpha}+w$, where $\|w\|_{L^3(\RR)}+\|\alpha_y\|_{L^2(\R_y)}\lesssim\|u\|_{\Hcrit}$ and $\alpha(0)=\alpha_0$. Since $w\in L^3(\RR)$ and $|\alpha(y)-\alpha_0|\lesssim \|u\|_{\Hcrit}|y|^{1/2}$, for fixed $c$ we have $F(c,\gamma_0)\to\pm 1$ as $\gamma_0\to\pm\infty$, so that at least one solution of \eqref{eq:definition-of-gamma_0} exists. Moreover, if $\gamma_0$ satisfies \eqref{eq:definition-of-gamma_0}, then
    \als{ \min\{1,|\gamma_0-\alpha_0|\}&\lesssim \left|\int\rho_{\gamma_0}
    \tanh_{\alpha_0} dx\,dy\right|\\
    &\leq\left|\int\rho_{\gamma_0}
    \tanh_\alpha dx\,dy\right|+\left|\int\rho_{\gamma_0}\cdot
    (\tanh_\alpha-\tanh_{\alpha_0}) dx\,dy\right|\\
    \eqref{eq:definition-of-gamma_0}\longrightarrow\,&\leq\left|\int\rho_{\gamma_0}
    w\, dx\,dy\right|+\|\alpha_y\|_{L^2(\R_y)}\int\rho_{\gamma_0}|y|^\aha\,dx\,dy\\
    &\lesssim \|u\|_{\Hcrit},
    }
    which in turn implies, since $u$ is small, that if the map $c\mapsto \gamma_0$ exists,
    \[ |\gamma_0(c)-\alpha_0(c)|\lesssim \|u\|_\Hcrit. \]
    Now note that, since $\rho$ is radially decreasing and with unitary integral,
    \als{
    \int_{\RR}\rho(x,y) \sech^2(x)\,\dxdy&\geq \aha\int_{-1}^1\sech^2(x)\,dx\\
    &=\tanh(1)\\
    &>3/4.
    }
    Writing $v^c_x=\sech^2_{\alpha}+w_x=\sech^2_{\alpha_0}+(\sech^2_{\alpha}-\sech^2_{\alpha_0})+w_x$, with a similar computation as above, there exists a universal $\delta>0$ such that, for $u$ small enough,
    \begin{equation}\label{eq:z-015}
        |F(c,\gamma')|<\delta\quad\implies\quad \de_{\gamma_0}F(c,\gamma')>3/4.
    \end{equation}
    In particular, $\de_{\gamma_0}F(c,\gamma_0)>3/4$ if $\gamma_0$ satisfies \eqref{eq:definition-of-gamma_0}. A smooth map $c\mapsto \gamma_0$ exists then locally by the implicit function theorem, and it extends to a global, unique map due to the fact that whenever \eqref{eq:definition-of-gamma_0} holds, $\de_{\gamma_0}F(c,\gamma_0)$ is strictly positive, so that two distinct zeroes of the function $\gamma_0\mapsto F(c,\gamma_0)$ cannot exist. An analogous computation to that of the case $\alpha_0$ shows that
    \als{
    \de_c\gamma_0(c)&=\frac{\aha\int_{\RR}\rho_{\gamma_0}\cdot (v^+-v^-)\sech^2\!\circ\hspace{2pt}\nu\, dx\,dy}{\int_{\RR}\rho_{\gamma_0}v^c_x\,dx\,dy}\\
    &=\frac{\int_{\RR}\rho_{\gamma_0}\cdot(v^+-v^-)\sech^2\!\circ\hspace{2pt}\nu\,dx\,dy}{\int_{\RR}\rho_{\gamma_0}\cdot \left( \sechscnu\cdot\fr14(v^+-v^-)^2+\etapcnu\cdot v^+_x+\etamcnu\cdot v^-_x\right)\,dx\,dy}.
    }
    The terms $\sechscnu$, $\etapmcnu$ are treated in the same way as in the case of $\alpha_0$ using Claim 1, while the remaining terms can be controlled using again the bound mentioned above on $(v^+-v^-)-2$ in $C_0(\R_y,L^2(\R_x))$. In the end, we get
    \begin{equation}\label{eq:z-016}
        \sup_{c\in\R}|\de_c \gamma_0(c)-1|\lesssim \|u\|_{\Hcrit}.
    \end{equation}
    As said above, the map $c\mapsto \gamma_0$ is a smooth change of variables.

    We have proved already estimate \eqref{eq:gamma_0_and_alpha_0_are_uniformly_close}. Estimate \eqref{eq:c_is_close_to_alpha_0_and_gamma_0_for_small_c} follows by combining the bounds on $|\de_c \alpha_0-1|$ and $|\de_c\gamma_0-1|$ we have already obtained, with the bounds
    \[
    |\gamma_0(0)|\lesssim \|u\|_{\Hcrit},\qquad |\alpha_0(0)|\lesssim \|u\|_{\Hcrit}.
    \]
    These two bounds are in turn equivalent by \eqref{eq:gamma_0_and_alpha_0_are_uniformly_close}, so we focus on proving that $|\alpha_0(0)|\lesssim \|u\|_{\Hcrit}$.
    
    \begin{claim}[2]
    If $x_0\in\R$ is such that $\nu(x_0,0)=0$, then
    \[
    |x_0-\alpha_0(c)|\lesssim \|u\|_{\Hcrit}.
    \]
    \end{claim}
    \begin{claimproof}[of Claim 2]
    Using \eqref{eq:definition_of_alpha_0} and the identity $\eta^++\eta^-=1$, it holds
    \als{
        \left|\int_{\R} (G_{\alpha_0}-G_{x_0})|_{y=0}\,dx\right|&=
        \left|\int_\R (v^c-G_{x_0})|_{y=0} dx\right|\\
        &=\left|\int_{\R}((v^+-v^-)(\etapcnu-\eta_{x_0}))|_{y=0}\,dx\right|\\
        &\lesssim \|u\|_{\Hcrit},
    }
    where the last inequality follows from Claim 1 and the bounds on $(v^+-v^-)-2$ as before. The claim follows by Lemma \ref{lemma:expDecomposition}, part \ref{item:expDec-001}.
    \end{claimproof}
    
    Fix now $c=0$. By Claim 2, it is enough to show that $|x_0|\lesssim \|u\|_{\Hcrit}$ for any $x_0$ as in the statement of the Claim (such a $x_0$ always exists since $x\mapsto \nu(x,0)-x$ is globally H\"older continuous by \eqref{eq:z-021}). From Definition \ref{def:VcAndvc}, and since $\rho$ is radially symmetric, it holds
    \[
    \int_{\RR} \nu(x,y)\rho(x,y)\,\dxdy=0.
    \]
    In particular,
    \als{
    |\nu(0,0)|&=\left|\int_{\RR} \nu(0,0)\rho(x,y)\,\dxdy\right|\\
    &=\left|\int_{\RR} (\nu(0,0)-\nu(x,y))\rho(x,y)\,\dxdy\right|\\
    \text{Lemma \ref{lemma:holder_regularity_for_difference_of_primitive_solutions}}\longrightarrow&\lesssim \|u\|_{\Hcrit} \int_{\RR} \rho(x,y) |(x,y)|^{1/4}\,\dxdy\\
    &\lesssim \|u\|_{\Hcrit}.
    }
    Now, by \eqref{eq:z-013},
    \als{
    |x_0|&\leq C\|u\|_{\Hcrit}|x_0|^{1/2}+|\nu(0,0)|\\
    &\leq C\|u\|_{\Hcrit}
    \implies \frac{|x_0|}{1+|x_0|^{1/2}}\lesssim \|u\|_{\Hcrit},
    }
    which proves the bound for small $u$.
    
    For the bi-Lipschitz bound, we proceed as follows. First, by estimate \eqref{eq:z-016}, it is enough to show the Lipschitz continuity of the forward map $(u,c)\mapsto (u,\gamma_0)$. By the same estimate and the triangle inequality, it is enough to show
    \begin{equation}\label{eq:z-017}
        |\gamma_{0,1}-\gamma_{0,2}|\lesssim (1+|c|)\|u_1-u_2\|_{\Hcrit}
    \end{equation}
    for small $u_1,u_2$, the corresponding $\gamma_{0,1},\gamma_{0,2}$ satisfying \eqref{eq:definition-of-gamma_0} with $v^c_{j}=\eleV^{(-1,1)}(u_j,(c,-c))$, and for $c\in \R$ which is shared by both solutions. By the bounds in Lemma \ref{lemma:holder_regularity_for_difference_of_primitive_solutions} and the analiticity of the map $u\mapsto V^+-V^-$, it holds
    \begin{equation}\label{eq:z-019}
        \|\nu_1-\nu_2\|_{C^{0,1/4}_\unif(\RR)}\lesssim \|u_1-u_2\|_{\Hcrit},
    \end{equation}
    with $\nu_j$ corresponding to $u_j$ and defined as in \eqref{z-018}. By the normalization condition \eqref{eq:normalization-condition}, it holds $\int \rho \nu_j\,\dxdy=-c$, $j=1,2$, in particular,
    \begin{equation}\label{eq:z-020}
        \int_{\RR} \rho\cdot (\nu_1-\nu_2)\,\dxdy=0.
    \end{equation}
    Now we consider the difference $v^c_1-v^c_2$, and write it as
    \als{
    v^c_1-v^c_2&=\etapcnu_1\cdot v^+_1+\etamcnu_1\cdot v^-_1 - \etapcnu_2\cdot v^+_2-\etamcnu_2\cdot v^-_2\\
    \text{Corollary \ref{cor:estimates-for-wtv-part-2} (a)}\longrightarrow &=\tanhcnu_1-\tanhcnu_2 + (\eta^+\circ \nu_1-\eta^+\circ \nu_2)\wv_1^+\\
    &\QQQ+(\eta^-\circ \nu_1-\eta^-\circ \nu_2)\wv_1^-+\OO_{L^3}(\|u_1-u_2\|_{\Hcrit}),
    }
    where $\wv^\pm_j=v^\pm_j\mp 1$. The remainder goes to zero linearly with $\|u_1-u_2\|_{\Hcrit}$ by Corollary \ref{cor:estimates-for-wtv-part-2}, since the map $u\mapsto \wv^\pm$ is analytic with values in $L^3(\RR)$. The rewriting above implies
    \als{
    \left|\int_{\RR}\rho_{\gamma_{0,1}}\cdot(v^c_1-v^c_2)\,\dxdy \right | &\leq \int_{\RR}\rho_{\gamma_{0,1}}\cdot|\tanhcnu_1-\tanhcnu_2|(1+|\wv_1^+|+|\wv_1^-|)\,\dxdy\\
    &\QQQ +\OO(\|u_1-u_2\|_{\Hcrit})\\
    &\lesssim \int_{\RR}\rho_{\gamma_{0,1}}\cdot|\nu_1-\nu_2| (1+|\wv_1^+|+|\wv_1^-|)\,\dxdy \\
    &\QQQ+ \|u_1-u_2\|_{\Hcrit}\\
    \eqref{eq:z-020},\,\eqref{eq:z-031}\longrightarrow &\lesssim |\gamma_0|\|\nu_1-\nu_2\|_{C^{0,1/4}_\unif(\RR)}+\|u_1-u_2\|_{\Hcrit}\\
    \eqref{eq:z-019}, \,\eqref{eq:c_is_close_to_alpha_0_and_gamma_0_for_small_c}\,\longrightarrow &\leq (1+|c|) \|u_1-u_2\|_{\Hcrit}.
    }
    where we used the estimate from Corollary \ref{cor:estimates-for-wtv-part-2}
    \begin{equation}\label{eq:z-031}
    \|\wv^\pm_j\|_{L^3(\RR)}\lesssim \|u_j\|_{\Hcrit}\ll 1.
    \end{equation}
    Since it holds \eqref{eq:definition-of-gamma_0} with $\gamma_{0,1}$ and $v^c_1$, we simply have
    \[
    \int_{\RR}\rho_{\gamma_{0,1}}v^c_2\,\dxdy\lesssim (1+|c|)\|u_1-u_2\|_{\Hcrit}.
    \]
    By the property \eqref{eq:z-015} applied to $u_2,v^c_2$, estimate \eqref{eq:z-017} is proved when the right hand side is less than a universal constant, which we can assume by the smallness of $u_1,u_2$.
\end{proof}

\begin{proof}[Proof of Lemma \ref{lemma:modifiedLocalSmoothing}]
    Assume $u\in \partial_xH^\infty(\Real^2)$ by density. By considering $\alpha_\eps=\alpha*\rho_\eps$ a regularization of $\alpha$, with regularization parameter $\eps= K c^\aha \|\alpha_y\|^2$ for a universal $K$ large enough, we can assume $\alpha\in H^\infty(\R)$ and that 
    \[
    \delta:=c^{1/4}\|\alpha_y\|_{L^\infty}
    \]
    is small. In fact, by the properties of the regularization, we have $c^{1/4}\|\alpha_{\eps,y}\|_{L^\infty}\leq c^{1/4}\eps^{-\aha} \|\alpha_y\|_{L^2}=K^{-1}$, and
    \[
    \braket{c^{\aha}(x-\alpha(y))}\leq \braket{c^{\aha}\|\alpha-\alpha_\eps\|_{L^\infty}} \braket{c^{\aha}(x-\alpha_\eps(y))},
    \]
    with $c^\aha\|\alpha-\alpha_\eps\|_{L^\infty}\leq c^\aha \eps^\aha \|\alpha_y\|_{L^2}=Kc^{3/4} \|\alpha_y\|_{L^2}^2=KL$.
    
    Call $u(t):=e^{tS}u_0$, and let $a(x)=(\pi+\arctan(c^{\aha}x))$. From
    $$ u_t-cu_x+u_{xxx}+3\partial_x^{-1}u_{yy}=0, $$
    one gets
    \begin{align*}
        \aha\frac{d}{dt}\int_{\RR} a_\alpha |u|^2\dxdy&=c\int a_\alpha uu_x -\int a_\alpha uu_{xxx}-3\int a_\alpha u\partial_x^{-1}u_{yy}\\
        &=-\frac{c}{2}\int a_{x,\alpha}|u|^2+\int a_{x,\alpha}uu_{xx}+\int a_\alpha u_xu_{xx}\\
        &\QQQ-3\int\alpha_y a_{x,\alpha}u\partial_x^{-1}u_y+3\int a_\alpha u_y\partial_x^{-1}u_y\\
        &=-\frac{3}{2}\int a_{x,\alpha}|u_x|^2+\frac{1}{2}\int a_{xxx,\alpha}|u|^2-\frac{c}{2}\int a_{x,\alpha}|u|^2\\
        &\QQQ-3\int\alpha_y a_{x,\alpha}u\partial_x^{-1}u_y-\fr 32\int a_{x,\alpha}|\partial_x^{-1}u_y|^2.
    \end{align*}
    Note that $|a_{xxx}|\leq \fr c2\,a_x$. Integrating in the time variable, we obtain the estimate
    \begin{align*}
        c\|\sqrt{a_{x,\alpha}}u\|_{L^2_TL^2}^2+\|\sqrt{a_{x,\alpha}}\partial_xu\|_{L^2_TL^2}^2&\\
        +\|\sqrt{a_{x,\alpha}}\partial_x^{-1}u_y\|_{L^2_TL^2}^2&\lesssim \|\sqrt{a_\alpha}u_0\|^2_{L^2}+\|a_x\|^\aha_{L^\infty}\|\alpha_y\|_{L^\infty_y}\\
        &\QQQ\Qq\times\|\sqrt{a_{x,\alpha}}u\|_{L^2_T L^2}\|\sqrt{a_{x,\alpha}}\partial_x^{-1}u_y\|_{L^2_TL^2}\\
        &\lesssim \|u_0\|_{L^2}^2+\delta\|\sqrt{a_{x,\alpha}}u\|_{L^2_T L^2}\|\sqrt{a_{x,\alpha}}\partial_x^{-1}u_y\|_{L^2_TL^2}.
    \end{align*}
    For $\delta\geq 0$ small enough, the above implies
    \begin{equation*}
    \|\sqrt{a_{x,\alpha}}\partial_xu\|_{L^2_TL^2}+\|\sqrt{a_{x,\alpha}}\partial_x^{-1}u_y\|_{L^2_TL^2}\lesssim \|u_0\|_{L^2}.
    \end{equation*}
    Substituting $a$ with its definition, we get the desired inequality.
\end{proof}

\begin{proof}[Proof of Lemma \ref{lemma:Lp+weight/decay+zero_mean_implies_Hardy_H1}]
    The statement is monotonically weaker as $p$ grows by H\"older's inequality, so we assume $p<\infty$. By rescaling the measure $\mu$, we can assume $\mu(\Bcal_1(0))=1$, in particular $\mu(B_{j})\leq 2^{jD}$ for $j\in\N$, where $B_{j}:=\Bcal_{2^j}(x_0)$. For the sake of exposition, we prove the Lemma under the assumption
    \begin{equation}\label{eq:z-009}
        \mu(B_j)\sim 2^{jD}
    \end{equation}
    and we mention at the end how to modify the proof in the general case.
    \vspace{1ex}
    
    \textbf{Step 1.} Let $\chi_j=\mathbbm 1_{B_j}$, and $A_j:=B_{j}\setminus B_{j-1}$. For $j\geq 0$, define
    \[
    f_j:=f\chi_j-\frac{\int_X f\chi_j\,d\mu}{\int_X\chi_{j}\,d\mu}\chi_j,
    \]
    and consider the decomposition
    \[
    f=f_0+\sum_{j\geq 1}f_j-f_{j-1}.
    \]
    All the functions in the decomposition are in $L^p(X)$, have mean zero and have support in a ball, so they are multiples of $p$-atoms, as in Definition \ref{def:Hardy_space_H1}. By what said in Remark \ref{rk:q-atoms,H1-BMO_duality} and the definition of $p$-atoms, to show the claim and the above bound, it is enough to show that
    \[
    \|f_0\|_{L^p}\lesssim \|wf\|_{L^p},\qquad\sum_{j\geq 1}2^{jD/p'}\|f_j-f_{j-1}\|_{L^p}\lesssim \|wf\|_{L^p},
    \]
    since the support of $f_j-f_{j-1}$ is contained in the ball $B_j$, whose measure is comparable to $2^{jD}$.
    \vspace{1ex}
    
    \textbf{Step 2.} The first bound is immediate. For the second bound, we have
    \[
    \|f_j-f_{j-1}\|_{L^p}\leq \|f\|_{L^p(A_j)}+\mu(B_j)^{-1/p'}\left|\int_{B_j}f\,d\mu\right|+\mu(B_{j-1})^{-1/p'}\left|\int_{B_{j-1}}f\,d\mu\right|.
    \]
    In addition to \eqref{eq:z-009}, we have the bounds
    \begin{align*}
    \left|\int_{B_j} f\,d\mu\right|&=\left|\int_{B_j^c} f\,d\mu\right|\leq \|1/w\|_{L^{p'}(B_j^c)}\|wf\|_{L^p}\\
    &\lesssim 2^{-j\eps}\|wf\|_{L^p},\\[1em]
    \|f\|_{L^p(A_j)}&\leq \|1/w\|_{L^{\infty}(B_j^c)}\|wf\|_{L^p}\\
    &\lesssim 2^{-j(D/p'+\eps)} \|wf\|_{L^p},
    \end{align*}
    which combined yield
    \[
    2^{jD/p'}\|f_j-f_{j-1}\|_{L^p}\lesssim 2^{-j\eps} \|wf\|_{L^p},
    \]
    so we obtain the claim and prove the proposition under the additional assumption \eqref{eq:z-009}.
    \vspace{1ex}
    
    \textbf{Step 3.} For a general doubling metric measure space we can modify the proof as follows. If $\mu(X)<\infty$, then it is easy to show by contradiction that $X$ is $d$-bounded, so the statement is immediate since any $L^p$ function with zero mean is also a multiple of a $p$-atom. If $\mu(X)=\infty$, we set $B_j=\Bcal_{R_j}(x_0)$, with $(R_j)_j$ being a sequence of radii $R_j\to\infty$ such that $R_0=1$, $R_j\geq 2R_{j-1}$, $2\mu(B_{j-1})\leq \mu(B_j)\leq 4\mu(B_{j-1})$, and repeat the same argument in the previous two steps without other changes.   
\end{proof}

\begin{proof}[Proof of Lemma \ref{lemma:heat_kernel_and_Hardy_space_H1}]
    The way of proving this estimate is classical.  We recall the definition of parabolic norm and metric given in Definition \ref{def:parabolic_BMO}. First, by the change of coordinates
    \[
    (x,y)\mapsto (x+2\lambda y,y),
    \]
    one can assume $\lambda=0$. For a convolution operator $K$ with kernel $K=K(z)$, $z=(x,y)$, consider the property
    \begin{equation}\label{eq:z-010}
        |K(z)|\lesssim |z|_{\rm p}^{-\alpha},\qquad |\partial_xK(z)|\lesssim |z|_{\rm p}^{-\alpha-1},\qquad |\partial_yK(z)|\lesssim |z|_{\rm p}^{-\alpha-2}.
    \end{equation}
    The two kernels $|\de_x|^{-1/2}\Gamma$, $|\de_x|^{1/2}\Gamma$ satisfy the property \eqref{eq:z-010} with $\alpha=1/2$, $\alpha=3/2$, respectively: this is easy to verify since the two kernels and the parabolic norm are homogeneous with respect to the parabolic rescaling $(x,y)\mapsto (sx,s^2y)$. It is thus enough to show that for a convolution operator $K$, it holds for $\alpha\in(0,3)$:
    \[
    \left[\text{$K$ satisfies the property }\eqref{eq:z-010}\right]\quad\implies\quad \left[\|Kf\|_{L^p(\RR)}\lesssim \|f\|_{\PHu(\RR)},\,\,\frac 1p=\fr\alpha 3\right].
    \]
    By the definition of $\PHux\lambda(\RR)$ and by linearity one can assume that $f$ is an $\infty$-atom. By scaling, we can assume
    \[\|f\|_{L^\infty}\leq 1,\qquad \supp f\subset \Bcal_1((0,0))=:B,\qquad \int_{\RR}f\,dx\,dy=0.
    \]
    Let $2B:=\Bcal_2((0,0))$. By the first estimate in \eqref{eq:z-010}, the kernel $K\in L^{p,\infty}(\RR)\subset L^1+L^\infty$, thus
    \[
    \|\Gamma f\|_{L^\infty(2B)}\lesssim 1.
    \]
    By combining the estimates \eqref{eq:z-010} as follows, for $|z|_\p\geq 2|w|_\p$, $w=(x',y')$ one obtains
    \begin{align*}
    |K(z-w)-K(z)|&\leq\left|x'\int_0^1 \partial_x K(z-sw)ds\right|+\left|y'\int_0^1 \partial_y K(z-sw)ds\right|\\
    &\lesssim\frac{|w|_{\rm p}}{|z|_{\rm p}^{\alpha+1}}+\frac{|w|^{2}_{\rm p}}{|z|_{\rm p}^{\alpha+2}}\\
    &\lesssim \frac{|w|_{\rm p}}{|z|_{\rm p}^{\alpha+1}}.
    \end{align*}
    With that, one can use the zero mean of $f$ and the bounds we have to estimate
    \als{
        |K f(z)|&=\left|\int_B [K(z-w)-K(z)]f(w)\,dw\right|\\
        &\lesssim \frac{1}{|z|_{\rm p}^{\alpha+1}}
        }
    for $z\in 2B^c$. It holds $1/|\cdot|_{\rm p}^{\alpha+1}\in L^{q,\infty}(\RR)$, $q=3/(\alpha+1)<p$, and clearly $\mathbbm 1_{2B^c}(z)\,1/{|z|_{\rm p}^{\alpha+1}}\leq 1/4$, so the estimate is proved by combining the bound on $2B$ with the one on $2B^c$.
\end{proof}

\bibliographystyle{hplain_custom.bst}
\bibliography{biblio.bib}

\end{document}